\documentclass[envcountsect,envcountsame]{svjour3}

\usepackage{amsmath,amssymb,amsfonts,amsbsy}
\usepackage{mathptmx}
\usepackage{latexsym}
\usepackage{bm}

\numberwithin{equation}{section}

\newtheorem{cor}[theorem]{Corollary}
\newtheorem{exam}[theorem]{Example}

\newtheorem{prop}[theorem]{Proposition}

\newcommand {\kB}{{\mathcal B}}

\newcommand {\kT}{{\mathcal T}}

\newcommand {\kN}{{\mathcal N}}

\newcommand {\kV}{{\mathcal V}}
\newcommand {\kO}{{\mathcal O}}

\newcommand\msk{\medskip}

\newcommand\bsk{\bigskip}

\newcommand{\SO}{{\rm SO}}

\newcommand{\tr}{{\rm tr}}

\newcommand{\Span}{{\rm span}}
\newcommand{\const}{{\rm const}}
\newcommand\diag{{\rm diag}}

\newcommand {\uu}{{\mathbf u}}
\newcommand {\vv}{{\mathbf v}}

\newcommand {\D}{{\mathrm D}}
\newcommand {\rP}{{\mathrm P}}
\renewcommand{\d}{{\mathrm d}}
\renewcommand{\i}{{\mathrm i}}
\newcommand {\e}{{\mathrm e}}

\newcommand  {\bR}{{\mathbb R}}
\newcommand  {\bS}{{\mathbb S}}

\newcommand  {\Q}{{ Q}}

\newcommand\LL{\bm{L}}

\newcommand\F{{\mathcal F}}

\newcommand\B{{\mathcal B}}
\newcommand\C{{\mathcal C}}

\newcommand\M{{\mathcal M}}
\newcommand\N{{\mathcal N}}
\newcommand\U{{\mathcal U}}
\newcommand\V{{\mathcal V}}

\newcommand\ee{{\mathbf e}}
\newcommand\z{{\mathbf z}}

\newcommand\al{\alpha}

\newcommand\gam{\gamma}
\newcommand\eps{\varepsilon}
\newcommand\lam{\lambda}
\newcommand\Lam{\Lambda}

\newcommand\vf{\varphi}
\newcommand\tvf{\tilde\varphi}
\newcommand\dt{\delta t}

\newcommand\til{\tilde}

\newcommand\tq{\widetilde \Q}
\newcommand\wtr{\widetilde R}

\newcommand\tqa{\widetilde \Q(\alpha)}

\newcommand {\id}{\operatorname{id}}

\newcommand{\details}[1]{}

\begin{document}

\title{Existence and stability of kayaking orbits for nematic liquid crystals in simple shear flow}

\titlerunning{Kayaking orbits for nematic liquid crystals in shear flow}

\author{David Chillingworth \and M. Gregory Forest \\ Reiner Lauterbach \and Claudia Wulff}

\institute{
D. R. J. Chillingworth \at Mathematical Sciences  \\
 University of Southampton \\
   Southampton  SO17 1BJ, UK  \\
   \email{drjc@soton.ac.uk}
   \and
M. G. Forest  \at  Departments of Mathematics \& Applied Physical Sciences \& Biomedical Engineering  \\
University of North Carolina at Chapel Hill \\
Chapel Hill, NC 27599-3250, USA \\
\email{forest@unc.edu}
\and
R. Lauterbach  \at  Fachbereich Mathematik  \\
Universit{\"a}t Hamburg \\
  20146 Hamburg, Germany  \\
  \email{lauterbach@math.uni-hamburg.de}
  \and
C. Wulff\quad\maltese\quad{\em Deceased 12 June 2021} \at Department of Mathematics  \\
 University of Surrey \\
 Guildford GU2 7XH, UK \\
    and \\
       Free University Berlin \\ Department of Mathematics \\ Arnimallee 2-6 \\ 14195 Berlin, Germany
}
\date{Received: date / Accepted: date}

\maketitle

%\date{\today}

\bsk

\begin{abstract}
  We use geometric methods of equivariant dynamical systems to address a long-standing open problem in the theory of nematic liquid crystals, namely a proof of the existence and asymptotic stability of kayaking periodic orbits in response to steady shear flow. These are orbits for which the principal axis of orientation of the molecular field (the director) rotates out of the plane of shear and  around the vorticity axis.
With a small parameter attached to the symmetric part of the velocity gradient, the problem can be viewed as a symmetry-breaking bifurcation from an orbit of the rotation group~$\SO(3)$ that contains both logrolling (equilibrium) and tumbling (periodic rotation of the director within the plane of shear) regimes as well as a continuum of neutrally stable kayaking orbits. The results turn out to require expansion to second order in the perturbation parameter.
  \msk
  \msk
\end{abstract}
\tableofcontents
\begin{center}
  \textbf{Dedication}
  \end{center}
  The first three named authors dedicate this paper to the fond memory of our late colleague Claudia Wulff, who passed away between the first provisional acceptance of this paper and its eventual publication. Throughout our work her cheerfulness, enthusiasm, clear geometric insight and scrupulous attention to detail have continually inspired us and sustained this project.  We are grateful to have had the good fortune to share this collaboration over many years with her.

\section{Introduction}
Nematic liquid crystals, regarded as fluids in which the high aspect ratio, rigid, rod molecules require descriptive variables for orientation as well as position, are observed to exhibit a wide range of prolonged unsteady dynamical responses to steady shear flow.  The mathematical study of these phenomena in principle involves the Navier-Stokes equations for fluid flow coupled with equations representing molecular alignment and nonlocal interactions between rod molecules, typically leading to PDE systems currently intractable to rigorous analysis on a global scale and resolved only through local analysis and/or numerical simulation.  It becomes appropriate therefore to deal with simpler models as templates for capturing some of the dynamical regimes of interest and their responses to physical parameters. Stability and bifurcation behaviours that are robust for finite-dimensional dynamical systems, and that numerically reflect the same orbits of interest (specifically, kayaking orbits) in infinite-dimensional systems, provide a framework for extension of rigorous results to the infinite-dimensional systems. 
\msk

Much of the work on dynamics of liquid crystals (and more generally, rigid large aspect ratio polymers) in fluid flow rests on models proposed by Hess~\cite{Hess76} and Doi~\cite{Doi81} that consider the evolution of the probability density on the 2-sphere (more accurately, projective space $\bR\rP^2$) representing unoriented directions of molecular alignment with the molecules regarded as rigid rods. Extensive theoretical and numerical investigations (\,\cite{BurFul90}, \cite{FaraoniEtAl1999}, \cite{FarhoudiRey1993}, \cite{LarOtt91}, \cite{MSV00}, \cite{MaffCres95}, \cite{OlmLu99}, \cite{RienHess}, \cite{RienackerEtAl2002}, \cite{RienackerEtAl2002b} to cite only a few) of these and related nematic director or orientation tensor models in 2D or 3D  reveal a wide range of periodic molecular dynamical regimes with evocative names~\cite{LarOtt91}  {\em logrolling}, {\em tumbling}, {\em wagging} and {\em kayaking} according to the behaviour (steady versus periodic) of the principal axis of molecular orientation (the nematic director) relative to the shear (flow velocity and velocity gradient) plane and vorticity axis (normal to the shear plane). Tumbling orbits, for which the principal axis of molecular orientation rotates periodically in the shear plane, are seen to be stable at low shear rates, but become unstable to out-of-plane perturbations and give way to kayaking orbits, for which the principal molecular axis is transverse to the shear plane, and rotates around the vorticity axis, reminiscent of the motion of the paddles propelling a kayak along the shear flow of a calm stream.
The limiting case is logrolling, a stationary state where the principal axis of the rod ensemble collapses onto the vorticity axis, while wagging corresponds to oscillations (but not complete rotations) of the molecular orientation in the shear plane about some mean angle, although wagging regimes do not appear in our analysis.  We note very recent experimental results \cite{FoxEtAl2020} coupled with the high-resolution numerical results of the Doi-Hess kinetic theory \cite{ForestEtAl2004b} that provide overwhelming evidence that the kayaking orbit is responsible for the anomalous shear-thickening response of a high aspect ratio, rodlike, liquid crystal polymer with the acronym PBDT.
The papers~\cite{ForestWang2003}, \cite{FoxEtAl2020} give extensive lists of literature references.
\msk

In the particular case of a steady shear flow and spatially homogeneous liquid crystal in a region in~$\bR^3$, the PDEs describing the evolution of orientational order can be simplified to an autonomous ODE in the setting of the  widely-used $\Q$-tensor model~\cite{deGP}, \cite{MottNewt}, \cite{SV} for nematic liquid crystals. The assumption of spatial homogeneity of course rules out many important applications, to display technology for example,   but nevertheless gives a worthwhile approximation in local domains of homogeneity (monodomains) away from boundaries and defects.
In this setting the propensity of a molecule to align in any given direction in $\bR^3$ is represented by an {\it order tensor}~$\Q$ belonging to the 5-dimensional space~$V$ of traceless symmetric $3\times3$ matrices,
\begin{equation}  \label{e:Vdef} 
V:= \{A\in\bR^{3\times3}:A^{\mathtt{t}}=A,\ \tr(A) = 0 \}
\end{equation}
where $\{\,\}^{\mathtt{t}}$ and $\,\tr\,$ denote transpose and trace respectively.  The tensor $\Q$ is interpreted as the normalised second moment of a more general probability distribution on~$\bR\rP^2$. All such $\Q$-tensor models can be associated with a moment-closure approximation of the Smoluchowski equation for the full orientational distribution function \cite{ForestWang2003}. The derivation of the equation yields technical problems concerning the approximation of higher-order moments, a topic of some discussion in the literature: see~\cite{Feng+98},~\cite{ForestWang2003}, ~\cite{KAC},~\cite{LeeEtAl06} for example.
In this context the dimensionless equation for the evolution of the orientational order takes the general form\footnote{In this paper we do not use bold face symbols for elements of~$V$, but reserve bold face for the higher order tensor~$\LL(\Q)$ and for vectors in~$\bR^3$. This matches the convention adopted by MacMillan in~\cite{MacM92a,MacM92b}. Lower case Greek symbols denote scalars.}   
\begin{equation}   \label{e:sys1}
\frac{\d \Q}{\d t} = F(\Q,\beta):= G(\Q)  +  \omega[W,\Q\,]  +  \beta \LL(\Q) D
\end{equation}
as an equation in $V\cong\bR^5$;  here $[W,\Q\,]=W\!\Q-\Q W$.
On the right hand side of~\eqref{e:sys1} the first term represents the molecular interactions in the absence of flow, derived for example from a Maier-Saupe interaction potential or Landau - de~Gennes free energy: thus $G$ is a frame-indifferent vector field in $V$.  In the second term, $W$ denotes the vorticity tensor, the anti-symmetric part of the (spatial homogeneous) velocity gradient, providing the rotational effect of the flow with constant coefficient~$\omega$.  In the third term $\LL(Q)$ is a linear transformation $V\to V$ applied to the rate-of-strain tensor $D$, the symmetric part of the velocity gradient,
and represents the molecular aligning effect of the flow: the linearity in~$D$ is a simplifying assumption. Here $\LL(Q)$ depends (not necessarily linearly) on $\Q$, and $\LL(\Q)D$ is frame-indifferent with respect to simultaneous coordinate choice for the flow and the molecular orientation.
  The coefficients $\omega$ and $\beta$ are constant scalars that depend on the physical characteristics of the liquid crystal molecule as well as the flow. In this study we take $\omega$ as fixed, and regard $\beta$ as a variable parameter.
\msk

In the Olmsted-Goldbart model~\cite{OlmGol} used in~\cite{Chillingworth2001}, \cite{VAWS2003} the term $\LL(\Q)D$ is simply a constant scalar multiple of~$D$.  A more detailed model for $\LL(\Q)D$ is the basis of a series of studies by the second author and co-workers~\cite{ForestEtAl2002}--\cite{ForestEtAl2004b},~\cite{LeeEtAl06}
as well as by many other authors~\cite{CRWX}, \cite{GrossoEtal01}, \cite{MarrucciMaffettone89}, \cite{PacZar}. We draw attention also to the earlier theoretical work~\cite{MacM92a,MacM92b} assuming a general form for~$\LL(\Q)D$ and where similar methods to ours are used to study equilibrium states (uniaxial or biaxial), although the question of periodic orbits in general and kayaking orbits in particular is hardly addressed, the existence of the latter having yet to be discovered.
\msk

We remark that although in this paper our underlying assumption is of spatial homogeneity   there have been studies of nematic liquid crystals dynamics in a nonhomogeneous environment: see among others~\cite{ChoF} for analytical results and~\cite{YWMF} for numerical simulations. 
\msk

A particular model of the form~\eqref{e:sys1} that \lq combines analytic tractability with physical relevance' \cite{MTZ} is the Beris-Edwards model~\cite{BE}, a basis for some more recent investigations~\cite{DMOY}, \cite{DHW}, \cite{MTZ}, \cite{WXZ} in both the PDE and ODE settings.
Here $G$ is the negative gradient of a degree four Landau-de~Gennes free energy function, while the term $\LL(\Q)D$ takes the form
\begin{equation}  \label{e:newterm}
\LL(\Q)D=\frac23D + [D,Q\,]^+ -2\tr(DQ)Q
\end{equation}
in which we use the notation
\begin{equation} \label{e:HK+def}
[H,K]^+:= HK + KH - \frac23 \tr(HK) I
\end{equation}
for any matrices $H,K\in V$; here and elsewhere $I$ denotes the $3\times3$ identity matrix. Observe that~\eqref{e:newterm} is a linear combination of a constant, a linear and a quadratic term in~$\Q$, that we denote (without their coefficients) respectively by $\LL^c(\Q)D,\,\LL^l(Q)D,\,\LL^q(\Q)D$. In this paper we initially work with an arbitrary choice of smooth\footnote{Throughout the paper we take smooth to mean $C^\infty$ although the results hold with sufficient finite order of differentiability.}  field $\LL(\Q)D$ subject to a natural assumption of frame-indifference. We then replace this by an arbitrary linear combination 
\begin{equation}  \label{e:lincombE}
\LL(\Q)D=  m_c\LL^c(\Q)D+  m_l\LL^l(Q)D +  m_q\LL^q(\Q)D
\end{equation}
which helps to keep track of the analysis, and also enables the results to apply to simpler models for which one or more of the $m_i$ may be zero.  For the Beris-Edwards model~\eqref{e:newterm} the ratios are $( m_c:  m_l:  m_q)=(2/3:1:-2)$, while for the Olmsted-Goldbart model~\cite{OlmGol}
the ratios are~$(1:0:0)$ and for the model in~\cite{MSV00} they are~$(\sqrt{3/10}:3/7:0)$. Moreover, in Appendix~\ref{s:genform} we pursue the analysis for general $\LL(\Q)D$, using the 7-term expression assumed for example in~\cite{MacM92a,MacM92b}, and show that with the exception of one term the results are the same as those for~\eqref{e:lincombE} albeit with different interpretation of the coefficients $ m_c, m_l, m_q$.  The exceptional term (being the symmetric traceless form of $Q^2D$) also fits into our overall framework as shown in the expressions~\eqref{e:newLam0} and~\eqref{e:newLam2} with~\eqref{e:7gens}.
\msk  

When $\beta=0$ the equation~\eqref{e:sys1} represents the {\em co-rotational case} or {\em long time regime}, as discussed in~\cite{MTZ}. If $\Q^*\in V$ satisfies $G(\Q^*)=0$ then frame-indifference of~$G$, interpreted as equivariance (covariance) of~$G$ under the action of the rotation group~$\SO(3)$ on~$V$, implies that every element~$\Q$ of the $\SO(3)$ group orbit $\kO$ of~$Q^*$ also satisfies $G(\Q)=0$.  If moreover $[W,\Q^*]=0$ then $F(\Q^*,0)=0$ and so $\Q^*$ is an equilibrium for~\eqref{e:sys1}: the rotational component of the shear flow leaves $\Q^*$ fixed.  This implies that $\Q^*$ has two equal eigenvalues, and if these are less than the third (principal) eigenvalue then $\Q^*$ represents a logrolling regime.
Moreover, $[W,\Q\,]$ is tangent to $\kO$ for every $Q\ne\Q^*\in\kO$ and so $\kO$ (which is topologically a copy of~$\bR\rP^2$) is an invariant manifold for the flow on $V$ generated by~\eqref{e:sys1} when $\beta=0$. The dynamical orbit of every such~$Q\in\kO$ is periodic, as it coincides with the group orbit of rotations about the axis orthogonal to the shear plane: in the language of equivariant dynamics~\cite{CL}, \cite{FDS}, \cite{HPF} it is a {\em relative equilibrium}. All of these periodic orbits represent kayaking regimes, except for a unique orbit representing tumbling, and they are neutrally stable with respect to the dynamics on $\kO$, as also is the logrolling equilibrium~$\Q^*$.
We discuss this geometry of the $\SO(3)$-action on~$V$ in more detail below; it plays a central role in what follows, as it must do in any global study of the system~\eqref{e:sys1}, an observation of course recognised by other authors~\cite{ForestEtAl2002}, \cite{MacM92a,MacM92b}.
\msk

There are a few rigorous mathematical proofs of the existence of tumbling
limit cycle orbits with limiting assumptions. By positing 2D rods, both with a tensor model~\cite{LeeEtAl06} and with the stochastic ODE~\cite{LelievreLeBris11}, proofs follow from the Poincar\'e-Bendixson theorem; for 3D rods with a tensor model the proof in~\cite{Chillingworth2001} uses geometric arguments on in-plane tensors. Until now, there has been no proof of existence of (stable) kayaking orbits, and the purpose of this paper is to provide a proof for second-moment tensor models~\eqref{e:sys1},\,\eqref{e:lincombE} at low rates of molecular interaction (although not necessarily low shear shear rates).  We thus consider a dynamical regime different from those considered by other authors in numerical simulations such as~\cite{RienackerEtAl2002b},~\cite{ForestWang2003}. A regime analogous to ours in considered in the theoretical work~\cite{MacM92a,MacM92b} using very similar methods, but in that case the molecules are assumed biaxial and it is equilibria rather than periodic orbits that are sought.
\msk

The approach we take is to regard $\beta$ as a small parameter and view~\eqref{e:sys1} as a perturbation of the co-rotational case. This enables us to use tools from equivariant bifurcation theory~\cite{CL}, \cite{GSS}, \cite{HPF}, \cite{Satt1978,Satt1979} and in particular
Lyapunov-Schmidt reduction over the group orbit~$\kO$ to obtain criteria for the persistence or otherwise of the periodic orbits of the co-rotational case after perturbation, and to determine the stability or otherwise of the resulting logrolling, tumbling and kayaking dynamics.
Our general results are independent of the choice of the
interaction field $G$, given that it is frame-indifferent and the logrolling state is an equilibrium: $G(\Q^*)=0$ (Assumptions~1 and~2 in Section~\ref{s:geomsym}) and also that the eigenvalues $\lam,\mu$ of the linearisation of~$G$ at~$\Q^*$ normal to $\kO$
are real and nonzero (Assumption~3 in Section~\ref{s:pert}). In addition we require a natural condition of frame-indifference for the perturbing field~$\LL(\Q)D$ (Assumption~4 in Section~\ref{s:pert}). Finally, the stability results require $\lam,\mu<0$ (Assumption~5 in Section~\ref{s:zeros}).  However, our methods do not allow us to make deductions when $\beta$ is large compared with the rotational coefficient~$\omega$. Other limit cycles are possible, and indeed are routinely observed numerically.
\msk

Our main result is Theorem~\ref{t:stability1} with Remark~\ref{r:klammu}, showing that the existence of a limit cycle kayaking orbit after perturbation depends on the ratio $\lam/\mu$ as well as the size of the product $\lam\mu$ relative to the rotation coefficient~$\omega$. We show also in Corollary~\ref{c:stabsum} that for the Beris-Edwards and Olmsted-Goldbart models the kayaking orbit is  linearly stable without further assumption.
\msk

This paper is organised as follows. In Section~\ref{s:geomsym} we discuss symmetries of the model and key features of the action of $\SO(3)$ on $V$ that it inherits from the usual action on~$\bR^3$. Of particular importance are the tangent and normal subspaces to the group orbit~$\kO$.  Section~\ref{s:pert} gives initial results showing the persistence of log-rolling and tumbling regimes after perturbation, and introduces the rotating coordinate system convenient for further analysis. In Section~\ref{s:poincare} a natural Poincar\'e section for the (dynamical) flow near $\kO$ is described and relevant first-order derivatives of the associated Poincar\'e map are calculated and shown to vanish.  Lyapunov-Schmidt reduction is applied in Section~\ref{s:lsred} to obtain a real-valued bifurcation function defined on a meridian of~$\kO$. This function happens to vanish to first order in $\beta$ and so we are obliged to pursue the $\beta$-expansion to second order. In Section~\ref{s:explicit} we choose $\LL(\Q)D$ explicitly as~\eqref{e:lincombE} and evaluate these second order terms.  Finally, in Section~\ref{s:zeros} the zeros of the bifurcation function are found and the conditions for existence and stability of kayaking motion are determined. For the specific cases of the Beris-Edwards and Olmsted-Goldbart models with Landau-de~Gennes free energy the criteria for existence and stability of kayaking orbits are stated explicitly.  Following a brief concluding section there are Appendices giving some technical results arising from symmetries that simplify the main calculations, as well as a discussion of how a fully general form of the molecular alignment term $\LL(\Q)D$ fits into the framework of our analysis.
\section{Geometry and symmetries of the system}  \label{s:geomsym}
The molecular interaction field~$G$ is independent of the coordinate frame and therefore equivariant (covariant) with respect to the action of the rotation group $\SO(3)$ on $V$ by conjugation induced from the natural action on $\bR^3$.  Therefore our first working assumption in this paper is the following.
\msk

\noindent{\bf Assumption 1}:  
$
\wtr G(\Q) = G(\wtr\Q)
$
for all~$\Q\in V$ and $R\in\SO(3)$
\msk

\noindent where we use the notation
\[
\wtr \Q:= R \Q R^{-1}.
\]

Further discussion of equivariant maps, in particular relating to the action of $\SO(3)$ on $V$ that we shall use extensively in this paper, is given in Appendix~\ref{s:emvf}.
\msk

Choosing coordinates $(x,y,z)\in\bR^3$  so that the shear flow velocity field has the form
$k(y,0,0)$ for constant $k\ne0$ the velocity gradient tensor is
\[
k \begin{pmatrix} 0 & 1 & 0 \\ 0 & 0 & 0 \\ 0 & 0 & 0 \\ 
  \end{pmatrix}
\]
with symmetric and anti-symmetric parts $kD/2$ and $-kW/2$ respectively, where 
\begin{equation} \label{e:DWdefs}
D = \begin{pmatrix} 0 & 1 & 0 \\ 1 & 0 & 0 \\ 0 & 0 & 0 \\ 
  \end{pmatrix}\,, \qquad 
W = \begin{pmatrix} 0 & -1 & 0 \\ 1 & 0 & 0 \\ 0 & 0 & 0 \\ 
  \end{pmatrix}.
\end{equation}
Without loss of generality we take $k=2$ since the coefficients $\omega$ and~$\beta$ in~\eqref{e:sys1} are at present arbitrary.
The rotational component $W$ corresponds to infinitesimal rotation about the $z$-axis.
\msk

A nonzero matrix $\Q\in V$ is called {\em uniaxial} if it has two equal eigenvalues less than the third, in which case it is invariant under rotations about the axis determined by the third eigenvalue. Matrices with three distinct eigenvalues are {\em biaxial}. In this paper an important role is played by the uniaxial matrix
\begin{equation}\label{e.Q*}
  \Q^*:=a
\begin{pmatrix} -1 & 0 & \,0\, \\ 0 & -1 & \,0\, \\ 0 & 0 & \,2\, \end{pmatrix}
\end{equation}
where $0<a<1/3$ for which the principal axis (largest eigenvalue) is the $z$-axis and about which $\Q^*$ is rotationally invariant. We take $a>0$ to ensure that $\Q^*$ is uniaxial, and the upper bound on $a$ is imposed for physical reasons since the second moment of the probability distribution defining the $\Q$-tensor has eigenvalues in the interval~$[0,1]$ and so those of $\Q$ are no greater than $2/3$: see~\cite{BMJ} for example.  We exclude $a=1/3$ as we shall need to work in a neighbourhood of~$\Q^*$.
\msk

Our second underlying assumption is that this phase is an equilibrium for the system~\eqref{e:sys1} in the absence of flow,  that is when $\omega=\beta=0$.  In other words
\msk

\noindent{\bf Assumption 2}: The coefficient $a$ is such that $G(\Q^*)=0$. 
\msk

With this assumption, the equivariance property of $G$ implies that $G$ vanishes on the entire $\SO(3)$-orbit~$\kO$ of $\Q^*$ in~$V$, and $\kO$ is an invariant manifold for the flow on $V$ generated by~\eqref{e:sys1} with $\beta=0$. The dynamical orbits on~$\kO$ coincide with the group orbits of rotation about the $z$-axis under which $\Q^*$ remains fixed, this being the only fixed point on~$\kO$ since if $\Q\in\kO$ and $[W,Q\,]=0$  then $\Q$ is a scalar multiple of and hence equal to~$\Q^*$.
\subsection{Rotation coordinates: the Veronese map}
For calculation purposes it is natural and convenient to take coordinates in~$V$ geometrically adapted to $\kO$.  We do this in a standard way by representing the orbit $\kO$ of~$\Q^*$ as the image of the unit sphere $\bS^2\subset\bR^3$ under the map
\begin{equation*}
  \kV:\bR^3\to V:\z\mapsto a(3\z\z^{\mathtt{t}}-|\z|^2I)
\end{equation*}
where again ${}^{\mathtt{t}}$ denotes matrix (or vector) transpose.   Here $\kV$ is the projection to $V$ of the case $n=3$ of the more general {\em Veronese map} construction $\bR^n\to\bR^m$ with $m=\binom{n}{2}$ and it represents $\kO$ as a {\em Veronese surface} in~$\bR^5$: see for example~\cite{GHAG} or~\cite{HAG}. It is straightforward to check that $\kV$ is equivariant with respect to the actions of $\SO(3)$ on $\bR^3$ and~$V$, that is if $R\in \SO(3)$ then
\begin{equation}   \label{e:eqvce}
\kV(R\z) = \wtr\kV(\z)
\end{equation}
for all $\z\in\bR^3$. 
Note that $\Q^*=\kV(\ee_3)$ where $\{\ee_1,\ee_2,\ee_3\}$ is 
the standard basis in~$\bR^3$, and that $\kV(\ee_1)$ and $\kV(\ee_2)$ are obtained from $\Q^*$ by permutation of the diagonal terms.
%For all $\Q\in\kO$ we have $|\Q|=|\Q^*|=\sqrt6a$.
\msk

On $V$ we have a standard inner product given by $\left<H,K\right>=\tr(H^{\mathtt{t}}K)=\tr(HK)$.
However, the Veronese map is quadratic and does not preserve inner products.  Nevertheless, up to a constant factor, its derivative does preserve inner products on tangent vectors to~$\bS^2$.  Explicitly 
  \begin{equation}  \label{e:dphi}
\D\kV(\z):\uu\mapsto a(3\z\uu^{\mathtt{t}}+3\uu\z^{\mathtt{t}}-2\z\cdot\uu\,I)
  \end{equation}
with the dot denoting usual inner product in~$\bR^3$,  from which it follows that for $\z\in \bS^2$ and $\uu,\vv\in \bR^3$  orthogonal to~$\z$ 
\begin{align}
  \D\kV(\z)\uu \cdot \D\kV(\z)\vv &=a^2\tr\bigl((3\z\uu^{\mathtt{t}}+3\uu\z^{\mathtt{t}}-2\z\cdot\uu\,I)(3\z\vv^{\mathtt{t}}+3\vv\z^{\mathtt{t}}-2\z\cdot\vv\,I)) \notag \\
  &=  a^2\tr ( \z\uu^{\mathtt{t}}\vv\z^{\mathtt{t}})=
 a^2\uu\cdot\vv.  \label{e:orthog}
\end{align}
Observe that the restriction of $\kV$ to $\bS^2$ is a double cover $\bS^2\to\kO$ since $\kV(-\z)=\kV(\z)$ for all~$\z\in \bR^3$.
Through $\kV$ the familiar latitude and longitude coordinates on $\bS^2$ go over to a corresponding coordinate system on~$\kO$. Any $\z\ne\ee_3\in \bS^2$ can be written using spherical coordinates as
\begin{equation}\label{e.spherCoord}
\z = R_\z\ee_3 = R_3(\phi)R_2(\theta)\,\ee_3
\end{equation}
for unique $\theta \bmod \pi$ and $\phi \bmod 2\pi$, 
where $R_j(\psi)$  denotes rotation by angle~$\psi$ around the $j$th axis  in $\bR^3$,~$j=1,2,3$, so that in particular 
\[
R_2(\theta)=\begin{pmatrix} \cos\theta & 0 & \sin\theta \\
                 0 & 1 & 0 \\
                 -\sin\theta & 0 & \cos \theta \end{pmatrix}, \quad
R_3(\phi)=\begin{pmatrix} \cos\phi & -\sin\phi & 0 \\
\sin\phi & \cos \phi & 0 \\
0 & 0 & 1  \end{pmatrix}.
\]
Hence by~\eqref{e.spherCoord} and equivariance~\eqref{e:eqvce} any $Z\in\kO$ can be written (not uniquely) as
\begin{equation}\label{e.Coord_O}
Z = \kV(\z) = \wtr_\z\Q^* = \wtr_3(\phi)\wtr_2(\theta)Q^* =: Z(\theta,\phi)
\end{equation}
for some $\z\in\bS^2$, as the counterpart of~\eqref{e.spherCoord} using rotations $\wtr$ on $V$ in place of $R$ on~$\bR^3$. We shall make frequent use of this notation throughout the paper.
\msk

By analogy with $\bS^2$ we call each closed curve $\theta=\const\ne 0 \bmod\pi$ on~$\kO$ a {\em latitude curve} and each curve $\phi=\const$ on~$\kO$ a {\em meridian}. It follows from~\eqref{e:orthog} that all latitude curves are orthogonal to all meridians.  The case $\theta=0\bmod\pi$ corresponds to~$Q^*$, and so we think of $Q^*$ as the {\em north pole} of~$\kO\cong\bR\rP^2$.
\begin{remark}
The expression~\eqref{e.spherCoord} provides the standard spherical coordinates on~$\bS^2$. Standard Euler angle coordinates on $SO(3)$ are obtained as the composition of three rotation matrices;   the Veronese coordinates for~$\kO$ provided by~\eqref{e.Coord_O} are obtained by disregarding one of those rotations.
  \end{remark}
\subsection{Isotypic decomposition}  \label{s:isotypic}
The rotation symmetry of $\kO$ about the north pole $\Q^*$ plays a fundamental role in our analysis of~\eqref{e:sys1} for sufficiently small nonzero~$\beta$, and enables us to choose coordinates in $V$ that are strongly adapted to the inherent geometry of the problem.  More generally, for any $\z\in \bS^2$ let
\[
\Sigma_\z=\{R\in\SO(3):R\z=\z\}\cong \SO(2)\subset\SO(3)
\]
denote the isotropy subgroup of $\z$ (namely the group of rotations about the $\z$-axis) under the natural action of $\SO(3)$ on $\bR^3$. Equivariance of $\kV$ implies that  $\Sigma_\z$ also fixes $Z=\kV(\z)$ in~$\kO$ under the conjugacy action, and moreover $Z$ is an isolated fixed point of $\Sigma_\z$ on~$\kO$ since $\z$ is an isolated fixed point of~$\Sigma_\z$ on $\bS^2$.
\msk

 At this point it is convenient to develop some further machinery from the theory of linear group actions to describe key features of the geometry highly relevant to our analysis. Introductions to the theory of group actions and orbit structures can be found for example   in~\cite{ABS},\,\cite{CHO},\,\cite{MZH}. We shall make much use of the further fact that corresponding to the action of $\Sigma_\z$ on $V$ there is an {\em isotypic decomposition} of~$V$ (for theoretical background to this notion see for example~\cite{CL}, \cite{FDS}, \cite{GSS}) into the direct sum of three $\Sigma_\z$-invariant subspaces
\begin{equation}  \label{e:isotyp}
  V = V_0^Z \oplus V_1^Z \oplus V_2^Z
\end{equation}
on each of which $\Sigma_\z$ acts differently:
the element $R_\z(\psi)\in\Sigma_\z$ denoting rotation about the $\z$-direction through angle~$\psi$ acts on $V_k^Z$ by rotation through $k\psi$ for $k=0,1,2$.  In particular, with $\z=\ee_3$ and $Z=Q^*$ writing $V_k^*=V_k^{Q^*}$ we have 
\begin{align}
  V_0^*&:=\Span\{E_0\} \label{e:spanq*} \\
  V_1^*&:=\Span\{E_1(\al)\}_{\al \in [0,2\pi)} \label{e:tq*} \\
  V_2^*&:=\Span\{E_2(\al)\}_{\al \in [0,\pi)}  \label{e:uq*}
  \end{align} 
where the mutually orthogonal matrices $E_0, E_1(\al),E_2(\al)$ are given by 
\begin{align}   \label{e:E_j_def}
&\hspace{3.5cm} E_0  := \frac{1}{a\sqrt 6} Q^*, \notag\\
E_1(\al) &:= \frac{1}{\sqrt 2}\begin{pmatrix} 0 & 0 & \cos\al \\
                         0 & 0 & \sin\al \\
                         \cos\al & \sin\al & 0
\end{pmatrix},	\quad
E_2(\al) := \frac{1}{\sqrt 2}\begin{pmatrix}
     \cos2\alpha & \sin2\alpha & 0  \\
     \sin2\alpha & -\cos2\alpha & 0  \\
     0 & 0 & 0
\end{pmatrix}
\end{align}
and we set 
\begin{align}\label{e.Eij_def}
E_{11} = E_1(0), \quad E_{12}  = E_1(\pi/2), \quad E_{21} = E_2(0), \quad
E_{22} = E_2(\pi/4).
\end{align}
Here $R_3(\phi)$ acts on $V_1^*$ and $V_2^*$ by 
\begin{equation}\label{e.R3Action}
\wtr_3(\phi)E_1(\al)=E_1(\al+\phi),\quad \wtr_3(\phi)E_2(\al)=E_2(\al+\phi)
\end{equation}
where we keep in mind that $E_2(\al)$ is defined in terms of $2\al$.
For $Z = Z(\theta,\phi)$ as in~\eqref{e.Coord_O} we use the notation
\begin{equation} \label{e:kzqznot}
E_1^Z(\al)= \wtr_3(\phi)\wtr_2(\theta)E_1(\al),\quad E_2^Z(\al)=\wtr_3(\phi)\wtr_2(\theta)E_2(\al)
\end{equation}
and
\begin{equation}  \label{e:E^Z_ij}
E^Z_{ij} =\wtr_3(\phi)\wtr_2(\theta)  E_{ij}, \quad i,j \in \{1,2\}.
\end{equation}
so that
\begin{align*}
  V_0^Z&=\Span\{E_0^Z\}  \\
  V_1^Z&=\Span\{E_1^Z(\al)\}_{\al \in [\,0,2\pi)}=\Span\{E_{11}^Z,\,E_{12}^Z\} \\
  V_2^Z&=\Span\{E_2^Z(\al)\}_{\al \in [\,0,\pi)}=\Span\{E_{21}^Z,E_{22}^Z\}\,. 
\end{align*}
A consequence of $\SO(3)$-equivariance is that for $Z\in\kO$ the derivative $\D G(Z):V\to V$ respects the decomposition~\eqref{e:isotyp} and commutes with the $\Sigma_\z$-rotations on each component.  A~further important consequence that simplifies several later calculations is the following.
\begin{prop}  \label{p:fixed}
If a differentiable function $f:V\to \bR$ is invariant under the action of~$\Sigma_\z$ then its derivative $\D f(Z):V\to\bR$ annihilates~$V_1^Z\oplus V_2^Z$.
  \end{prop}
\proof
If $f(\wtr \Q)=f(\Q)$ for all $R\in \Sigma_\z$ and $\Q\in V$ then $\D f(\wtr \Q)\wtr=\D f(\Q)$
and so in particular $\D f(Z)\wtr=\D f(Z)$ for all~$R\in\Sigma_\z$.  The only {\em linear} map $V\to\bR$ invariant under all rotations of $V_1^Z$ and of $V_2^Z$ must be zero on those components. 
\qed
\subsection{Alignment relative to the flow}  \label{s:inout}
Since the element $R_3(\pi)\in \SO(3)$ acts on~$V_k^*$ by a rotation through~$k\pi$ it follows that $V_0^* \oplus V_2^*$ is precisely the fixed-point space for the action of $R_3(\pi)$ on~$V$. Thus~$\Q=(q_{ij})\in V$ is fixed by~$\wtr_3(\pi)$ if and only if $q_{13}=q_{23}=0$, in which case $q_{33}$ is an eigenvalue with eigenspace the $z$-axis and the other eigenspaces lie in (or coincide with) the $x,y$-plane. It is immediate to check that if $\Q=pE_0 + qE_2(\al)$ then the eigenvalues of $\Q$ are $2p/\sqrt6$ and $(-p\pm \sqrt3q)/\sqrt6$ and so $\Q$ has two equal eigenvalues precisely when \begin{equation}  \label{e:ineigs}
  q=0\quad \text{or} \quad q=\pm\sqrt3p.
\end{equation}
In the first case
$\Q=pE_0$, while in the second case the eigenvalues are $2p/\sqrt6$ (repeated) and $-4p/\sqrt6$ so that if $p<0$ then $\Q$ is uniaxial with principal axis lying in the $x,y$-plane.
\details{From (2.2) we see that $E_0$ is diagonal with eigenvalues $-1/\sqrt 6$ (double) and $2/\sqrt 6$, while $E_2(\al)$  has eigenvalues $\pm 1/\sqrt 2$ and zero.
  }

From the point of view of the liquid crystal orientation relative to the shear flow such matrices~$\Q$ are called {\em in-plane}; nonzero matrices which are not in-plane are called {\em out-of-plane.}  
This agrees with standard terminology where tumbling and wagging dynamical regimes are described as in-plane  (see~\cite{FaraoniEtAl1999}, \cite{RienHess} for example), while logrolling and kayaking are out-of-plane. 
\msk

Let $C$ denote the equator $\{\theta=\pi/2\}$ of $\bS^2$, and let $\C=\kV(C)\subset\kO$ which we also call the {\em equator} of~$\kO$. It is straightforward to check that 
\begin{align}\label{e.equator}
  \C &= \{\kV(\cos\phi,\sin\phi,0) : 0 \le \phi < 2\pi \} \notag \\
  &= a\sqrt6\{\cos\tfrac{2\pi}3\,E_0 + \sin\tfrac{2\pi}3\,E_2(\phi) : 0 \le \phi < 2\pi\}\subset \kO \subset V.
\end{align}
\details{Verification of \eqref{e.equator}
\begin{align*}
\kV(\cos\phi,\sin\phi,0) &  = a (3 {\cos \phi \choose \sin\phi}  (\cos\phi, \sin\phi)- \id)
= a\left( \begin{array}{ccc} 3\cos^2\phi - 1 & 3\sin\phi\cos\phi & 0\\
 3\sin\phi\cos\phi &3\sin^2\phi - 1  & 0\\
 0 &  0 & -1\end{array}\right) \\
 & =  a\left( \begin{array}{ccc} \frac32\cos2\phi + \frac32- 1 & \frac32\sin2\phi  & 0\\
\frac32\sin2\phi &- \frac32 \cos2\phi + \frac32- 1  & 0\\
 0 &  0 & -1\end{array}\right)
\end{align*}
since $\cos2\phi=\cos^2\phi-\sin^2\phi = \cos^2\phi - (1-\cos^2\phi) = 2\cos^2\phi-1$
and so $\cos^2\phi = \frac12 \cos2\phi + \frac12$
and $\sin^2\phi =  \cos^2(\phi-\frac\pi{2}) = \frac12 (\cos(2\phi-\pi)) + \frac12
= - \frac12 \cos2\phi + \frac12$.}
\begin{prop}  \label{p:np+eq}
  {}\hfill
  \[ \kO\cap(V_0^*\oplus V_2^*) = \{\Q^*\} \cup \C.  \]
  \end{prop}
\begin{proof}
  Since $V_0^*\oplus V_2^*$ is the orthogonal complement to $V_1^*$ we see $\Q\in V_0^*\oplus V_2^*$ if and only if $\left<\Q,E_1(\al)\right>=0$ for all~$\al$. If $Z=\kV(\z)\in\kO$ then
  \[
  \left<Z,E_1(\al)\right>=3a\,\tr(\z\z^{\mathtt t}E_1(\al))=3a\z\cdot E_1(\al)\z.
  \]
  With $\z=(\cos\phi\sin\theta,\sin\phi\sin\theta,\cos\theta)^{\mathtt t}$ in usual spherical coordinates we find $\z\cdot E_1(\al)\z=(1/\sqrt2) \sin2\theta\cos(\phi-\al)$ which vanishes for all~$\al$ just when $\sin2\theta=0$, that is $\theta=0$ or $\theta=\pi/2$ corresponding to $Z=\Q^*$ or $Z\in\C$ respectively. \qed
  \end{proof}
\msk

When $\beta=0$ the equation~\eqref{e:sys1} reduces on~$\kO$ to
\[
\frac{\d \Q}{\d t} = \omega[W,\Q\,]
\]
since $G(\Q)=0$ for $\Q\in\kO$, giving solution curves $t\mapsto\wtr_3(\omega t)\Q$ each of which has least period $2\pi/\omega$ apart from the equilibrium $Q^*$ and the equator~$\C$: this has least period $\pi/\omega$, the equator $C$ of $\bS^1$ being a double cover of $\C$ via the Veronese map.  A matrix $\Q\in\C$ is in-plane and its dynamical orbit corresponds to steady rotation of period~$\pi/\omega$ about the origin in the shear plane, and so $\C$ represents a tumbling orbit.  All latitude curves of~$\kO$ other than the equator~$\C$ represent kayaking orbits of period~$T_0=2\pi/\omega$ and of neutral stability on~$\kO$ and so most of them are unlikely to persist for~$\beta\ne0$. The geometry can be visualised as follows: removing the poles at $\z=\pm\ee_3$ from $\bS^2$ leaves an (open) annulus foliated by circles of latitude, so that removing $\Q^*$ from $\kO$ leaves a M\"obius strip foliated by closed latitude curves each of which traverses the strip twice since $Z(\pi/2+\theta,\phi)=Z(\pi/2-\theta,\phi+\pi)$,
except for the \lq central curve' $\C$ given by $\theta=0$ which traverses it only once.
\msk

%In contrast, those $\Q\in\Span\{K(\al)\}_{\al\in [0,2\pi)}$ have eigenvalues $\nu,0,-\nu$ for $\nu\in\bR$ and are biaxial (no rotational symmetry) when $\nu\ne0$.  Note that $\wtr_3(\pi)K(\al)=K(\al+\pi)=-K(\al)$.
%
\subsection{Tangent and normal vectors to the group orbit~$\kO$}
The 2-dimensional tangent space $\kT^Z$ to $\kO$ at $Z\in\kO$ is spanned by infinitesimal rotations of~$Z$, that is
\[
\kT^Z=\Span\,\{\,[W_i,Z], i=1,2,3\,\}
\]
where
\begin{equation*}  %\label{e:wq}
\frac{\d}{\d\theta}\wtr_i(\theta)Q\big|_{\theta=0} = [W_i,Q\,] = W_i\Q -\Q W_i 
\end{equation*}
with
\begin{equation}  \label{e:wdefs}
W_1 = \begin{pmatrix}
  0\, & 0 & 0 \\  0\, & 0 & {-1} \\ 0\, & {1} & 0
\end{pmatrix}\,\quad
W_2 = \begin{pmatrix}
  \,0 & 0 & \,1 \\  \,0 & 0 & \,0 \\ -1\, & 0 & \,0
\end{pmatrix}\,\quad
W_3 = \begin{pmatrix}
  0 & {-1} & 0 \\ {1} & 0 & 0 \\ 0 & 0 & 0
  \end{pmatrix}.
 \end{equation}
However, for $Z\ne \Q^*$ the tangent space~$\kT^Z$ is also spanned by the tangents at $Z$ to the meridian and latitude curve of~$\kO$ through $Z$.
\begin{lemma}  \label{l:tanspaces}
  Let $Z\in\kO$ with $Z\ne\Q^*$.  The {\rm(}$1$-dimensional\/{\rm)} tangent spaces at $Z$ to the meridian and latitude curve of~$\kO$ through $Z$ are spanned by $E_{11}^Z$ and $E_{12}^Z$ respectively. 
\end{lemma}
\proof If $\phi=0$ the vectors $R_2(\theta)\,\ee_1$ and $\ee_2=R_2(\theta)\,\ee_2$ are respectively tangent to the meridian and latitude of $\,\bS^2$ through~$\z\in \bS^2$, and so applying $R_3(\phi)$ gives that the vectors $R_\z\ee_1$ and $ R_\z\ee_2$ are respectively tangent to the meridian and latitude through~$\z$ in the general case.
Therefore the corresponding tangent spaces at $Z=\V(\z)\in\kO$ are spanned by $\D\kV(\z)R_\z\ee_j$ for $j=1,2$ respectively.  The equivariance property~\eqref{e:eqvce} gives $\D\kV(R\z) R=\wtr\D\kV(\z)$ for any $\z\in\bS^2$ and $R\in\SO(3)$, and so as $\z=R_\z\ee_3$
\begin{equation}\label{e.DNue_j}
\D\kV(\z)R_\z\ee_j=\D\kV(R_\z\ee_3)R_\z\ee_j =\wtr_\z \,\D\kV(\ee_3)\ee_j
\end{equation}
for $j=1,2$. It is immediate to check using~\eqref{e:dphi} and 
\eqref{e:E_j_def} that   
\begin{equation}  \label{e:dphie3}
\D\kV(\ee_3)\ee_1=3 \sqrt 2 a E_{11},\quad \D\kV(\ee_3)\ee_2=3\sqrt 2a E_{12}
\end{equation}
and so applying $\wtr_\z $ gives the result.
\qed
\begin{cor}  \label{c:tanspacesK}
$\kT^Z=\Span\{E_{11}^Z,E_{12}^Z\} = V_1^Z$.   \qed
\end{cor}

The latitude curve through $Z=Z(\theta,\phi)\in\kO$ is the orbit of $Z$ under the action of $\Sigma_{\ee_3}=\{R_3(\phi)\}_{\phi\in[\,0,2\pi)}$ and so its tangent at~$Z$ is spanned by~$[W_3,Z\,]$. Indeed we find   
\begin{equation}  
[W_3,Z\,]=3\sqrt 2 a\sin\theta\,E_{12}^Z \label{e.W3Z}
\end{equation}
which we shall make use of below.
\details{CW: \fbox{CW adapted computation to  new $W_i$} any stuff in details environment
not to be included in final version: 
\begin{align*}
[W_3, Z\,] & = [W_3, \tilde R_\z\Q^*] = \tilde R_\z [\tilde R_\z^{-1} W_3,\Q^*] = 
\tilde R_\z [\tilde R_2(-\theta)  W_3,\Q^*]   = \tilde R_\z [\widehat{ R_2(-\theta)  e_3},\Q^*] 
= \tilde R_\z [(\cos \theta \hat \ee_3 - \sin\theta  \hat \ee_1),\Q^*]  
\\&
= \tilde R_\z [ (\cos \theta W_3- \sin\theta  W_1),\Q^*]   = 
\cos \theta\tilde R_\z [ W_3,\Q^*] - \sin\theta \tilde R_\z [ W_1,\Q^*] =- \sin\theta \tilde R_\z [ W_1,\Q^*]
 =  \sin\theta\sqrt 3  K^Z(\pi/2)%\label{e.W3Z}
\end{align*}
since 
\[
[W_1,Q^*] =
 \D {\mathcal V} (\ee_3))) \ee_1 \times \ee_3)
  =-\D{\mathcal V}(\ee_3) \ee_2 = -\sqrt3K(\frac\pi2).
\]
}
\msk

  From Corollary~\ref{c:tanspacesK} it follows that the normal space $\kN^Z$ to~$\kO$ at~$Z$ (the orthogonal complement in $V$ to the tangent space~$\kT^Z$) is given by
\begin{equation}  \label{e:normspace}
\kN^Z=V_0^Z\oplus V_2^Z.
\end{equation}
\section{The dynamical system after perturbation}  \label{s:pert}
Since $G$ is $\SO(3)$-equivariant and so in particular is equivariant with respect to the action of the isotropy subgroup~$\Sigma_\z$ on~$V$, the fact that $\Sigma_\z$ fixes~$Z$ means that the derivative $\D G(Z):V\to V$ respects the decomposition~\eqref{e:isotyp}.  Moreover, Assumption~2 and equivariance imply that~$G$ vanishes on the entire orbit~$\kO$ and so $\D G(Z)$ vanishes on~$\kT^Z=V_1^Z$.
\msk

Let~$\lam$  denote the eigenvalue of~$\D G(Z)$
on~$V_0^Z=\Span\{Z\}$, which by equivariance is independent of~$Z\in\kO$.  Since $\D G(Z)$ commutes with the rotation action of~$\Sigma_\z$ on~$V_2^Z$ its two eigenvalues on~$V_2^Z$ are complex conjugates and again independent of~$Z$; we assume them to be real (as they will be in the gradient case, of most interest to us) and denote them by~$\mu$ (repeated).
\msk

\noindent{\bf Assumption 3}:\quad  $\mu \in\bR$ and $\lam\mu\ne0$.
\msk

Even without the assumption $\mu\in\bR$ but with $\lam$ and $\Re(\mu)$ both nonzero the manifold $\kO$ is normally hyperbolic and therefore it persists as a unique nearby smooth flow-invariant manifold $\kO(\beta)$ for~\eqref{e:sys1} for sufficiently small~$|\beta|>0$:  see~\cite{FEN71}, \cite{HPS77} for the general theory invoked here.  Our interest is to discover which periodic orbits on~$\kO$ persist as periodic orbits after such a perturbation.
\begin{remark} The same approach is used in~\cite{MacM92a,MacM92b} to detect steady states (equilibria) bifurcating from more general group orbits. The geometry of the tangent and normal spaces to all orbits of $\SO(3)$ in~$V$ is exploited there in a significant way, although using constructions slightly different from ours.
\end{remark}
\msk

We now make explicit the assumption of linearity and frame-indifference of the contribution to~\eqref{e:sys1} from the non-rotational component of the shear flow.   The frame-indifference is natural for a physical model, while the linearity is generally assumed for simplicity: see for example~\cite{MacM89} and compare equation~(4) in~\cite{HND}.
\msk

\noindent{\bf Assumption 4}: The term $\LL(\Q)D$ is linear in~$D$, and
$\LL(\wtr\Q)\wtr D=\wtr\,\LL(\Q)D$ for all $\Q\in V$ and $R\in\SO(3)$.
\msk

It is immediate to check that Assumption~4 holds for~\eqref{e:lincombE}.  As a consequence we have the following elementary result.
\begin{prop} \label{p:Eindiff}
  If $\Q\in V$ is fixed by the action of $R\in\SO(3)$ then $\wtr \LL(\Q)D=\LL(\Q)\wtr D$. \qed
\end{prop}

\begin{cor}  \label{c:VIninv}
Each term of $F(\cdot,\beta)$ maps $\kN^*:=\kN^{\Q^*}$ into itself, and so 
the subspace $\kN^*$ is invariant under the flow of $F(\cdot,\beta)$ for all~$\beta$. 
\end{cor}
\proof
  Using \eqref{e:normspace} we see from Section~\ref{s:inout} that $ \kN^*$ is the fixed-point subspace for the action of~$R_3(\pi)$ on~$V$.  If $\Q\in \kN^*$ then $G(\Q)\in \kN^*$ by equivariance, and $[W,Q\,]\in \kN^*$ since~$W\in \kN^*$. Also Proposition~\ref{p:Eindiff} gives $\wtr_3(\pi)\LL(\Q)D=\LL(\Q)\wtr_3(\pi)D =\LL(\Q)D$ and so $\LL(\Q)D \in \kN^*$. \qed
\msk

From the symmetry and Corollary~\ref{c:VIninv} we have two immediate results: the north pole $\Q^*$ equilibrium (logrolling) and the equator~$\C$ periodic orbit (tumbling) persist after perturbation. 
\begin{prop} \label{e.perturbedEquator} Let $\omega\ne0$ be fixed.
For sufficiently small $|\beta|$ there exist for~\eqref{e:sys1}
\begin{itemize}
\item[(i)] a smooth family of equilibria $\Q^*\!(\beta)$ in $\kN^*$ with~$\Q^*\!(0)=\Q^*$;
\item[(ii)] a smooth family of periodic orbits $\C(\beta)$ in $\kN^*$ with $\C(0)=\C$ and with period tending to~$\pi/\omega$ as $\beta\to 0$.
\end{itemize}
\end{prop}
\proof (i) The eigenvalues of $\D G(\Q^*)$ are $\lam$, $0$ (repeated) and $\mu$ (repeated) with eigenspaces $V_0^*,V_1^*,V_2^*$ respectively, and the corresponding eigenvalues of $\Q\mapsto\omega[W,\Q\,]$ are $0,\pm\i\omega,\pm2\i\omega$ by~\eqref{e:spanq*}--\eqref{e:uq*} and the remarks preceding.  Hence the eigenvalues of~$\D F(\Q^*,0)$ are
\[
\lam, \pm\i\omega, \mu\pm2\i\omega
\]
and so by the Implicit Function Theorem there exists a smooth family of equilibria $\Q^*\!(\beta)$ with~$\Q^*\!(0)=\Q^*$ and with (for $\beta$ fixed) $\Q^*\!(\beta)$ the only equilibrium close to~$\Q^*$. Since $F(\cdot,\beta)$ maps $ \kN^*$ to itself by Corollary~\ref{c:VIninv}, the Implicit Function Theorem restricted to $ \kN^*$ implies that $\Q^*\!(\beta)\in  \kN^*$. 

(ii) The equator~$\C$ lies in $\kO\cap  \kN^*$ and is an isolated periodic orbit in~$\kN^*$ with characteristic multipliers there~$\e^{\pi\lam/\omega}$ and $\e^{\pi\mu/\omega}$ (repeated). We seek a fixed point for the first-return map on a local Poincar\'e section. Since the multipliers differ from~$1$, the Implicit Function Theorem applied on $ \kN^*$  gives the result.
\qed
\subsection{Rotated coordinates}  \label{s:comoving}
The effect of the perturbation $\beta \LL(\Q)D$ on the system~\eqref{e:sys1} when $\beta\ne0$ is most usefully understood in terms of a {\em co-moving coordinate frame} that rotates with the unperturbed system ($\beta=0$), since in these coordinates the rotation term~$[W,\Q\,]$ vanishes (cf.~\cite[Section~2]{MTZ}). Explicitly,  with $W=W_3$ and the substitution
\[
Q=\wtr_3(\omega t)Q_R
\]
and writing $\,\dot{}\,$ for $\tfrac{\d}{\d t}$ we have
\begin{equation*} 
  \dot\Q_R=\wtr_3(-\omega t)\dot\Q - \omega\wtr_3(-\omega t)[W_3,\Q\,] 
\end{equation*}
and so from~\eqref{e:sys1}
\begin{equation*}
 \dot\Q_R=\wtr_3(-\omega t)\bigl(G(\wtr_3(\omega t)\Q_R) +\omega[W_3,\Q\,]\bigr) +\beta\widetilde \LL(t,\Q_R) D - \omega\wtr_3(-\omega t)[W_3,\Q\,],
\end{equation*}
that is
\begin{equation}  \label{e:rotsys}
 \dot\Q_R =G(\Q_R) + \beta\widetilde \LL(t,\Q_R) D
  \end{equation}
using $\SO(3)$-equivariance of $G\,$; here for any $\Q\in V$ we write
\begin{equation} \label{e:tilEdef}
  \widetilde \LL(t,\Q) := \wtr_3(-\omega t)L(\wtr_3(\omega t)\Q)
  =\LL(\Q)\wtr_3(-\omega t)
\end{equation}
using Assumption~4 on frame-indifference of~$\LL(\Q)$.
Thus in~\eqref{e:rotsys} and with $\Q_R$ again written as~$\Q$ the rotation term $[W,\Q\,]$ has been removed from~\eqref{e:sys1} at a cost of replacing $D$ by the time-dependent term~$\wtr_3(-\omega t)D$. 
\msk

For given $\beta$ we denote the flow of ~\eqref{e:sys1} by $\varphi^t(\cdot,\beta):V\to V$, and denote the time evolution map of the nonautonomous system \eqref{e:rotsys} by
\[
\Phi^{\,t,\,t_0}(\cdot,\beta):V\to V.
\]
 To simplify notation in what follows we choose $t_0=0$ and write for $\Q\in V$
\[
\tvf^{\,t}(\Q,\beta):=\Phi^{\,t,0}(\Q,\beta).
\]
Observe in particular that for $T_0=2\pi/\omega$
\begin{equation}  \label{e:tilphi}
\vf^{T_0}(\Q,\beta)=\wtr_3(2\pi)\tvf^{T_0}(\Q,\beta)=\tvf^{T_0}(\Q,\beta).
\end{equation}
\subsection{Local linearisation: the fundamental matrix}
 An important role will be played by the linear transformation ({\em fundamental matrix})
\begin{equation}   \label{e:fund}
M(t,\Q) := \D \tvf^t(\Q,0):V\to V
\end{equation}
that satisfies the local linearisation of~\eqref{e:sys1}  (also called the {\em variational equation}~\cite[Ch.VIII]{LDE}, \cite[p.23]{HDE})  along the $\tvf$-orbit of~$\Q$ when $\beta=0\,$, namely
 \begin{equation}  \label{e:vareq}
\dot M(t,\Q)=\D G(\tvf^t(\Q,0))\,M(t,\Q), \quad M(0,\Q)=\id.
\end{equation}
For $Z=\kV(\z)\in\kO$ we have $G(Z)=0$ and so $\tvf^t(Z,0)=Z$ for all~$t\in\bR$ when $\beta=0$.  The variational equation~\eqref{e:vareq} for $\Q=Z$ thus becomes
\begin{equation}\label{e.dotM}
\dot M(t,Z)=A^Z\,M(t,Z), \quad M(0,Z)=\id
\end{equation}
where
\begin{equation}\label{e.AZ}
A^Z:=\D G(Z)
\end{equation}
is independent of~$t$.  Moreover, since $A^Z$ is $\Sigma_\z$-equivariant it has the decomposition
\begin{equation}  \label{e:AZdecomp}
A^Z=\lam p_0^Z + 0p_1^Z + \mu p_2^Z
\end{equation}
in terms of the linear projections $p_i^Z:V\to V_i^Z$ for $i=0,1,2$ and so
 \begin{equation}  \label{e:expA}
M(t,Z)=e^{tA^Z}=\diag\,\{e^{\lam t}, 1, e^{\mu t}\,\}
  \end{equation}
  with respect to the same decomposition~\eqref{e:isotyp}. In particular we have the following key fact.
\begin{corollary}  \label{c:ptm=pt}
$p_1^Z\,M(t,Z)=p_1^Z$ for $Z\in\kO$.  \qed
\end{corollary}
In what follows we shall make much use of this result, which states that the tangent space $\kT^Z=V_1^Z$ to $\kO$ at $Z$ consists of equilibria of the  variational equation at~$Z$.
\section{The Poincar\'e map}  \label{s:poincare}
All points $Z\in\kO$ satisfy $\vf^{T_0}(Z,0)=Z$ for $T_0=2\pi/\omega$. Our aim is to discover which of these periodic orbits persist for sufficiently small $|\beta|>0$, and to discern their stability. Systems of the form~\eqref{e:rotsys} (not necessarily with symmetry) have a long pedigree in the differential equations literature; in our application the symmetry plays a crucial role. The method we use is to apply Lyapunov-Schmidt reduction to a Poincar\'e map to obtain a 1-dimensional bifurcation function, and to look for its simple zeros when $\beta\ne0$: by standard arguments as in~\cite{BFL}, \cite{CLN}, \cite{CC}, \cite{CH}, \cite{GS1} for example, these correspond to persistent periodic orbits.  The existence of zeros $Z$ for small $|\beta|$ is established by taking a series expansion of the bifurcation function in terms of $\beta$ with coefficients functions of~$Z$.  Expressions for these coefficients in a general setting are given in~\cite{CLN}, and in principle we could simply set out to evaluate these expressions in our case.  However, in so doing we could lose sight of important geometric features of $V$ that are fundamental to the shear flow problem, and therefore instead we re-derive the relevant terms explicitly in our symmetric setting.   
\msk

\subsection{Poincar\'e section}
Let $Z=Z(\theta,\phi)\in\kO$ as in \eqref{e.Coord_O} with $Z\ne\Q^*$.
Let $\B^*$ denote the orthonormal basis for $V$:
\begin{equation}  \label{e:b*def}
  \B^*=\{E_0,E_{11},E_{12},E_{21},E_{22}\}
\end{equation}
where $E_0$ and $E_{ij}$ for $i,j \in \{1,2\}$ are defined in 
\eqref{e:E_j_def} and \eqref{e.Eij_def}.
Let $\B^Z$ denote the rotated basis (also orthonormal)    
\begin{equation}  \label{e:bzbasis}
\B^Z=\wtr_\z \B^*=\{E_0^Z,E_{11}^Z,E_{12}^Z,E_{21}^Z,E_{22}^Z\}
\end{equation}
with notation as in~\eqref{e:E^Z_ij}. 
From~\eqref{e:normspace} the $3$-dimensional normal space $\kN^Z$ to $\kO$ in~$V$ at $Z\in\kO$ is 
\begin{equation}\label{e.NZ}
\kN^Z= V_0^Z\oplus V_2^Z = \Span\{E_0^Z,E_{11}^Z, E_{12}^Z\}
\end{equation}
so that $V=\kT^Z\oplus\kN^Z$ by Corollary~\ref{c:tanspacesK},
and so for sufficiently small $\eps_0>0$ the union
\[
\U^{\eps_0}:=\bigcup_{Z\in\kO,\,0\le\eps<\eps_0}\bigl(Z+\kN^Z_\eps)
\]
forms  an open tubular neighbourhood of $\kO$ in~$V$, where $\kN^Z_\eps=\{\Q\in\N^Z:|\Q|<\eps\}$.
%  Recall that $\kN^Z= \kN^*$ in the particular case~$Z=\Q^*$. 
\msk

To construct a Poincar\'e section for the flow of~\eqref{e:sys1} we restrict $Z$ to lie on a chosen meridian
\[
\M =\M_\phi := \{Z(\theta,\phi), \theta \in [0,\pi)\}
\]
on $\kO$, so that
\begin{equation}  \label{e:psect}
\U^{\eps_0}_\M:=\bigcup_{Z\in \M,\,0\le\eps<\eps_0}\bigl(Z+\kN^Z_\eps)
\end{equation}
is a smooth $4$-manifold that intersects~$\kO$ transversely along~$\M$. Moreover, $F(Z,0)$  is nonzero and orthogonal to~$\U^{\eps_0}_\M$   for all $Z\in \M\setminus\Q^*$  since from~\eqref{e:sys1}   
\[
F(Z,0)=\omega[W_3,Z\,] = 3\sqrt 2 \omega a\sin\theta\,E^Z_{12}
\]
by Assumption~2 and~\eqref{e.W3Z}, while Lemma~\ref{l:tanspaces} shows that $E^Z_{12}$ is orthogonal to $\kN^Z$ and to~$\M$.
\msk

 Thus $\U^{\eps_0}_\M$ is a global (along $\M$) Poincar\'e section for all the (periodic) orbits through $\M\setminus\Q^*$ generated by the unperturbed vector field~$F(\cdot,0)$. The least period for $Z\in\M\setminus\Q^*$ is $T_0=2\pi/\omega$, with the exception that if $Z$ lies on the equator $\C$ then the least period is~$T_0/2=\pi/\omega$.  We next show that there exists $0<\eps\le\eps_0$ such that the corresponding $\U^{\eps}_\M$ is in an appropriate sense a Poincar\'e section for all orbits close to~$\kO$ generated by the perturbed vector field~$F(\cdot,\beta)$ including those lying in~$\Q^*+\kN^*_\eps$.
\begin{prop}  \label{p:pmap}
  Let $\M=\M_{\phi_{\,0}}$ be a meridian of $\kO$ with $\U^{\eps_0}_\M$ a tubular neighbourhood of $\kO$ restricted to $\M$ constructed using the normal bundle as in~\eqref{e:psect}. Then there exists $\beta_0>0$ and $0<\eps\le\eps_0$ and a smooth function
  \[
T: \U^{\eps}_\M \times (-\beta_0,\beta_0)\to \bR
  \]
such that if  $\Q\in\U^{\eps}_\M$ and $\Q\notin\Q^*+\kN^*_\eps$ then the future ($t\ge 0$) trajectory of the system~\eqref{e:sys1} from $\Q$ leaves $\U^{\eps}_\M$ and remains in $\U^{\eps_0}$, meeting $\U^{\eps_0}_\M$ for the second time when~$t=T(\Q,\beta)$.
Furthermore, $T(\Q,\beta)\to T_0=2\pi/\omega$ as
$(\Q,\beta)\to(\Q^0,0)$ with $\Q^0\in\M\cup(\Q^*+\kN^*_\eps)$.
\end{prop}
A key part of~Proposition~\ref{p:pmap} is the smoothness of $T$ on all of its domain including $\bigl(\Q^*+\kN^*_\eps\bigr)\times (-\beta_0,\beta_0)$, since there $F(\cdot,\beta)$ lies in $\kN^*$ by Corollary~\ref{c:VIninv} and so $T$ is not strictly a \lq time of second return'.
\msk

\details{\begin{align*}
[W_2, K(\pi/2)] &=
\left(\begin{array}{ccc} 0 & 0 & 1\\
0 & 0 & 0\\
-1 & 0 & 0
\end{array}\right)a \sqrt 3\left(\begin{array}{ccc} 0 & 0 & 0\\
0 & 0 & 1\\
0 & 1 & 0
\end{array}\right) - a \sqrt 3\left(\begin{array}{ccc} 0 & 0 & 0\\
0 & 0 & 1\\
0 & 1 & 0
\end{array}\right)\left(\begin{array}{ccc} 0 & 0 & 1\\
0 & 0 & 0\\
-1 & 0 & 0
\end{array}\right) \\
& = \tilde Q(\pi/4)
  \end{align*}
  }
\proof
Let
\[
\Q=Z+U^Z=\wtr_\z(\Q^*+U)\in\U_\M^{\eps_0}
\]
where $Z=Z(\theta,\phi)\in\kO$
as in~\eqref{e.Coord_O} and $U^Z\in \kN^Z$ with $U\in\kN^*$.  Then
  \begin{align}
    \dot \Q &= \frac{\partial \Q}{\partial\theta} \dot\theta + \frac{\partial \Q}{\partial\phi} \dot\phi +\frac{\partial \Q}{\partial U}\dot U\notag \\
    &=\wtr_\z\Big(\dot\theta\,[W_2,\Q^*+U]  + \dot\phi[\wtr_2(-\theta)W_3,\Q^*+U]
 + \dot U \Big) \label{e:dotQ}
    \end{align}
  using
  \[
  \frac{\partial{}}{\partial\psi} \wtr_j(\psi)= \wtr_j(\psi)[W_j,\cdot\ ]
  \]
  for $j=2,3$.
  \details{
\[
\tilde R_3(\phi) [W_3, \tilde R_2(\theta) \Q^*]
= \tilde R_3(\phi) \tilde R_2(\theta)[\tilde R_2(-\theta)W_3, \Q^*]
= \tilde R_\z [\tilde R_2(-\theta)W_3, \Q^*]
\]
  }
  We show that there is a positive smooth function on a neighbourhood of~$\M$ that coincides with $\dot\phi$ away from $\Q^*+\kN^*_\eps$, so that time~$t$ can in effect be replaced there by angle~$\phi$.
\msk

Writing
  \[
  U=u_0E_0+u_1E_{21}+u_2E_{22}\in \kN^*=V_0^*\oplus V_2^*
  \]
where $u_i\in\bR,\, i=0,1,2$  we make use of the identities
  \begin{equation}  \label{e:wids}
\begin{aligned}[t]
 [W_1,E_0]&= -\sqrt3E_{12}  \\
  [W_1,E_{21}]&= -E_{12} \\
  [W_1,E_{22}]&= \phantom{-}E_{11}
\end{aligned}
\qquad
\begin{aligned}[t]
 [W_2,E_0] &= \sqrt3E_{11}  \\
  [W_2,E_{21}]&= -E_{11} \\
  [W_2,E_{22}]&= -E_{12}
\end{aligned}
\qquad
\begin{aligned}[t]
 [W_3,E_0] &=0  \\
  [W_3,E_{21}]&= 2E_{22} \\
  [W_3,E_{22}]&= -2E_{21}
  \end{aligned}
  \end{equation}
  as well as 
\begin{equation}  \label{e:r2w3}
\wtr_2(-\theta)W_3=\cos\theta\,W_3 - \sin\theta\,W_1.
\end{equation} 
Inspecting the terms on the right hand side of~\eqref{e:dotQ} we find from~\eqref{e:wids}
\begin{equation}  \label{e:dotthetaterm}
[W_2,\Q^*+U\,]=(3\sqrt2a+\sqrt3u_0)E_{11} -  u_1E_{11} -u_2E_{12}
\end{equation}
and using~\eqref{e:r2w3}
\begin{align}
[\wtr_2(-\theta)W_3,&\,\Q^*+U\,]=\cos\theta[W_3,\Q^*+U\,]-\sin\theta[W_1,\Q^*+U\,]\notag\\
  &= 2\cos\theta (u_1E_{22}- u_2 E_{21})-\sin\theta(-(3\sqrt2a+\sqrt3u_0)E_{12} -u_1E_{12} + u_2E_{11}).  \label{e:dotphiterm}
\end{align}
Next, we take the inner product of~\eqref{e:dotQ} with $\widetilde E^Z:=\wtr_\z\widetilde E$ where
\[
\widetilde E:= u_2E_{11} + (3\sqrt2a+\sqrt3u_0 - u_1)E_{12} \ \in \kT^* = V_1^*
\]
orthogonal to the right hand side of~\eqref{e:dotthetaterm} and to $\dot U\in V_2^*$: since $\wtr_\z$ preserves inner products this annihilates the $\dot\theta$ and $\dot U$ terms in~\eqref{e:dotQ} 
and leaves
\begin{equation}  \label{e:phidot}
  \left<\widetilde E^Z,\dot\Q\right>
  =\dot\phi \,b(a,u)\sin\theta
\end{equation}
\details{
We have $\left< \tilde E^Z, \dot U \right>=0$ since $U \in N^Z$.
So
\begin{align*}\left<\tilde E^Z, \dot Q \right> &= \dot\theta <\tilde E^Z * (4.8)>
  + \dot\phi \wtr_\z \left<\tilde E^Z * (4.9) \right> \\
  &= \dot\phi \wtr_\z \left<\tilde E^Z * (4.9)\right> \\
&= \dot\phi \wtr_\z\left< \tilde E,  2\cos\theta(u_1E_{22}- u_2E_{21}) - \sin \theta(-(3\sqrt2 a + \sqrt3 u_0+u_1) E_{12} + u_2 E_{22}) \right> \\
  &=   \sin \theta\dot\phi \wtr_\z \left<\tilde E,((3\sqrt2 a + \sqrt3 u_0+u_1) E_{12} - u_2 E_{22}) \right>.
  \end{align*}
}
where
\begin{align}
  b(a,u)&= (3\sqrt2a +\sqrt3u_0 + u_1)(3\sqrt2a +\sqrt3u_0 - u_1)-u_2^2 \notag \\
  &= (3\sqrt2a +\sqrt3u_0)^2 - (u_1^2+u_2^2) \notag \\
  &=(18a^2 + O(|u|))  \notag
\end{align}
with $u=(u_0,u_1,u_2)$.
% so that $b(a,u)$ is nonzero for $|u|<\eps$ sufficiently small.
\msk

The next step is to replace $\dot \Q$ in~\eqref{e:phidot} by the right hand side $F(\Q,\beta)$ of~\eqref{e:sys1}.  Since $G$ respects the isotypic decomposition~\eqref{e:isotyp} we have by equivariance $G(Z+U^Z)\in \kN^Z$ and so 
\begin{equation}  \label{e:Gterm}
\left<\widetilde E^Z,G(Z+U^Z)\right> = 0.
\end{equation}
Also
\begin{align}  \label{e:w3term}
  \left<\widetilde E^Z,[W_3,Z+U^Z]\right> &=\left<\widetilde E,[\wtr_2(-\theta)W_3,\Q^*+U]\right>  \notag \\
  &= b(a,u)\sin\theta
\end{align}
as in~\eqref{e:dotphiterm}, \eqref{e:phidot}.
Finally, 
\begin{align}
  \left<\widetilde E^Z,\LL(\Q)D\right>&=\left<\widetilde E,\wtr_\z^{-1}\LL(\Q)D\right>
  =\left<\widetilde E,\LL(\Q^*+U)\wtr_\z^{-1}D\right>  \label{e:kzedterm}
\end{align}
from the frame-indifference Assumption~4.
Writing $\wtr_\z^{-1}D=D^0_T+D^0_N$ with $D^0_T\in\kT^*=V_1^*$ and $D^0_N\in\kN^*= V_0^*\oplus V_2^*$ we see from Corollary~\ref{c:VIninv} that $D^0_N$ makes zero contribution to~\eqref{e:kzedterm}, and so we focus on~$D^0_T$.  We have by elementary matrix evaluation 
\begin{align}
  \left<E_{11},D^0_T\right>=\left<E_{11},\wtr_\z^{-1}D\right>
  &=\left<\wtr_2(\theta)E_{11},\wtr_3(-\phi)D\right>  = \frac1{\sqrt2} \sin 2\theta\sin2\phi
\end{align}
\details{By (2.14) we have
$R_3(-\phi) D = R_3(-\phi) \sqrt2 E_{22} = \sqrt 2 E_{2}(\pi/4 - \phi)$
and so, by (6.6)
\begin{align*}
  \langle R_2(\theta) E_{11}, R_3(-\phi) D \rangle
&=   \sqrt2 < \langle R_2(\theta) E_1(0), E_{2}(\pi/4-\phi)> \\
&= \sqrt 2 /2\cos(2( \pi/4 -\phi))\sin2\theta  \\
&= 1/\sqrt 2 \cos(  \pi/2 - 2\phi)\sin2\theta  \\
  &= 1/\sqrt 2 \sin(2\phi)\sin2\theta
  \end{align*}
  }
since $D=\sqrt 2 E_{22}$, while   
\begin{align}
  \left<E_{12},D^0_T\right>=\left<E_{12},\wtr_\z^{-1}D\right>
  &=\left<\wtr_2(\theta)E_{12},\wtr_3(-\phi)D\right>=\sqrt 2\sin\theta\cos2\phi,
\end{align}
\details{
  \begin{align*}
    \langle R_2(\theta) E_{12}, R_3(-\phi) D \rangle
&=   \sqrt2 < R_2(\theta)  E_1(\pi/4),  E_{2}(\pi/4-\phi) > \\
    &= \sqrt2\sin(\pi/2-2\phi) \sin\theta = \sqrt2 \cos(2\phi) \sin\theta
    \end{align*}
  }
hence 
\begin{equation}  \label{e:dt0}
D_T^0=\sqrt2\sin\theta(\cos\theta\sin2\phi E_{11} + \cos2\phi E_{12}).
\end{equation}
Therefore from~\eqref{e:kzedterm} and~\eqref{e:dt0}
\begin{equation} \label{e:edq}
  \left<\widetilde E^Z,\LL(\Q)D\right> = \sqrt2\,L_{12}(\theta,\phi)\sin\theta
\end{equation}
where 
\begin{equation}
  L_{12}(\theta,\phi)=
\left<\widetilde E,\,\LL(\Q^*+U) (\cos\theta\sin2\phi\,E_{11} + \cos2\phi\,E_{12})\right>.
\end{equation}
Substituting~\eqref{e:Gterm}, \eqref{e:w3term} and~\eqref{e:edq} into~\eqref{e:phidot} with $\dot \Q = F(\Q,\beta)$ we obtain
\begin{equation}  \label{e:phidoteq}
 \dot\phi\,b(a,u)\sin\theta = \omega b(a,u)\sin\theta
 +\beta \sqrt2\,L_{12}(\theta,\phi)\sin\theta.
\end{equation}
Taking $\eps>0$ small enough so that $b(a,u)>0$ and
dividing~\eqref{e:phidoteq} through by $b(a,u)\sin\theta$, for $\beta$ sufficiently small we have $\dot\phi>\omega/2$ for $\theta\ne0$ and we observe that
$\dot\phi$ extends smoothly to $\theta=0$, corresponding to $\Q\in \kN^*$.
\msk

Consequently in $(\theta,\phi,U)$ coordinates for sufficiently small $|\beta|$ and $|u|$ the flow has positive component in the $\phi$ direction.  Since $\kO$ (given by $u=0$) in invariant under the flow of~\eqref{e:sys1} when $\beta=0$ and is given by $\phi$-rotation only, it follows that for $\eps$ and $|\beta|$ sufficiently small and $\Q=Z+U^Z\in\U^{\eps}$ we can define $T(\Q,\beta)$ to be the time-lapse from $\phi={\phi}_{0}$ to $\phi=\phi_{\,0}+2\pi$ if $Z\ne\Q^*$ and to be $T_0=2\pi/\omega$ when~$Z=\Q^*$.  \qed
\msk

Now we are able to define a Poincar\'e map close to $\kO$ and for sufficiently small~$|\beta|$.
\begin{definition}
  The {\em Poincar\'e map}  $P: \U^{\eps}_\M\times\bR  \to \U^{\eps_0}_\M$ is given by
  \begin{equation}  \label{e:poincare}
P(\Q,\beta):=\varphi^{T(\Q,\beta)}(\Q,\beta) \in \U^{\eps_0}_\M
  \end{equation}
where $T(\Q,\beta)$ is as defined in Proposition~\ref{p:pmap}.  
\end{definition}
By construction, every $\Q\in \U^{\eps}_\M$ lies in $Z+\kN^Z$ for some $Z=Z(\theta,\phi)$, where $\phi$ is unique $\bmod\,2\pi$ provided $\theta\ne0$, that is $Z\ne \Q^*$. Denoting $\phi=m(\Q)$ we can characterise  $T(\Q,\beta)$ for $\Q\notin \Q^*+\N_\eps^*$ as the unique value of $t$ close to $T_0=2\pi/\omega$ such that
\begin{equation} \label{e:elleq}
m(P(\Q,\beta)) =m \bigl(\varphi^{T(\Q,\beta)}(\Q,\beta)\bigr) = m(\Q).
\end{equation}
The bifurcation analysis that follows proceeds by expanding $P(\Q,\beta)$ in terms of the perturbation parameter~$\beta$.
\subsection{First order $\beta$-derivatives  }
Differentiating~\eqref{e:poincare} with respect to~$\beta$ gives
\begin{equation}  \label{e:diffpb}
 P'(\Q,\beta)=T'(\Q,\beta)F(P(\Q,\beta),\beta)
                 +\bigl(\vf^{T}\bigr)'(\Q,\beta)|_{T = T(\Q,\beta)}
\end{equation}
where here and throughout we use ${}'$ to denote differentiation with respect to the second component~$\beta$. At $(\Q,\beta)=(Z,0)$ the expression~\eqref{e:diffpb} becomes
\begin{align}  
  P'(Z,0) &= T'(Z,0) F(Z,0) + \bigl(\vf^{T_0}\bigr)'(Z,0) \notag \\
  &=  T'(Z,0)F(Z,0) + \bigl(\tvf^{T_0}\bigr)'(Z)  \label{e:diffpb0}
\end{align} 
using \eqref{e:tilphi}. We now turn attention to evaluating~$T'(Z,0)$.
\msk

Differentiating~\eqref{e:elleq} with respect to~$\beta$ at $(\Q,\beta)=(Z,0), Z\ne Q^*$ gives
\begin{equation}  \label{e:diffbeta}
T'(Z,0)\,\D m(Z)F(Z,0) +  \D m(Z)\bigl(\vf^{T_0}\bigr)'(Z,0)=0.
\end{equation}
By construction $ m(\Q)= m(Z)$ for all $\Q\in Z+\kN^Z$ and therefore $\D m(Z)$ annihilates~$\kN^Z$.
Recall from Lemma~\ref{l:tanspaces} that the tangent space to $\M$ at~$Z$ is spanned by  $E_{11}^Z(0)$ while the tangent space to the latitude curve  through~$Z$ is spanned by $E_{12}^Z(\pi/2)$.  It follows that the derivative $\D m(Z):V\to\bR$ annihilates $E_{11}^Z(0)$ and is an isomorphism from
$\Span\{E_{12}^Z(\pi/2)\}$ to $\bR$, so that in particular
\begin{equation}  \label{e:dlfnz}
\D m(Z)F(Z,0)\ne 0
\end{equation}
since from~\eqref{e.W3Z} we see 
\begin{equation}\label{e.FZ}
F(Z,0) = \omega [W_3,Z\,] = 3\sqrt2 a\omega\sin\theta E_{12}^Z.
\end{equation}
We next introduce a variable that plays a central role in subsequent calculations.
\begin{definition} \label{d:defy}
 \quad $y(t,\Q):=\bigl(\tvf^t\bigr)'(\Q,0)$.
\end{definition}
From~\eqref{e:tilphi} we see in particular
\begin{equation}\label{e.yT}  
 y(T_0,\Q)= \bigl(\varphi^{T_0}\bigr)'(\Q,0).
 \end{equation}
With this notation we can write~\eqref{e:diffbeta} as
\begin{equation}  \label{e:diffbeta2}
T'(Z,0)\,\D m(Z)F(Z,0) +  \D m(Z)y(T_0,Z)=0.
\end{equation}
A consequence of Assumption~4 is that the second term in~\eqref{e:diffbeta2} vanishes.
\begin{lemma} \label{l:Dmy0}
 \quad $\D m(Z)y(T_0,Z)=0$.
  \end{lemma}
\proof Substituting $Q_R=\tvf^t(Q,\beta)$ into~\eqref{e:rotsys} and differentiating with respect to~$\beta$ at $\beta=0$ shows that $y(t,\Q)$ satisfies the differential equation
\begin{equation}  \label{e:xieq}
\dot y(t,\Q)=\D G(\Q) y(t,\Q) +\widetilde \LL(t,\Q)D
\end{equation}
with $y(0,\Q)=0$ and $\widetilde \LL(t,\Q)D$ as in~\eqref{e:tilEdef}.  Solving this equation by the usual variation of constants formula~\cite{HDE} we obtain 
\begin{equation}  \label{e:ytq}
 y(t,\Q)=\int_0^tM(t-s,\Q)\widetilde \LL(s,\Q)D\d s
\end{equation}
in terms of the fundamental matrix $M(t,\Q)$ as in~\eqref{e:fund}, and so in particular for each $Z\in\kO$
\begin{equation} \label{e:yint}
y(T_0,Z) = \int_0^{T_0}\bigl(e^{\lam(T_0-s)}p_0^Z\widetilde \LL(s,Z)D + p_1^Z\widetilde \LL(s,Z)D +e^{\mu(T_0-s)} p_2^Z \widetilde \LL(s,Z)D\bigr) \d s
\end{equation}   
using~\eqref{e:expA}.  Hence
\begin{equation} \label{e:pty}
p_1^Z\,y(T_0,Z) = p_1^Z \int_0^{T_0}\widetilde \LL(s,Z)D\d s = 0
\end{equation}
as is clear from~\eqref{e:tilEdef}. Thus $y(T_0,Z)\in \kN^Z$ and so $\D m(Z) y(T_0,Z)=0$ and the lemma is proved.
\qed
\msk

In view of Lemma~\ref{l:Dmy0} the expression~\eqref{e:diffbeta}  becomes
\[
T'(Z,0)\,\D m(Z)F(Z,0) = 0
\]
and hence from \eqref{e:dlfnz} we arrive at the following key result.
\begin{prop}  \label{p:tpzero}
  \quad $T'(Z,0)=0$ for all $Z\in\kO\setminus\Q^*$, and so by continuity for all~$Z=\kO$ .  \qed 
\end{prop}
The analogous result holds for the $\Q$-derivative $\D T(\Q,\beta)$ at $(Z,0)$.
\begin{prop}  \label{p:dtzero}
 \quad $\D T(Z,0)=0$ for all $Z\in\kO.$
\end{prop}
\proof Here differentiating~\eqref{e:elleq} with respect to~$\Q$ at $(\Q,\beta)=(Z,0),Z\ne Q^*$ gives
\begin{equation}
\D m(Z)\bigl((\D T(Z,0)H)F(Z,0) + \D \tvf^{T_0}(Z,0)H\bigr) = \D m(Z)H
\end{equation}
for $H\in V$, that is
\begin{equation}
(\D T(Z,0)H)\,\D m(Z)F(Z,0) + \D m(Z)\e^{T_0 A^Z}H= \D m(Z)H
\end{equation}
from~\eqref{e:fund} and~\eqref{e.AZ}. Since $e^{T_0A^Z}$ respects the splitting $V=\kT^Z\oplus \kN^Z$ (see~\eqref{e:expA}) and $\D m(Z)$ annihilates $\kN^Z$ we deduce $\D T(Z,0)H=0$ for $H\in \kN^Z$ using~\eqref{e:dlfnz}, while if  $H\in \kT^Z$ then  $\e^{T_0 A^Z}H=H$ and so also $\D T(Z,0)H=0$.  The result follows for $Z=Q^*$ by continuity.  
\qed   
\section{Lyapunov-Schmidt reduction}  \label{s:lsred} 
Our aim in this section is to seek solutions $\Q=\Q(\beta)\in\U^{\eps}_\M$  for sufficiently small  $|\beta|>0$
to the equation
\begin{equation} \label{e:pfix}
P(\Q,\beta) =\Q,
\end{equation}
where $P: \U^{\eps}_\M\times\bR  \to \U^{\eps_0}_\M$ is as in~\eqref{e:poincare}, and to determine the stability of the $T(\Q(\beta),\beta)$-periodic orbit of~\eqref{e:sys1} that each of these represents.  Of particular interest are  out-of-plane solutions, corresponding to kayaking orbits.
We apply Lyapunov-Schmidt reduction to \eqref{e:pfix} along $\M$ exploiting the $\SO(3)$-invariant tangent and normal structure to~$\kO$.
\msk

Lyapunov-Schmidt reduction is a fundamental tool in bifurcation theory, and amounts to a simple application of the Implicit Function Theorem.
 Accounts can be found in many texts such as~\cite[Sect.~5.3]{AMP},\,\cite[Sect.~4.4]{ChLe}  ,\,\cite[Sect.~2.4]{CH},\,\cite[Sect.~I\S3]{GS1},\,\cite[Sect.~I.2]{KIEL},\,\cite[Sect.~2.2]{KieLau83},\,\cite[Sect.~3.1]{VBH0} and surveys~\cite{Chic},\,\cite{JKH},\,\cite{JEM}. Although the method is local in origin, it can be applied globally on a manifold on which a given vector field vanishes, or on which given mapping is the identity, by piecing together local constructions and invoking the uniqueness clause of the Implicit Function Theorem. This is the version we use here, which
fits into the general framework of~\cite{BFL,CLN,LNT} and has significant overlap with the geometric methods of~\cite{MacM92a,MacM92b}.  
\msk

Let $\Q\in \kN^Z$.  Then 
$\Q_N =Q$ and $\Q_T =0$ where the suffices $N,T$ will denote projections to $\kN^Z,\kT^Z$ respectively. 
Hence~\eqref{e:pfix} is equivalent to the pair of equations
\begin{align}
  P_N(\Q,\beta) &=\Q_N =Q \label{e:normeq}  \\
  P_T(\Q,\beta) &=\Q_T =0. \label{e:taneq} 
  \end{align}
 When $\beta=0$ the equation~\eqref{e:normeq} is satisfied by $\Q=Z$, and by~\eqref{e:expA} the $\Q$-derivative
\[
\D P_N(Z,0)|_{\kN^Z}:\kN^Z \to \kN^Z
\]
has eigenvalues $\{e^{\lam T_0},e^{\mu T_0},e^{\mu T_0}\}$ with $\lam,\mu$ both nonzero, so
\[
\D P_N(Z,0)|_{\kN^Z}-\id_{\kN^Z} :\kN^Z \to \kN^Z
\]
is an isomorphism. It follows by the Implicit Function Theorem  and the (smooth) local triviality of the normal bundle, as well as the compactness of~$\M$, that for all sufficiently small~$|\beta|$ there exists a smooth section
\[
Z\mapsto\sigma(Z,\beta)\in \kN^Z
\]
of the normal bundle of $\kO$ restricted to $\M$ such that for sufficiently small~$|\beta|$ the map
\[
\M\to \U^\eps_\M : Z \mapsto Z + \sigma(Z,\beta)
\]
has the property that
\begin{equation}  \label{e:section}
P_N(Z+\sigma(Z,\beta),\beta) = Z+\sigma(Z,\beta) \in \kN^Z
\end{equation}
for all $Z\in\M$, with $\sigma(Z,0)=0$.
\msk

It therefore remains to solve the equation~\eqref{e:taneq} along~$\M$ given~\eqref{e:section}, that is to solve the {\em reduced equation} or {\em bifurcation equation}
\begin{equation} \label{e:bifeqn0}
P_T(Z+\sigma(Z,\beta),\beta)= 0 \in \kT^Z
\end{equation}
for $(Z,\beta)\in \M\times\bR$ and for $|\beta|$ sufficiently small.
Since $\kT^Z=V_1^Z=\Span\{E_{11}^Z,E_{12}^Z\}$ and by construction the Poincar\'e map $P$ has no component in the direction of the vector field $E_{12}$, the bifurcation equation~\eqref{e:bifeqn0} can by Lemma~\ref{l:tanspaces} be written more specifically as
\begin{equation} \label{e:bifeqn1}
P_{11}(Z+\sigma(Z,\beta),\beta)= 0
\end{equation}
with $P_{11}=p_{11}^ZP$ where $p_{11}^Z$ denotes projection to~$\Span\{E_{11}^Z\}$. 
%that is
%\begin{equation}   \label{e:bifeqn2}
%P_K(Z+\sigma(Z,\beta),\beta)  =\frac{1}{6a^2} \left< K^Z(0),P(Z+\sigma(Z,\beta),\beta) \right> K^Z(0).
%\end{equation}
We thus seek the zeros of the {\em bifurcation function} $\F(\cdot,\beta):\M\to\bR$ where
\begin{equation}  \label{e:bifdef} 
P_{11}(Z+\sigma(Z,\beta),\beta) = \F(Z,\beta)E_{11}^Z
  \end{equation}
for sufficiently small $|\beta|>0$. We shall find these by
taking a perturbation expansion of~$\F(Z,\beta)$ in terms of~$\beta$.
\subsection{Perturbation expansion of the bifurcation function}
First, we need a $\beta$-expansion of the Poincar\'e map $P$ which we write as
\begin{equation}
P(\Q,\beta)=P^{\,0}(\Q) + \beta P^1(\Q) + \beta^2 P^2(\Q) + O(\beta^3)
\end{equation}
for $Q\in \kN^Z, Z\in\M$. We also make use of the \lq approximate' Poincar\'e map 
\begin{equation}\label{e.tildeP} 
  \til P(\Q,\beta) :=\tvf^{T_0}(\Q,\beta),
\end{equation}
 with $\beta$-expansion
\begin{equation}
  \til P(\Q,\beta) = \til P^{\,0}(\Q) + \beta \til P^1(\Q) + \beta^2 \til P^2(\Q) + O(\beta^3)
\end{equation}
noting that $\til P(Z,0) =P(Z,0)$ by \eqref{e:tilphi}.
Although $\til P$ is not the same as $P$, the next result shows that up to second order in~$\beta$ at $\Q=Z \in \kO$  it differs from~$P$ only in the direction of the unperturbed vector field~$F(Z,0)$.
\begin{prop}  \label{p:tilP}
  \[ P^{\,i}(Z) = \til P^{\,i}(Z)  \]
for $i=0,1$, and
  \[ P^2(Z) - \til P^2(Z)  \in \Span\{F(Z,0)\}.  \]
\end{prop}
\proof
Of course $P^{\,0}(Z)=\til P^{\,0}(Z) = Z$, and from~\eqref{e:diffpb0} we have  $P^1(Z)=\til P^1(Z)$ since $T'(Z,0)=0$ by Proposition~\ref{p:tpzero}.  Next, differentiating~\eqref{e:diffpb} with respect to~$\beta$ at $(\Q,\beta)=(Z,0)$ we obtain
\begin{equation*}
 P^2(Z) =\frac12  P''(Z,0) =\frac12 T''(Z,0)F(Z,0) + \til P^2(Z)   % \label{e:p2z}
\end{equation*}
again using (twice) the fact that $T'(Z,0)=0$.
\qed
\msk

In expanding $P(Z+\sigma(Z,\beta),\beta)$ we shall require the first and second $\Q$-derivatives $\D P(\Q,\beta)$ and $\D^2P(\Q,\beta)$   of $P$ at $(\Q,\beta)=(Z,0)$.  Recall that the tangent space to $\U_\M$ at $Z\in\M$ is
$F(Z,0)^\perp=\Span\{E_{11}^Z\}\oplus \kN^Z$. 
\begin{prop}   \label{p:dpz}
  \begin{equation} \D P^{\,0}(Z) = \D \til P^{\,0}(Z) \label{e:dp0z}  \end{equation}
  while for $H,K\in F(Z,0)^\perp$
  \begin{equation} \D P^1(Z)H - \D \til P^1(Z)H  \in \Span\{F(Z,0)\} \label{e:dp1z} \end{equation}
 and
  \begin{equation} \D^2 P^{\,0}(Z)(H,K) - \D^2\til P^{\,0}(Z)(H,K)  \in \Span\{F(Z,0)\}.  \label{e:d2p0z}  \end{equation}
\end{prop}
\proof
 For $H\in F(Z,0)^\perp$  we have
\begin{equation}  \label{e:diffPQ}
\D P(\Q,\beta)H = \bigl(\D T(\Q,\beta)H\bigr) F(P(\Q,\beta),\beta) + \D\vf^t(\Q,\beta)|_{t=T(\Q,\beta)}H
\end{equation}
which at $(\Q,\beta)=(Z,0)$ becomes
\[
\D P(Z,0)H = \bigl(\D T(Z,0)H\bigr) F(Z,0) + \D\til P(Z,0)H
\]
giving~\eqref{e:dp0z} in view of Proposition~\ref{p:dtzero}.
The expression~\eqref{e:diffPQ} shows that    $\D P$ and $\D\vf^t|_{t=T(\Q,\beta)}$ 
differ by a scalar multiple of $F(P(\Q,\beta),\beta)$, and moreover this scalar multiple $\D T(Q,\beta)H$ vanishes when $(\Q,\beta)=(Z,0)$ by Proposition~\ref{p:dtzero}. Hence on one further differentiation both the $\Q$-derivative and the $\beta$-derivative of  $\D P$ at $(Z,0)$ differ from those of $\D\vf^t|_{t=T(Z,0)}= \D\tvf^{T_0}=\D\til P$ only by a scalar multiple of $F(Z,0)$. Therefore $\D^2P^{\,0}(Z)$ and $\D P^1(Z)$ differ from $\D^2\til P^{\,0}(Z)$ and $\D\til P^1(Z)$ respectively by scalar multiples of~$F(Z,0)$.     \qed
\subsection{First order term of the bifurcation function}  
Here we denote 
\[
P_{11}'(Z,0):=\frac{\d}{\d\beta}  P_{11}(Z+\sigma(Z,\beta),\beta)|_{\beta=0}
=\F'(Z,0)E_{11}^Z
\]
as in~\eqref{e:bifdef}, and likewise for the second derivatives.
\begin{prop} \label{p:dbpkzero}
\quad
$
 P_{11}'(Z,0)= 0.
$
\end{prop}
\proof
Differentiating~\eqref{e:bifeqn1} with respect to~$\beta$ at $\beta=0$ gives
\begin{align}
  P_{11}'(Z,0)&=p^Z_{11}\bigl(\D P^{\,0}(Z)\sigma'(Z,0)+P^1(Z)\bigr) \\
  &= p^Z_{11}M(T_0,Z)\sigma'(Z,0) + p^Z_{11}\til P^1(Z)  \label{e:pkprime}
\end{align}
using \eqref{e:fund} and  Proposition~\ref{p:tilP} for $i=0,1$.  Now
\begin{equation}  \label{e.pM}
p^Z_{11} M(T_0,Z) = p^Z_{11}
\end{equation}
by Corollary~\ref{c:ptm=pt} and $p^Z_{11}\sigma'(Z,0)=0$ since $\sigma'(Z,0)\in \kN^Z$.  Also 
$\til P^1(Z) = y(T_0,Z)$ as in~\eqref{e.yT}, and $p^Z_{11} y(T_0,Z)=0$ from~\eqref{e:pty}. Thus both terms on the right hand side of~\eqref{e:pkprime} vanish.  
\qed
\msk
 
  A geometric interpretation of Proposition~\ref{p:dbpkzero} is that {\em to first order} in~$\beta$ the $\SO(3)$-orbit $\kO$, on which every dynamical orbit (other than the fixed point~$Q^*$) is $2\pi/\omega$ periodic, perturbs to an invariant manifold with the same dynamical property, so that neutral stability of all periodic orbits is preserved.
\msk
\subsection{Second order term of the bifurcation function}
Given that the first order term in the $\beta$-expansion of~$\F(Z,\beta)$ vanishes by Proposition~\ref{p:dbpkzero} we turn to the second order term.  Differentiating $P_{11}(Z,\beta)$ twice with respect to~$\beta$ at $\beta=0$ we obtain from the left hand side of~\eqref{e:bifeqn1}
\details{CW added, do not remove:
$P_{11}'(Z,\beta) = p^Z_K \D P(Z,\beta) \sigma'(Z,\beta) + p^Z_K P'(Z,\beta)$ 
}
\begin{eqnarray}  \label{e:b2deriv}
 P_{11}''(Z,0)&=\D^2P^{\,0}_{11}(Z)(\sigma'(Z,0))^2 + 2\D P^1_{11}(Z)\sigma'(Z,0) \notag \\
  &+\D P^{\,0}_{11}(Z)\sigma''(Z,0) +2 P^2_{11}(Z)
\end{eqnarray}
where we write $P^{\,i}_{11}$ for $p^Z_{11}P^{\,i}, i=0,1,2$.  
\begin{remark}
The expression~\eqref{e:b2deriv} is a particular case of the formula for the second order term of the bifurcation function in a general setting derived in~\cite[Appendix A]{CLN}.
\end{remark}
To evaluate~\eqref{e:b2deriv} a significant simplification can be made.
\begin{prop}   \label{p:replace}
$P$ may be replaced by~$\til P$ in all terms on the right hand side of~\eqref{e:b2deriv}. 
\end{prop}
\proof  By Proposition~\ref{p:tilP} and Proposition~\ref{p:dpz}  each term differs from its counterpart with $\til P$ by a scalar multiple of $F(Z,0)$, which is annihilated by~$p_{11}^Z$.  \qed
\msk

We next investigate in turn each of the terms of~\eqref{e:b2deriv} with $\til P$ in place of~$P$.
\subsubsection{First $\Q$-derivative of $\til P^{\,0}$}
As $\sigma(Z,\beta)\in \kN^Z$ its $\beta$-derivatives also lie in $\kN^Z$,  and with $\D\til P^{\,0}(Z)=M(T_0,Z)$ it follows from \eqref{e.pM} that
\begin{equation}  \label{e:dp0}
\D\til P_{11}^{\,0}(Z)\sigma''(Z,0)=p^Z_{11}\sigma''(Z,0)=0.
\end{equation}
\subsubsection{Second $\Q$-derivative of $\til P^{\,0}$}
Expanding  
\[
\tvf^t(\Q,\beta)=\tvf_0^t(\Q)+\beta\tvf_1^t(\Q)+\beta^2\tvf_2^t(\Q)+O(\beta^3)
\]
so that in particular $\tvf_0^t(\Q)=\tvf^t(\Q,0)$, we see from~\eqref{e:rotsys} with $\beta=0$ that $\D^2\tvf_0^t(\Q)$ satisfies the equation
\begin{equation}
\D^2\dot\tvf_0^t(\Q)=\D^2G(\tvf_0^t(\Q))\bigl(\D\tvf_0^t(\Q)\bigr)^2 + \D G(\tvf_0^t(\Q))\D^2\tvf_0^t(\Q)
  \end{equation}
and so we obtain from the variation of constants formula and~\eqref{e:fund}
  \begin{equation}
    \D^2\til P^{\,0}(Z)= \D^2\tvf_0^{T_0}(Z) =\int_0^{T_0} M(T_0-s,Z)\,\D^2G(Z)\bigl(M(s,Z)\bigr)^2 \d s.
  \end{equation}
Since $\sigma'(Z,0)\in \kN^Z$ and so by \eqref{e:expA} also 
  \(
  M(s,Z)\sigma'(Z,0)\in \kN^Z=V_0^Z\oplus V_2^Z
  \)
we  have
  \begin{equation}  \label{e:d2p0}
    \D^2\til P_{11}^{\,0}(Z)(\sigma'(Z,0))^2
    =\int_0^{T_0} p_{11}^Z\D^2G(Z)\bigl(M(s,Z)\sigma'(Z,0)\bigr)^2 \d s=0
  \end{equation}
using Corollary~\ref{c:ptm=pt} and the bilinear property of $D^2G(Z)$ given in Proposition~\ref{p:d2ghk}.1.
  \subsubsection{First $\Q$-derivative of $\til P^1$}
  By definition of the solution to~\eqref{e:rotsys} through $\Q$ at $t=0$ we have
\begin{equation} \label{e:flow}
\dot\tvf^t(\Q,\beta) = G(\tvf^t(\Q,\beta)) + \beta \widetilde \LL(t,\tvf^t(\Q,\beta)) D.
\end{equation}
Differentiating with respect to $\beta$ at $\beta=0$ we obtain
\begin{equation}  \label{e:vf1dot}
\dot\tvf_1^t(\Q) = \D G(\tvf^t_0(\Q))\tvf_1^t(\Q) + \widetilde \LL(t,\tvf^t_0(\Q))D .
\end{equation}
% see  \eqref{e.yT} and \eqref{e:xieq}, 
Differentiating~\eqref{e:vf1dot} now with respect to~$\Q$ at $\Q=Z$ gives for $H\in V$
\begin{equation} \label{e:dvf1dot}
  \D\dot\tvf_1^t(Z)H=B^Z(\D\tvf_0^t(Z)H,\tvf_1^t(Z)) + A^Z\,\D\tvf_1^t(Z)H
     + \bigl(\D\widetilde \LL(t,Z)\D\tvf_0^t(Z) H\bigr) D
\end{equation}
with notation
\begin{equation}
 B^Z:=\D^2G(Z).
\end{equation}
and $A^Z=\D G(Z)$ as in~\eqref{e:AZdecomp}.
Now  $\til P^1(\Q)=\tvf_1^{T_0}(\Q)$ while $\D\tvf_0^t(Z)=e^{tA^Z}$ and $\tvf_1^t(Z)=y(t,Z)$ by Definition~\ref{d:defy}, so the variation of constants formula gives
\begin{equation} \label{e:dp1zh}
\D\til P^1(Z)H=\int_0^{T_0}e^{(T_0-s)A^Z}\Bigl(B^Z\bigl(e^{sA^Z}H,y(s,Z)\bigr)+\bigl(\D\widetilde \LL(s,Z) e^{sA^Z} H\bigr) D\Bigr)\d s .
\end{equation}
To evaluate the term involving $\D\til P_1$ in~\eqref{e:b2deriv} we must next substitute $H=\sigma'(Z,0)\in \kN^Z$ into~\eqref{e:dp1zh}.  We write
\begin{equation}  \label{e:pnpt}
p_T^Z=p_1^Z\,,\quad p_N^Z=p_0^Z+p_2^Z
  \end{equation}
to emphasise the tangent and normal character of these projections.
\begin{prop}
\begin{equation}  \label{e:sigprime}
  \sigma'(Z,0) = (\id_{\kN^Z}-e^{T_0A_N^Z})^{-1} y_N(T_0,Z)
  \end{equation}
where $A_N^Z:=p_N^ZA^Z|_{\kN^Z}$ {\rm(}that is the $\kN^Z$-block of $A^Z${\rm)} and  $y_N(t,Z):=p_N^Zy(t,Z)$ with $y(t,Z)$ as in Definition~\ref{d:defy}.
\end{prop}
\proof
Differentiating~\eqref{e:section} with respect to $\beta$ at $\beta=0$ yields
\begin{equation}
e^{T_0A_N^Z}\sigma'(Z,0) + P'(Z,0) = \sigma'(Z,0) \in \kN^Z 
\end{equation}
by \eqref{e:expA} and Proposition \ref{p:dpz}. This gives the result since $P'(Z,0)=\bigl(\tvf^{T_0}\bigr)'=y(T_0,Z)$ using Proposition \ref{p:tilP} for $i=1$.
\qed
\msk

\noindent Now substituting~\eqref{e:sigprime} for $H$ into~\eqref{e:dp1zh}  and again making use of Proposition~\ref{p:d2ghk}.1 gives  
\begin{align}
  \D&\til P^1_{11}(Z)\sigma'(Z,0)  
   =\int_0^{T_0}B^Z_{11}
  \bigl(e^{sA_N^Z}(\id_{\kN^Z}-e^{T_0A_N^Z})^{-1} y_N(T_0,Z),y_T(s,Z)\bigr)\d s  \notag \\
  & \qquad + \int_0^{T_0} p_{11}^Z\bigl(\D\widetilde \LL(s,Z) e^{sA_N^Z} (\id_{\kN^Z}-e^{T_0A_N^Z})^{-1} y_N(T_0,Z)\bigr) D\,\d s
  \label{e:longint1}
\end{align}
where $y_T(t,Z):=p_T^Zy(t,Z)$ and $B^Z_{11}:=p_{11}^ZB^Z$.
\msk

Finally, to complete the evaluation of~\eqref{e:b2deriv} we make explicit 
the term involving $\til P^2(Z)$ in that equation.
\subsubsection{The term $\til P^2(Z)$}
An expression for $\til P^2(Z)$ is obtained by differentiating~\eqref{e:flow} twice with respect to~$\beta$ at $(\Q,\beta)=(Z,0)$.  We find
\[
(\dot \tvf^t)'(\Q,\beta) = \D G(\tvf(\Q,\beta)) (\tvf^t)'(\Q,\beta) + \widetilde \LL(t, \tvf^t(\Q,\beta))D + \beta\bigl( \D \widetilde \LL(t, \tvf^t(\Q,\beta)) (\tvf^t)'(\Q,\beta)\bigr)D
\]
and so a second differentiation at $(\Q,\beta)=(Z,0)$ with $\tvf_1(Z) =\tvf'(Z,0)$ and $\tvf_2(Z) =\frac12 \tvf''(Z,\beta)|_{\beta=0}$ gives  
\[
 {2}\dot\tvf_2^t(Z)=B^Z\bigl(\tvf_1^t(Z)\bigr)^2 +  {2}A^Z\tvf^t_2(Z) + 2\bigl(\D\widetilde \LL(t,Z)\tvf^t_1(Z)\bigr)D.
\]
Since   $\til P^2(Z)=\tvf^{T_0}_2(Z)$  
 the variation of constants formula yields the expression 
\begin{equation*}
  \til P^2(Z)=\int_0^{T_0}M(T_0-s,Z)
     \Big(\big(\D\widetilde \LL(s,Z)y(s,Z)\big)D+   {\frac12} B^Z(y(s,Z))^2\Big)\d s
\end{equation*}
(cf. \cite[$f_2(z)$ on p.577]{LNT})
where $M(t,Z)=e^{tA_N^Z}$ and $y(t,Z)$ is as in \eqref{e:ytq}.
Then
\begin{equation}  \label{e:pkp2}
 \til P^2_{11}(Z)=\int_0^{T_0}
     p_{11}^Z\bigl(\D\widetilde \LL(t,Z)y(t,Z)\bigr)D\d t + \int_0^{T_0}B^Z_{11}(y_N(t,Z),y_T(t,Z))\d t
\end{equation}
from~Corollary~\ref{c:ptm=pt} and the bilinearity property~\eqref{e:b1}.
\msk
 
From~\eqref{e:b2deriv} with \eqref{e:dp0}, \eqref{e:d2p0} and Proposition \ref{p:dbpkzero} 
we therefore arrive at the following conclusion:
\begin{prop}  \label{p:biffn}
   We have
  \begin{equation}  \label{e:f2z}  
   P_{11}(Z+\sigma(Z,\beta),\beta) =\beta^2F_2(Z) + O(\beta^3)
  \end{equation}
  where 
  \begin{equation}\label{e:F_2}
F_2(Z)= \frac12 P_{11}''(Z,0) = p_{11}^Z\D\til P^1(Z)\sigma'(Z,0)+ \til P^2_{11}(Z)
  \end{equation}
with the terms on the right hand side given by the expressions~\eqref{e:longint1} and~\eqref{e:pkp2}.  \qed
  \end{prop}
Observe that \eqref{e:F_2} can be simplified using  \eqref{e:longint1},~\eqref{e:pkp2} and bilinearity of $B^Z$.  Denoting
\begin{equation}\label{e.chi}
\chi(t,Z):=e^{tA_N^Z}\sigma'(Z,0) + y(t,Z)
\end{equation}
and decomposing as usual $\chi=\chi_N+\chi_T$ with the obvious notation we can re-express~\eqref{e:F_2} as
\begin{equation}
  F_2(Z)=\int_0^{T_0}B^Z_{11}\bigl(\chi_N(t,Z),y_T(t,Z)\bigr)\d t
            + \int_0^{T_0} p_{11}^Z\big(\D\widetilde \LL(t,Z)\chi(t,Z)\big)D\d t.  \label{e:f2zterms}
\end{equation}
The bifurcation function $\F(\cdot,\beta)$ in~\eqref{e:bifdef} satisfies $\F'(Z,0)=0$ from Proposition~\ref{p:dbpkzero} and
\[
\F''(Z,0)E_{11}^Z =2F_2(Z)
\]
with $F_2(Z)$ given by~\eqref{e:f2zterms}.
\section{Explicit calculation of the bifurcation function} \label{s:explicit}
For explicit calculation of the second order term $\F''(Z,0)$ we now take $Z=Z(\theta,\phi)$ and express the bifurcation function~\eqref{e:bifdef} in terms of $\theta$ and~$\phi$.  The choice of $\phi$ is arbitrary so we expect the existence and stability results for periodic orbits to be independent of~$\phi$, but nevertheless we retain $\phi$ at this stage as a check on the calculations. 
\msk

Up to this point our analysis has assumed little more than the $\SO(3)$-equivariance (that is, frame-indifference) of the vector field~$G$ and the perturbation term~$\LL(\Q)D$ in the system~\eqref{e:sys1} and the fact that~$\Q^*$ is an equilibrium for the unperturbed ($\beta=0$) system.  To proceed further and evaluate $F_2(Z)$ we now need to make an explicit choice for the form of~$\LL(\Q)D$.
\subsection{Choices for the perturbing field $\LL(\Q)D$}  \label{s:choices}
We consider in turn the three terms comprising the field $\LL(\Q)D$ in~\eqref{e:lincombE}, that is
\begin{itemize}
\item[(i)]  $\LL^c(\Q)D=D$
\item[(ii)] $\LL^l(\Q)D 
  =[D,\Q\,]^+$
  \item[(iii)] $\LL^q\Q)(D)=\tr(D\Q)\Q$
\end{itemize}
where $D$ as in~\eqref{e:DWdefs} represents the symmetric part of the shear velocity gradient and we recall the notation~\eqref{e:HK+def}.  
From \eqref{e.R3Action} we obtain:
\begin{lemma}  \label{l:etilde}    
In the co-moving coordinate frame as in Section~\ref{s:comoving} the perturbation terms become respectively  
\begin{itemize}
\item[(i)] $\widetilde \LL^c(t,\Q)D:=\wtr_3(-\omega t) D = \sqrt 2 E_2(\tfrac\pi4-\omega t)$
\item[(ii)] $\widetilde \LL^l(t,\Q)D:=\wtr_3(-\omega t) [D,\wtr_3(\omega t)\Q\,]^+ =\sqrt 2 [ E_2(\tfrac\pi4-\omega t),\Q\,]^+$
  \item[(iii)] $\widetilde \LL^q(t,\Q)D:=\tr(D\,\wtr_3(\omega t)\Q)\Q
    =\tr(\wtr_3(-\omega t)D\,\Q)\Q  =\sqrt 2\,\tr(E_2(\tfrac\pi4-\omega t)Q)Q.$ 
\qed   
\end{itemize} 
\end{lemma}
\details{$\tr(D\tilde R\Q)=\tr(DR\Q R^T)=\tr(R^TDR\Q)$} 
Taking the derivative with respect to the $\Q$ variable we obtain
\begin{prop}  \label{p:dtile}
In the respective cases (i),(ii),(iii) for $\Q,H\in V$
\begin{itemize}
\item[(i)] $\D\widetilde \LL^c(t,\Q)D=0$ 
\item[(ii)] $\big(\D\widetilde \LL^l(t,\Q)H\big)D=\sqrt2\,[E_2(\tfrac\pi4-\omega t),H]^+$
  \item[(iii)] $\big(\D\widetilde \LL^q(t,\Q)H\big)D =\sqrt2\, \tr(E_2(\tfrac\pi4-\omega t)H)\Q
      + \sqrt2\,\tr(E_2(\tfrac\pi4-\omega t)\Q)H.$  \qed
\end{itemize}
\end{prop}
We next need expressions for the components of $E_2(\pi/4-\omega t)$ in the basis $\B^Z$ as in~\eqref{e:bzbasis}. These could formally be found using $5\times5$ Wigner rotation matrices describing the action of $\SO(3)$ on~$V$ as in physics texts such as~\cite{Rose}, but in our case it will be simpler to calculate directly.
\subsection{Expression of $E_2(\pi/4-\omega t)$ in the vector basis $\B^Z$.}  \label{s:tqexpr}
For any $E_2(u)$ and any $Z=Z(\theta,\phi)\in\kO$ and $\Q\in V$ we have from \eqref{e.R3Action}
\begin{align}
    \langle E_2(u),\wtr_\z\Q\rangle &= \langle E_2(u),\wtr_3(\vf)\wtr_2(\theta)Q\rangle 
  =\langle\wtr_3(-\varphi) E_2(u),\wtr_2(\theta)Q\rangle \notag \\
  &=\langle E_2(u-\varphi),\wtr_2(\theta)Q\rangle. \label{e:inner}
\end{align}
Calculating $\wtr_2(\theta)\Q$ for $\Q=E_0,E_1(\al),E_2(\al)$ in turn we find by elementary matrix multiplication 
\begin{equation}  \label{e:rq*}
\wtr_2(\theta)E_0= \frac{1}{\sqrt 6}
\begin{pmatrix}
  2\sin^2\theta-\cos^2\theta & 0 & 3\sin\theta\cos\theta \\
  0 & -1 & 0 \\
  3\sin\theta\cos\theta & 0 & 2\cos^2\theta-\sin^2\theta
\end{pmatrix}
\end{equation}
\details{\fbox{CW reincluded and corrected sign notation for rotations, please do not remove:}
\begin{align*}
\wtr_2(\theta)Q^* & = a \left(\begin{array}{ccc} c& 0 & s\\
0 & 1 & 0 \\
-s & 0 & c \end{array} \right){\rm diag}(-1, -1, 2) \left(\begin{array}{ccc} c& 0 & -s\\
0 & 1 & 0 \\
s & 0 & c \end{array} \right)\\
& = a \left(\begin{array}{ccc} c& 0 & s\\
0 & 1 & 0 \\
-s & 0 & c \end{array} \right) \left(\begin{array}{ccc} -c& 0 & s\\
0 & -1 & 0 \\
2s & 0 & 2c \end{array} \right) = \left(\begin{array}{ccc} 2s^2-c^2 &  0 &3cs \\
 0 &-1  & 0  \\
 3cs &  0 &  2c^2-s^2\end{array} \right) 
\end{align*}
}
while
\begin{equation}  \label{e:rk}
\wtr_2(\theta)E_1(\al)= \frac{1}{\sqrt 2}
\begin{pmatrix}
  \cos\al\sin2\theta & \sin\al\sin\theta & \cos\al\cos2\theta \\
    \sin\al\sin\theta & 0 & \sin\al\cos\theta \\
 \cos\al\cos2\theta & \sin\al\cos\theta & -\cos\al\sin2\theta
\end{pmatrix}
\end{equation}
\details{\fbox{CW using now correct sign notation for rotations:} since
\begin{align*}
\wtr_2(\theta)K(\al) &=\sqrt3 a \left(\begin{array}{ccc}c &0&s\\ 
0 &1&0\\
 -s &0&c\end{array} \right) 
 \left(\begin{array}{ccc} 0& 0&c_a\\ 
0 &0&s_a\\
 c_a &s_a&0\end{array} \right) 
 \left(\begin{array}{ccc} c&0&-s\\ 
 0&1&0\\
 s &0&c\end{array} \right)\\
 & = 
 \sqrt3 a \left(\begin{array}{ccc}sc_a & ss_a& cc_a \\ 
0  &0 &s_a \\
 cc_a  &cs_a & -sc_a\end{array} \right)
  \left(\begin{array}{ccc} c&0&-s\\ 
 0&1&0\\
 s &0&c\end{array} \right)\\
 & =  \sqrt3 a\left(\begin{array}{ccc} 2scc_a &ss_a &-s^2 c_a + c^2 c_a \\ 
ss_a  &0 & s_a c\\
 -s^2 c_a + c^2 c_a  & s_a c& -2scc_a\end{array} \right)
\end{align*}
}
and
\begin{equation}  \label{e:rtqa}
\wtr_2(\theta)E_2(\al)= \frac{1}{\sqrt 2}
\begin{pmatrix}
  \cos2\al\cos^2\theta & \sin2\al\cos\theta & -\cos2\al\sin\theta\cos\theta \\
  \sin2\al\cos\theta & -\cos2\al & -\sin2\al\sin\theta \\
  -\cos2\al\sin\theta\cos\theta & -\sin2\al\sin\theta & \cos2\al\sin^2\theta
\end{pmatrix}.
\end{equation}
\details{\fbox{CW used correct sign notation for rotations now:}
since
\begin{align*}
\wtr_2(\theta)\sqrt3\tqa &= \sqrt3 a \left(\begin{array}{ccc}c &0&s\\ 
0 &1&0\\
 -s &0&c\end{array} \right) \left(\begin{array}{ccc} c_2& s_2&0\\ 
s_2 &-c_2&0\\
 0 &0&0\end{array} \right) \left(\begin{array}{ccc} c&0&-s\\ 
 0&1&0\\
 s &0&c\end{array} \right)\\
 & =  \sqrt3 a \left(\begin{array}{ccc} 
c c_2 &cs_2&0\\
s_2  &-c_2&0\\
 -sc_2 &-ss_2&0\end{array} \right) \left(\begin{array}{ccc} c&0&-s\\ 
 0&1&0\\
 s &0&c\end{array} \right) = 
 \sqrt3 a \left(\begin{array}{ccc}c^2 c_2  &cs_2&-scc_2\\ 
 s_2 c&-c_2&- ss_2\\
- scc_2&-ss_2&s^2 c_2\end{array} \right)
\end{align*}
}

Using~\eqref{e:inner} and elementary computation we obtain the following results needed to compute the coefficients of $\widetilde \LL(t,\Q)D$ in the basis $\kB^Z$ at $Z\in \kO$.
\begin{prop}  \label{p:tqcoeffs}
  \begin{equation}\label{10.57}
\langle  E_2(u),E_0^Z\rangle=\frac{\sqrt{3}}2\cos2(u-\varphi)\sin^2\theta
  \end{equation}
  while   
  \begin{equation}
  \langle E_2(u), E_1^Z(\al)\rangle=\frac1{ 2}\cos2(u-\varphi)\cos\al\sin2\theta 
  +\sin2(u-\varphi)\sin\al\sin\theta  \label{10.59}
  \end{equation}
  and
\begin{align} \langle E_2(u),E_2^Z(\al)\rangle&=\frac1{ 2}\cos2(u-\varphi)\cos2\al(1+\cos^2\theta) \notag \\
  &\qquad+ \sin2(u-\varphi)\sin2\al\cos\theta. \label{10.58}
\end{align}  \qed
\end{prop}

Using Proposition~\ref{p:tqcoeffs} we see that $E_2(\pi/4-\omega t)$  is expressed in terms of the orthonormal basis $\B^Z$ as follows: 
\begin{cor}  
 \label{c:tqexpr}
     \begin{equation}  \label{e:tqlist} E_2(\tfrac\pi4-\omega t) =c_{01}^ZE_0^Z+c_{11}^ZE_{11}^Z+c_{12}^ZE_{12}^Z+c_{21}^ZE_{21}^Z+c_{22}^ZE_{22}^Z
    \end{equation}
    where the coefficients $c_{01}^Z$ etc. depending on $(t,\theta,\phi)$ are given by
    \begin{align*}
      c_{01}^Z&=c_{01}(\theta)\sin(2\omega t+2\phi)  \notag \\
      c_{11}^Z&=c_{11}(\theta)\sin(2\omega t+2\phi)  \notag \\
      c_{12}^Z&=c_{12}(\theta)\cos(2\omega t+2\phi)   \notag \\
      c_{21}^Z&=c_{21}(\theta)\sin(2\omega t+2\phi)  \notag \\
      c_{22}^Z&=c_{22}(\theta)\cos(2\omega t+2\phi)   \notag     
    \end{align*}
   and  where 
    \begin{align}      \bigl(c_{01}(\theta),c_{11}(\theta),c_{12}(\theta),\,&c_{21}(\theta),c_{22}(\theta)\bigr)\notag \\
      &=\frac1{2}\bigl(\sqrt3\sin^2\theta, \sin2\theta,\, 2\sin\theta,\,
 (1+\cos^2\theta),\,2\cos\theta\,\bigr).  \label{e:mdefs}
    \end{align}
    \qed
\end{cor}
\subsection{Calculation of $y(t,Z)$}  \label{s:ycalc}
Armed with these coefficients we are now in a position to calculate~$y(t,Z)$
and subsequently $\chi(t,Z)$, needed in order to evaluate~\eqref{e:f2zterms}.
We consider in turn the three cases (i),(ii) and (iii) of Section~\ref{s:choices}, denoting the corresponding $y$ by $y^c,y^l,y^q$ respectively.
\msk

\paragraph{{\bf Case (i):} $\widetilde \LL(t,\Q)D=\widetilde \LL^c(t,\Q)D= \sqrt 2 E_2(\tfrac\pi4-\omega t)$}{}\hfill
\msk

From~\eqref{e:ytq} and using~\eqref{e:expA} we have 
  \begin{align}
    y^c(t,Z)&=\sqrt2 \int_0^te^{\lam(t-s)}\,p_0^ZE_2(\tfrac\pi4-\omega s)\d s   +\sqrt2\int_0^tp_1^Z E_2(\tfrac\pi4-\omega s)\d s \\
    & \quad +\sqrt2 \int_0^te^{\mu(t-s)}p_2^ZE_2(\tfrac\pi4-\omega s)\d s.
    \label{e.ytZDecomp}
  \end{align}
  For convenience we now introduce the polar coordinate notation
\begin{equation} \label{e:triangle}
(\nu,2\omega)= r_\nu(\cos2\gamma_\nu,\sin2\gamma_\nu)
\end{equation}
for $\nu=\lam,\mu$, as well as the abbreviations  
 \begin{align}
   S(t,\phi,\nu):&=\int_0^te^{\nu(t-s)}\sin(2\omega s+2\phi)\d s\notag \\
           &=\frac1{r_\nu}
   \bigl(e^{\nu t}\sin(2\phi+2\gamma_\nu)-\sin(2\omega t+2\phi+2\gamma_\nu)\bigr)  \label{e:stfn} \\
   C(t,\phi,\nu):&=\int_0^te^{\nu(t-s)}\cos(2\omega s+2\phi)\d s \notag \\
           &=\frac1{r_\nu}
   \bigl(e^{\nu t}\cos(2\phi+2\gamma_\nu)-\cos(2\omega t+2\phi+2\gamma_\nu)\bigr) \label{e:ctfn}
 \end{align}
  with the limiting cases
 \begin{align}
   S(t,\phi,0)&=\frac1{2\omega}(\cos2\phi-\cos(2\omega t+2\phi))  \label{e:stfz} \\
   C(t,\phi,0)&=\frac1{2\omega}(\sin(2\omega t+2\phi)-\sin2\phi).   \label{e:ctfz}
 \end{align}
The cases when $t=T_0=2\pi/\omega$ will also be important:
  \begin{align}
   S(T_0,\phi,\nu)&=\frac1{r_\nu}(e^{\nu T_0}-1)\sin(2\phi+2\gamma_\nu) \label{e:st0fnu} \\
   C(T_0,\phi,\nu)&=\frac1{r_\nu}(e^{\nu T_0}-1)\cos(2\phi+2\gamma_\nu). \label{e:ct0fnu}
 \end{align}
Using these we obtain from Corollary~\ref{c:tqexpr} the following expression for $y^c(t,Z)$ in terms of the basis~$\B^Z$.
 \begin{prop}  \label{p:ytzexpr1}   
We have
    \begin{equation}  \label{e:ylist} y^c(t,Z)=y_{01}^cE_0^Z+y_{11}^cE_{11}^Z+y_{12}^cE_{12}^Z+y_{21}^cE_{21}^Z+
      y_{22}^c E_{22}^Z 
    \end{equation}
    where 
      \begin{align*}
        y_{01}^c&= \sqrt 2c_{01}(\theta)S(t,\phi,\lam) \\
        y_{11}^c&=\sqrt 2c_{11}(\theta)S(t,\phi,0)     \\
        y_{12}^c&=\sqrt 2c_{12}(\theta)C(t,\phi,0)  \\ 
        y_{21}^c&=\sqrt 2c_{21}(\theta)S(t,\phi,\mu)  \\
        y_{22}^c&=\sqrt 2c_{22}(\theta)C(t,\phi,\mu) .         
      \end{align*}
 \end{prop}
 \begin{proof}
 By Corollary \ref{c:tqexpr} and \eqref{e.ytZDecomp}
 \begin{align*}
 y_{01}^c & =  \langle y^c(t,Z), E_0^Z\rangle
 = \sqrt2 \int_0^t e^{\lambda (t-s)} \langle E_2(\tfrac\pi4-\omega s), E_0^Z \rangle  \d s\\
 & 
  =  \sqrt2\int_0^t e^{\lambda (t-s)} c_{01}(\theta) \sin(2\omega s+ 2\phi) \d s =\sqrt 2 c_{01}(\theta) S(t,\phi,\lambda)
 \end{align*}
 and
 \begin{align*}
 y_{11}^c & =  \langle y^c(t,Z),E_{11}^Z \rangle
 = \sqrt2 \int_0^t \langle E_2(\tfrac\pi4-\omega s), E_{11}^Z \rangle  \d s
 \\
 & = \sqrt2\int_0^t c_{11}(\theta) \sin(2\omega s +2\phi) \d s = \sqrt2 c_{11}(\theta) S(t,\phi,0).
 \end{align*}
 The calculations for $y_{12}^c,y_{21}^c$ and $y_{22}^c$ are very similar.
 \qed
 \details{ \fbox{CW reincluded computations. Please do not remove.}
  \begin{align*}
 y_{22}^c & =\frac1{6a^2} \langle y(t,Z), \sqrt 3 \tilde\Q_Z(\pi/4) \rangle
 = \int_0^t e^{\mu (t-s)} \langle \sqrt3\tilde\Q_Z(\frac\pi4-\omega s), \sqrt 3 \tilde\Q_Z(\pi/4) \rangle  \d s
 \\
 & = \int_0^t e^{\mu (t-s)} c_{22}(\theta) \cos(2\omega s +2\phi) \d s = c_{22}(\theta) C(t,\phi,\mu).
 \end{align*}
  and
    \begin{align*}
 y_{21}^c & =\frac1{6a^2} \langle y(t,Z), \til\Q(0) \rangle
 = \int_0^t e^{\mu (t-s)} \langle \sqrt3\tilde\Q_Z(\frac\pi4-\omega s),  \til\Q(0)\rangle  \d s
 \\
 & = \int_0^t  e^{\mu (t-s)} c_{21}(\theta) \sin(2\omega s +2\phi) \d s = c_{21}(\theta) S(t,\phi,\mu).
 \end{align*}
 and
   \begin{align*}
 y_{12}^c& =\frac1{6a^2} \langle y(t,Z), K^Z(\pi/2) \rangle
 = \int_0^t   \langle \sqrt3\tilde\Q_Z(\frac\pi4-\omega s),   \tilde K^Z(\pi/2)\rangle  \d s
 \\
 & = \int_0^t  c_{12}(\theta) \cos(2\omega s +2\phi) \d s = c_{12}(\theta) C(t,\phi,0).
 \end{align*}
}
  \end{proof}  
 \msk
 
\paragraph{{\bf Case (ii):} $\widetilde \LL (t,\Q)D=\widetilde \LL^l(t,\Q)D=\sqrt2 [E_2(\tfrac\pi4-\omega t),\Q\,]^+$}{}\hfill   
\msk

\begin{align}
  \widetilde \LL^l(t,Z)D&=\sqrt 2[E_2(\tfrac\pi4-\omega t),Z\,]^+ \notag \\
 &= \sqrt 2( 2ac_{01}^Z E_0^Z+ac_{11}^Z E_{11}^Z +ac_{12}^ZE_{12}^Z-2ac_{21}^ZE_{21} ^Z-2ac_{12}^ZE_{22}^Z )
\label{e:Ecase2}
\end{align}
since Proposition~\ref{p:Leigs} shows that $\widetilde \LL^l(t,Z)D$ differs from $\widetilde \LL^c(t,Z)D$ only in that the coefficients of $E_0^Z,E_1^Z(\al),E_2^Z(\al)$ are multiplied by $2a,a,-2a$ respectively. Hence in this case the result is the following. 
\begin{prop}  \label{p:ylexpr}
The components of $y^l$ are given by
  \begin{equation}  \label{e:y2list}
    (y_0^l(t,Z),y_1^l(t,Z),y_2^l(t,Z))=(2a\,y_0^c(t,Z), a\,y_1^c(t,Z), - 2a\,y_2^c(t,Z))
  \end{equation}
  where $y_i$ denotes $p_i^Zy,\, i=0,1,2$.  \qed
\end{prop}  
\paragraph{{\bf Case (iii):} $\widetilde \LL(t,\Q)D=\widetilde \LL^q(t,\Q)D=  \sqrt2 \tr(E_2(\tfrac\pi4-\omega t)Q)\,Q$} \hfill
\msk

In this case
\begin{equation}  \label{e:tilEq}
\widetilde \LL^q(t,Z)D = \sqrt2 \tr(E_2(\tfrac\pi4-\omega t)Z)Z = 
3\sqrt6 a^2  \sin(2 \omega t +2\phi)\sin^2\theta   E_0^Z
\end{equation}
using~\eqref{10.57}, and so
\details{
with $u = \frac{\pi}4-\omega t$ we have
 $\cos 2(u-\phi) = \cos( \frac\pi2-2 \omega t -2\phi) = \cos ( -\frac\pi2+2 \omega t +2\phi)
 = \sin(2 \omega t +2\phi) $
and so from (6.10)
$tr(\tq(\frac{\pi}4-\omega t) Z) =
 3 \sqrt 3 a^2 \sin(2 \omega t +2\phi)\sin^2\theta$.
}
\begin{equation}   \label{e:y3list}
    y^q(t,Z) = y_0^q = y_{01}^q E_0^Z
    \end{equation}
where
\begin{align}  y_{01}^q &= 3\sqrt6\,a^2 \sin^2\theta\int_0^t e^{\lam(t-s)}\sin(2\omega s + 2\phi)\d s \\&= 3\sqrt6\,a^2   \sin^2\theta\,S(t,\phi,\lam)  
  =  6a^2\,y_{01}^c \label{e:yhat} 
  \end{align}
from~\eqref{e:ylist}.  Thus 
\begin{equation}   \label{e:yqexpr}
 (y_0^q(t,Z),y_1^q(t,Z),y_2^q(t,Z))  = 6 a^2  (y_0^c(t,Z),0,0).
\end{equation}
\subsection{Calculation of $\chi(t,Z)$} \label{s:bifc1} 
From~\eqref{e:sigprime} and~\eqref{e.chi} we have
\begin{equation}  \label{e:chiexpr}
\chi(t,Z)=e^{tA_N^Z}(\id_{\kN^Z}-e^{T_0A_N^Z})^{-1} y_N(T_0,Z)+ y(t,Z)
\end{equation}
where we recall $y_N=y_0+y_2$.
Again we consider in turn the cases~(i),(ii) and~(iii), using respective notation $\chi^c,\chi^l,\chi^q$.
\msk

\paragraph{{\bf Case (i):} $\widetilde \LL^c(t,\Q)D= \sqrt 2 E_2(\pi/4-\omega t)$}{}\hfill
\msk

Here
 \begin{equation}\label{e.chiZGen}
  \chi_0^c(t,Z) = e^{\lam t}(1-e^{\lam T_0})^{-1} y_0^c(T_0,Z) + y_0^c(t,Z)
\end{equation}
and using Proposition~\ref{p:ytzexpr1} with~\eqref{e:stfn} and~\eqref{e:st0fnu} we find  
 \details{
   $ \chi_Z(t) = e^{\lam t}(1-e^{\lam T_0})^{-1} S(T_0,\phi,\lambda) c_{01}(\theta) 
   + S(t,\phi,\lambda) c_{01}(\theta) $
   }
 \begin{align*}
  \chi_0^c(t,Z)  & = 
 -\sqrt2\,e^{\lam t}c_{01}(\theta)\frac1{r_\lam}\sin(2\phi+2\gamma_\lam)\,E_0^Z
 \\
 & \quad +\sqrt2\,c_{01}(\theta)\frac1{r_\lam} \bigl(e^{\lam t}\sin(2\phi+2\gamma_\lam)-\sin(2\omega t+2\phi+2\gamma_\lam)\bigr)\,E_0^Z
\end{align*}
giving
\begin{equation} \label{e:chiz}
 \chi_0^c(t,Z) =-\sqrt2\,c_{01}(\theta)\frac1{r_\lam}\sin(2\omega t+2\phi+2\gamma_\lam)\,E_0^Z.
\end{equation}
Likewise 
\begin{equation}  \label{e:chiu}
  \chi_2^c(t,Z) = e^{\mu t}(1-e^{\mu T_0})^{-1} y_2^c(T_0,Z) + y_2^c(t,Z)
\end{equation}
which gives using~\eqref{e:ctfn} and~\eqref{e:ct0fnu} as well
\details{using Proposition~\ref{p:ytzexpr1} and \eqref{e:stfn}, \eqref{e:ctfn}, \eqref{e:st0fnu} and \eqref{e:ct0fnu},
\begin{align*}
\chi_U(t,Z) &=  e^{\mu t}(1-e^{\mu T_0})^{-1}\left(  c_{21}(\theta) S(T_0,\phi,\mu)\sqrt 3 \tilde\Q_Z(0)+c_{22}(\theta) C(T_0,\phi,\mu)\sqrt 3 \tilde\Q_Z(\pi/4)\right)\\
& \quad +  c_{21}(\theta) S(t,\phi,\mu)\sqrt 3 \tilde\Q_Z(0)+c_{22}(\theta) C(t,\phi,\mu)\sqrt 3 \tilde\Q_Z(\pi/4)\\
& = - e^{\mu t} \frac{ c_{21}(\theta)  }{r_\mu}\sin(2\phi+2r_\mu)\sqrt 3 \tilde\Q_Z(0) - e^{\mu t} \frac{ c_{22}(\theta)  }{r_\mu}\cos(2\phi+2r_\mu) \sqrt 3 \tilde\Q_Z(\pi/4)
\\ & \quad +\frac{ c_{21}(\theta)  }{r_\mu}\left( \e^{\mu t} \sin(2\phi+2r_\mu) -  \sin(2\omega t + 2\phi+2r_\mu)\right) \sqrt 3 \tilde\Q_Z(0) \\
& \quad +  \frac{ c_{22}(\theta)  }{r_\mu}\left(\e^{\mu t} \cos(2\phi+2r_\mu)  -\cos(2\omega t+ 2\phi+2r_\mu)  \right) \sqrt 3 \tilde\Q_Z(\pi/4)
\end{align*}
}
\begin{align}
  \chi_2^c(t,Z)&=-\sqrt2\,c_{21}(\theta)\frac1{r_\mu}\sin(2\omega t+2\phi+2\gamma_\mu)E_{21}^Z  \notag \\
 &\qquad\quad - \sqrt2\,c_{22}(\theta)\frac1{r_\mu}\cos(2\omega t+2\phi+2\gamma_\mu)E_{22}^Z \label{e:chiUc}
\end{align}
while the definition~\eqref{e.chi} gives
\begin{equation}
  \chi_1^c(t,Z) = y_1^c(t,Z)
 = \sqrt2\,c_{11}(\theta)S(t,\phi,0)  E_{11}^Z + \sqrt2\,c_{12}(\theta)C(t,\phi,0) E_{12}^Z \label{e:chiTc}
\end{equation}
from~Proposition~\ref{p:ytzexpr1}.
\msk

\paragraph{{\bf Case (ii):} $\widetilde \LL^l(t,\Q)D=\sqrt2[E_2(\pi/4-\omega t),\Q\,]^+$}{}\hfill   
\msk

Since $\chi$ is linear in~$y$  (see~\eqref{e.chiZGen}) we immediately deduce from Proposition~\ref{p:ylexpr} the relations
\begin{equation}  \label{e:tilchi}
  (\chi_0^l(t,Z),\chi_1^l(t,Z),\chi_2^l(t,Z))=
  (2a\chi_0^c(t,Z),a\chi_1^c(t,Z),-2a\chi_2^c(t,Z)).
\end{equation}
\msk

\paragraph{{\bf Case (iii):} $\widetilde \LL^q(t,\Q)D=\sqrt2\tr(E_2(\pi/4-\omega t)Q)\,Q$} \hfill
\msk

Again since $\chi$ is linear in~$y$ it follows from~\eqref{e:yqexpr} that 
\begin{equation}    \label{e:chihat}
  (\chi_0^q(t,Z),\chi_1^q(t,Z),\chi_2^q(t,Z))
  = (6a^2\chi_0^c(t,Z),0,0).
\end{equation}
\subsection{The bifurcation function}  \label{s:biffn}
We are now ready to calculate the terms appearing in the expression~\eqref{e:f2zterms} that determine the bifurcation function. With $y_T=y_1$ and $\chi_N=\chi_0+\chi_2$
the first term is
\begin{align}
  \int_0^{T_0}B^Z_{11}\bigl(\chi_N&(t,Z),y_1(t,Z)\bigr)\d t  \notag \\
  &=\int_0^{T_0}B^Z_{11}\bigl(\chi_0(t,Z),y_1(t,Z)\bigr)\d t
      + \int_0^{T_0}B^Z_{11}\bigl(\chi_2(t,Z),y_1(t,Z)\bigr)\d t. \label{e:chiN}
\end{align}
We evaluate this initially for $\chi_N^c,y_1^c$ and then use~\eqref{e:y2list},\eqref{e:yqexpr},\eqref{e:tilchi} and~\eqref{e:chihat} to evaluate~\eqref{e:chiN} with
\begin{align}
  y_1 &=  m_cy_1^c +  m_ly_1^l +   m_qy_1^q   \quad\text{and}  \\
  \chi_i &=  m_c\chi_i^c +  m_l\chi_i^l +   m_q\chi_i^q \quad\text{for\ } i=0,2
\end{align}
using the bilinearity of~$B^Z$. Substituting $\chi_0^c(t,Z)$
from~\eqref{e:chiz} and using Proposition~\ref{p:ytzexpr1} and Corollary~\ref{c:hkform2} we find
\begin{align}
\int_0^{T_0}& B^Z_{11}\bigl(\chi_0^c(t,Z),\,y_1^c(t,Z)\bigr)\d t  \notag
\\ &=
 - \frac{\sqrt2}{r_\lam} c_{01}(\theta)\,B^Z_{11}\big(E_0^Z,\int_0^{T_0}\sin(2\omega t+2\phi+2\gam_\lam)y_1^c(t,Z)\d t \big) 
\notag  \\
 & = 
- \frac{1}{a\sqrt3r_\lam} c_{01}(\theta)\lambda \int_0^{T_0}y_{11}^c\sin(2\omega t+2\phi+2\gam_\lam)  \d t  \,  E_{11}^Z  \notag \\
& = - \frac{\sqrt2}{a\sqrt3r_\lam} c_{01}(\theta)\lambda c_{11}(\theta) \int_0^{T_0} \sin(2\omega t+2\phi+2\gam_\lam)  S(t,\phi,0) \d t \,  E_{11}^Z \notag \\
&=  \frac{T_0}{\sqrt6 a }\tau_\lam c_{01}(\theta)c_{11}(\theta) \,  E_{11}^Z
\label{e.p_K^ZBchiZy}
\end{align}
since by~\eqref{e:stfz}
\begin{align*}
  \int_0^{T_0}\sin(2\omega t+2\phi+2\gam_\lam) S(t,\phi,0)\d t
  &=-\frac1{2\omega}\int_0^{T_0}\sin(2\omega t+2\phi+2\gam_\lam) \cos(2\omega t+ 2\phi) \d t \notag \\
  &=-\frac1{4\omega}\int_0^{T_0}\bigl(\sin (4\omega t +4\phi +2\gamma_\lambda)+ \sin2\gamma_\lam\bigr) \d t
\end{align*}
\details{since
\[
\sin A \cos B = \frac12(\sin(A+B) + \sin(A-B))
\]
}
and $\sin2\gamma_\lam=2\omega/r_\lam$ as in~\eqref{e:triangle}. Here we have introduced the notation 
\[
\tau_\nu:= \nu/r_\nu^2
\]
for $\nu=\lam,\mu$.
\msk

Likewise from \eqref{e:chiUc} with Proposition~\ref{p:ytzexpr1} and Corollary~\ref{c:hkform2} we have
\begin{align}
\int_0^{T_0}&B^Z_{11}\bigl(\chi_2^c(t,Z),\,y_1^c(t,Z)\bigr)\d t  \notag\\
& =
 - \frac{\sqrt2}{r_\mu} c_{21}(\theta)\,B^Z_{11}\big(E_{21}^Z,\int_0^{T_0}\sin(2\omega t+2\phi+2\gam_\mu)y_1^c(t,Z)\d t \big) 
 \notag  \\
 &\qquad - \frac{\sqrt2}{r_\mu} c_{22}(\theta)\,B^Z_{11}\big(E_{22}^Z,\int_0^{T_0}\cos(2\omega t+2\phi+2\gam_\mu)y_1^c(t,Z)\d t \big)\notag  \\
 & = \frac{\sqrt2\mu}{3 ar_\mu} c_{21}(\theta)c_{11}(\theta) 
 \int_0^{T_0}\sin(2\omega t+2\phi+2\gam_\mu) S(t,\phi,0) \d t  \,E_{11}^Z\notag \\
 &\qquad + \frac{\sqrt2\mu}{3 a r_\mu} c_{22}(\theta)c_{12}(\theta) \int_0^{T_0}\cos(2\omega t+2\phi+2\gam_\mu) C(t,\phi,0) \d t \,E_{11}^Z \notag \\
&= - \frac{T_0} {3\sqrt2 a }\tau_\mu\bigl(c_{21}(\theta)c_{11}(\theta) + c_{22}(\theta)c_{12}(\theta)\bigr)\,E_{11}^Z 
\label{e.p_K^ZBchiUy}
\end{align}
using
\begin{align*}
  \int_0^{T_0}\cos(2\omega t+2\phi+2\gam_\mu) C(t,\phi,0) \d t
  &=\frac1{2\omega}\int_0^{T_0}\cos(2\omega t+2\phi+2\gam_\mu)\sin(2\omega t+2\phi) \d t \notag \\
  &=\frac1{4\omega}\int_0^{T_0}\bigl(\sin (4\omega t +4\phi +2\gamma_\lambda)-\sin2\gamma_\lam\bigr)\d t.
\end{align*}
\details{where we use  
\[
\sin A \cos B = \frac12(\sin(A+B) + \sin(A-B))
\]
}
Observe that as anticipated the expressions~\eqref{e.p_K^ZBchiZy} and~\eqref{e.p_K^ZBchiUy} do not depend on the meridianal angle~$\phi$.
\msk

We now turn to the second term appearing in the expression~\eqref{e:f2zterms} for the bifurcation function, namely
\begin{equation}
\int_0^{T_0} p_{11}^Z\big(\D\widetilde \LL(t,Z)\chi(t,Z)\big)D\,\d t.  \label{e:dechi}
\end{equation}
For Case~(i) with $\LL=\LL^c$ of course $\D \LL^c=0$ and so we focus on Case~(ii) with $\LL=\LL^l$.
Proposition \ref{p:dtile} (ii) and Corollary~\ref{c:tqexpr} together with
Proposition \ref{p.[H,K]^+_1} give
\begin{align}
  p_{11}^Z\big(\D&\widetilde \LL^l(t,Z)\chi(t,Z)\big)D=\sqrt2 p_{11}^Z[ E_2(\tfrac\pi4-\omega t),\chi(t,Z)]^+ \notag \\
  &= \left(\frac{1}{\sqrt3}(\chi_{01}c_{11}^Z + \chi_{11}c_{01}^Z) +  (\chi_{11}c_{21}^Z + \chi_{21}c_{11}^Z + \chi_{22}c_{12}^Z +  \chi_{12}c_{22}^Z)\right)E_1^Z(0)  \label{e:determ}
\end{align}
where $\chi_{01}=\chi_{01}(t,Z)$ etc. denote the coefficients of $\chi(t,Z)$ in the basis~$\B^Z$.   We now take $\chi=\chi^c$ and
evaluate~\eqref{e:dechi} by integrating~\eqref{e:determ} from $t=0$ to $t=T_0$.  Straightforward trigonometrical integrals using~\eqref{e:chiz} and  Corollary~\ref{c:tqexpr} give
\begin{align}
  \int_0^{T_0}\chi_{01}^c c_{11}^Z\,\d t & =\int_0^{T_0} (-\frac{\sqrt2}{r_\lambda} c_{01}(\theta) \sin(2\omega t + 2\phi + 2 \gamma_\lambda))
  ( c_{11}(\theta) \sin(2\omega t + 2\phi)) \d t  \notag
  \\
   & =- \frac{T_0}{\sqrt2}\tau_\lam c_{01}(\theta) c_{11}(\theta)  \label{e:intmhd}
  \end{align}
since $\lam = r_\lam\cos 2\gamma_\lam$.   
\details{CW reincluded, please do not remove:
    \begin{align*}%\label{e.sin^2Int}
 \int_0^{T_0}\sin(2\omega t + 2 \phi + 2 \gamma_\lambda) \sin(2\omega t + 2 \phi) \d t   =     \cos(2 \gamma_\lambda) \int_0^{T_0}\sin^2(2\omega t + 2 \phi) \d t    =   \cos(2 \gamma_\lambda) \frac{T_0}{2}
 = \frac{\lambda T_0}{2r_\lambda}   \end{align*}
    using that $\cos(2x) = \cos^2 x-\sin^2 x = 1-2\sin^2 x$.
    {Easier: $2\sin A\sin B=\cos(A-B)-\cos(A+B)$.}}
Moreover, by \eqref{e:chiTc} and Corollary~\ref{c:tqexpr}
\begin{equation} \label{e:intdhm}
  \int_0^{T_0}\chi_{11}^c c_{01}^Z \d t  = \sqrt2 c_{11}(\theta) c_{01}(\theta) \int_0^{T_0}  S(t,\phi,0)\sin(2\omega t +2\phi) \d t =0
\end{equation}
and also
\begin{equation}  \label{e:intdhb}
  \int_0^{T_0}\chi_{11}^c c_{21}^Z\,\d t =  \sqrt2 c_{11}(\theta)\,c_{21}(\theta) \int_0^{T_0} S(t,\phi,0) \sin(2\omega t + 2\phi)\d t =  0,
  \end{equation}
while from~\eqref{e:chiUc}
\begin{equation}  \label{e:intbhd}
\int_0^{T_0}\chi_{21}^c c_{11}^Z\,\d t  = - \frac{T_0}{\sqrt 2} \tau_\mu c_{21}(\theta) c_{11}(\theta).
 \end{equation}
\details{CW added, please do not remove, changed. 
\begin{align*} 
\int_0^{T_0} \chi_{21}^c c_{11}^Z\,\d t  &=\int_0^{T_0}
- \frac{\sqrt2}{r_\mu}c_{21}(\theta)\sin(2\omega t+ 2\phi + 2\gamma_\mu)   c_{11}(\theta) \sin(2\omega t + 2\phi) \d t  =-\frac{\sqrt2}{r_\mu}c_{21}(\theta)c_{11}(\theta)\frac{\mu T_0}{2r_\mu}
\end{align*}
}
Similarly we find 
\begin{equation}  \label{e:intsim1}
\int_0^{T_0}\chi_{22}^c c_{12}^Z\,\d t  = - \frac{T_0}{\sqrt 2} \tau_\mu c_{22}(\theta) c_{12}(\theta)
 \end{equation}
and
\begin{equation}  \label{e:intsim2}
  \int_0^{T_0}\chi_{12}^c c_{22}^Z\,\d t =  \sqrt2 c_{12}(\theta)\,c_{22}(\theta) \int_0^{T_0} C(t,\phi,0) \cos(2\omega t + 2\phi)\d t =  0.
  \end{equation}
\details{[{\bf New}]
 By Corollary \ref{c:tqexpr} and \eqref{e:chiUc},\eqref{e:chiTc}
\begin{align*}
\int_0^{T_0}\til\chi_{12}^c c_{22}^Z \dt =   c_{12}(\theta) c_{22}(\theta) \int_0^{T_0} \cos(2\omega t + 2\phi) C(t,\phi,0)\d t
= \frac{1}{2\omega} c_{12}(\theta) c_{22}(\theta) \int_0^{T_0} \cos(2\omega t + 2\phi)  \sin(2\omega t + 2\phi) \d t
=0
\end{align*}
and by \eqref{e:ctfz}
\begin{align*}
\int_0^{T_0} \chi_{22}^c c_{12}^Z \til\d t =
\int_0^{T_0} \frac{-\sqrt 2}{r_\mu}c_{22}(\theta) \cos(2\omega t + 2\phi +2\gamma_\mu) c_{12}(\theta)\cos(2\omega t + 2\phi) \d t =
 \frac{-\sqrt 2 }{ r_\mu}\cos\gamma_\mu \frac{T_0}2c_{12}(\theta)c_{22}(\theta)
=\frac{-\sqrt 2 }{ r_\mu}\frac{\mu}{r_\mu} \frac{T_0}2c_{12}(\theta)c_{22}(\theta)
\end{align*}
}
\msk    
 
Thus~\eqref{e:determ} and~\eqref{e:intmhd}-\eqref{e:intsim2} give
\begin{align}
  \int_0^{T_0}  p_{11}^Z&\big(\D\widetilde \LL^l(t,Z)\chi^c(t,Z)\big)D\,\d  t \notag \\
  &=-\frac{T_0}{\sqrt 2}
 \left( \frac{1}{\sqrt 3}\tau_\lam c_{01}(\theta)c_{11}(\theta) + \tau_\mu (c_{21}(\theta)c_{11}(\theta) + c_{22}(\theta)c_{12}(\theta)) \right) E_{11}^Z. 
 \label{e.DETerm}
\end{align}
Using the above calculations, we can now evaluate the bifurcation function for the term $\LL(\Q)D$ as a linear combination~\eqref{e:lincombE} of Cases~(i),(ii),(iii). The corresponding $y$ term has the form
\begin{equation}
y(t,\Q) =  m_c y^c(t,\Q) +  m_l y^l(t,\Q) +  m_q y^q(t,\Q)
\end{equation}
with $y^c$,$y^l$ and~$y^q$ as in Proposition~\ref{p:ytzexpr1} with \eqref{e:y2list} and~\eqref{e:y3list},\eqref{e:yhat} respectively, while
\begin{equation}
\chi(t,\Q) =  m_c \chi^c(t,\Q) +  m_l \chi^l(t,\Q) +  m_q \chi^q(t,\Q)
\end{equation}
with the relevant components given   
by~\eqref{e:chiz}, \eqref{e:chiUc}, \eqref{e:chiTc} for~$\chi^c$, by~\eqref{e:tilchi} for~$\chi^l$ and by~\eqref{e:chihat} for~$\chi^q$.   
\msk

First take the restricted case~$ m_q=0$. Here we find from~\eqref{e:f2zterms}
\begin{align}
  F_2(Z)&= \sum_{i,j\in\{c,l\}}\lam_i\lam_j\int_0^{T_0}B^Z_{11}\bigl(\chi_N^i(t,Z),y_T^j(t,Z)\bigr)\d t  \notag \\
  &\qquad\qquad + m_l \sum_{j=c,l}\int_0^{T_0} p_{11}^Z\big(\D\widetilde \LL^l(t,Z)\chi^j(t,Z)\big) D\d t
\label{e:sumij}
\end{align}
since $\D\widetilde \LL^c(t,Z)=0$.  Writing the bifurcation function $\F(Z,\beta)$ in coordinates as
\[
\F(Z(\theta,\phi),\beta)=f(\theta,\phi,\beta)
\]
so that  
\begin{equation}  \label{e:f2zcoords}
F_2(Z)=\frac12\F''(Z(\theta,\phi),0)E_{11}^Z=\frac12f''(\theta,\phi,0)E_{11}^Z=:f_2(\theta)E_{11}^Z,
\end{equation}
here dropping the redundant variable~$\phi$,
we observe from  \eqref{e.p_K^ZBchiZy} and \eqref{e.p_K^ZBchiUy} (for $B^Z$) and~\eqref{e.DETerm} (for $\D\widetilde \LL$) that each term in $f_2(\theta)$  is a linear combination of the two terms:
\begin{align}  \label{e:termsdef}
s_0(\lam,\theta)&:= \frac{T_0}{\sqrt6 a} \tau_\lam c_{01}(\theta)c_{11}(\theta)\\
s_2(\mu,\theta)&:=-\frac{T_0}{3\sqrt2 a}\tau_\mu \big(c_{21}(\theta)c_{11}(\theta) + c_{22}(\theta)c_{12}(\theta)\big),
\end{align}
the notation reflecting the fact that it is only the components $\chi_0$ and $\chi_2$ of $\chi$ that play any role here.
\msk

The coefficients of these arising from the various terms that appear in~\eqref{e:sumij} are respectively as follows:
\msk

  \begin{tabular}{rccl}
   {}& $s_0(\lam,\theta)$ & $s_2(\mu,\theta)$  {}\\
  $\int_0^{T_0}B^Z_{11}\bigl(\chi_N^c(t,Z),y_T^c(t,Z)\bigr)\d t$ & $1$
   & $1$  
   & by \eqref{e.p_K^ZBchiZy},\eqref{e.p_K^ZBchiUy} \\
  \rule{0ex}{3.0ex}$\int_0^{T_0}B^Z_{11}\bigl(\chi_N^c(t,Z),y_T^l(t,Z)\bigr)\d t$&  $a$  &$a$ 
   &by \eqref{e:y2list} \\
  \rule{0ex}{3.0ex}$\int_0^{T_0}B^Z_{11}\bigl(\chi_N^l(t,Z),y_T^c(t,Z)\bigr)\d t$& $2a$ & $-2a$ 
   &by  \eqref{e:tilchi}\\
  \rule{0ex}{3.0ex}  $\int_0^{T_0}B^Z_{11}\bigl(\chi_N^l(t,Z),y_T^l(t,Z)\bigr)\d t$ & $2a^2$ & $-2a^2$  & by \eqref{e:tilchi},\eqref{e:y2list}\\
 \rule{0ex}{3.0ex}  $\int_0^{T_0} p_{11}^Z\big(\D\widetilde \LL^l(t,Z)\chi^c(t,Z)\big)D\,\d t$ & $-a $ & $3a$  & by \eqref{e.DETerm} \\
  \rule{0ex}{3.0ex} $\int_0^{T_0} p_{11}^Z\big(\D\widetilde \LL^l(t,Z)\chi^l(t,Z)\big)D\,\d t$ &  $-2a^2$ & $-6a^2 $
 & by \eqref{e:tilchi}
  \end{tabular}
\msk
  
\noindent and so collecting up terms in~\eqref{e:sumij} gives 
\begin{equation}   \label{e:f2cz}
f_2(\theta) = \til\Lambda_0\,s_0(\lam,\theta)
+\til\Lambda_2\,s_2(\mu,\theta)
\end{equation}
where
\begin{align}  \label{e:lam12}
\til\Lam_0 &:= ( m_c^2 + 3a  m_c m_l + 2a^2 m_l^2) - (a m_c m_l+2a^2 m_l^2) =  m_c^2 + 2a m_c m_l  \\
\til\Lam_2 &:=( m_c^2-a m_c m_l-2a^2 m_l^2)+(3a m_c m_l-6a^2 m_l^2)
= m_c^2 + 2a  m_c m_l -8a^2 m_l^2.
\end{align}
\details{CW do not remove:
\begin{align*} 
\til\Lambda_1 &= \lambda_c^2  + a \lambda_c\lambda_l +2a \lambda_c\lambda_l
+ 2 a^2 \lambda_l^2- a \lambda_c\lambda_l - 2a^2 \lambda_l^2\\
\til\Lambda_l &=\lambda_c^2 + a\lambda_c\lambda_l -2a \lambda_c\lambda_l
-2a^2\lambda_l^2 + 3a \lambda_c\lambda_l - 6a^2\lambda_l^2.
\end{align*}
}
Now consider terms involving $ m_q$, not yet included.  Since $y_1^q=0$ from~\eqref{e:yqexpr} and $\chi_N^q=\chi_0^q$ from~\eqref{e:chihat} the only terms that arise from~$B^Z$ are 
  \begin{align}
    \int_0^{T_0}B^Z_{11}\bigl(\chi_0^q(t,Z),y_1^c(t,Z)\bigr)\d t
    &= 6a^2\int_0^{T_0}B^Z_{11}\bigl(\chi_0^c(t,Z),y_1^c(t,Z)\bigr)\d t \notag \\
    &= 6a^2 s_0(\lam,\theta) E_{11}^Z 
  \end{align}
from~\eqref{e.p_K^ZBchiZy}, and likewise from~\eqref{e:y2list}
   \begin{equation}  \label{e:Bql}
    \int_0^{T_0}B^Z_{11}\bigl(\chi_0^q(t,Z), y_1^l(t,Z)\bigr)\d t
    = 6 a^3  s_0(\lam,\theta) E_{11}^Z.
   \end{equation}
Regarding terms arising from $\D\widetilde \LL$,  we have from~\eqref{e:determ} and~\eqref{e:chihat} 
\begin{equation}
  \big(\D\widetilde \LL^l(t,Z)\chi^q(t,Z)\big) D= \frac{1}{\sqrt 3} \chi_{01}^qc_{11}^Z E_{11}^Z 
  = 2{\sqrt 3}a^2\chi_{01}^cc_{11}^Z E_{11}^Z  
\end{equation}
and so 
\begin{equation}  \label{e:detilchihat}
  \int_0^{T_0} p_{11}^Z\big(\D\widetilde \LL^l(t,Z)\chi^q(t,Z)\big)D\d t = - 6 a^3
s_0(\lam,\theta) E_{11}^Z 
\end{equation}     
using~\eqref{e:intmhd}. Observe also that from Proposition~\ref{p:dtile} (iii) and~\eqref{10.57}, \eqref{e:chiTc}
\begin{align}
  \int_0^{T_0} p_{11}^Z\big(\D\widetilde \LL^q(t,Z)&\chi^c(t,Z)\big)D\d t
  =\sqrt 2\int_0^{T_0}\left<E_2(\tfrac\pi4-\omega t),Z\right>\chi_{11}^c(t,Z)E_{11}^Z\d t \notag \\
  &=\int_0^{T_0}3\sqrt2\,a\,\sin(2\omega t+2\phi)\sin^2\theta c_{11}(\theta)S(t,\phi,0)E_{11}^Z \d t = 0 \label{e:intq1}
\end{align}
as in~\eqref{e:intdhm}, and so likewise
\begin{equation}  \label{e:intq2}
\int_0^{T_0} p_{11}^Z\big(\D\widetilde \LL^q(t,Z)\chi^l(t,Z)\big)D\d t = 0
\end{equation}
and also
\begin{equation}  \label{e:intq3}
\int_0^{T_0} p_{11}^Z\big(\D\widetilde \LL^q(t,Z)\chi^q(t,Z)\big)D\d t = 0
\end{equation}
because $\chi_{11}^q=0$. Thus the contribution to $F_2(Z)$ arising from the $\D\widetilde \LL$ term in~\eqref{e:f2zterms} depends only on the linear (in $\Q$) contribution~$\LL^l(\Q)D$.
\msk

Therefore the contribution of the $ m_q$-term in $\LL(\Q)D$ is merely to replace $\til\Lam_i$ in~\eqref{e:lam12} by $\Lam_i$ for $i=0,2$ where
\begin{equation}   \label{e:lamplus}
  \Lam_0=  m_c^2 + 2a   m_c m_l + 6a^2 m_c m_q 
\end{equation}
and $\Lam_2=\til\Lam_2$,
the $ m_l m_q$ terms from~\eqref{e:Bql} and~\eqref{e:detilchihat} fortuitously cancelling. 
\section{Zeros of the bifurcation function, periodic orbits and stability}  \label{s:zeros}
Since from~\eqref{e:mdefs}  
\begin{align}
  c_{01}(\theta) c_{11}(\theta)& = \frac{ \sqrt 3}{4} \sin^2\theta \sin2\theta \label{e:mdval} \\
  c_{21}(\theta)c_{11}(\theta) & = \frac14(1+ \cos^2\theta) \sin2\theta  \label{e:bdval} \\
   c_{22}(\theta)c_{12}(\theta) & = \sin\theta\cos\theta = \frac12\sin2\theta \label{e:ceval} \\
\end{align}  
the expression~\eqref{e:f2cz} using $\Lam_0,\Lam_2$ becomes
$f_2(\theta)=\frac{T_0}{12\sqrt2 a}\sin2\theta\widehat f_2(\theta)$ where 
\begin{equation}   \label{e:f2cz+}
\widehat f_2(\theta)
=3\Lambda_0\tau_\lam  \sin^2\theta 
-\Lambda_2\tau_\mu(3+\cos^2\theta)
\end{equation}
with $\tau_\lam,\tau_\mu$ both nonzero by Assumption~3.
\begin{prop}  \label{p:gensol}
  Solutions $\theta\in[0,\pi)$ to $f_2(\theta)=0$ are given by
  \begin{equation}  \label{e:pi2}
\theta=0,\frac\pi2
  \end{equation}
and by solutions $\theta$ to
  \begin{equation}  \label{e:othersol1}
3\Lam_0\tau_\lam \sin^2\theta
  = \Lam_2\tau_\mu(3+\cos^2\theta),
  \end{equation}
 that is (assuming $\Lam_2\ne0$)
  \begin{equation}  \label{e:othersol2}
(3\frac{\Lam_0}{\Lam_2}\frac{\tau_\lam}{\tau_\mu} + 1) \sin^2\theta = 4.
  \end{equation}
    \qed
\end{prop}
If $\Lam_2=0$ then solutions $\theta\ne0,\frac\pi2\in[0,\pi)$ to~\eqref{e:othersol1} exist only if also $\Lam_0=0$ in which case $f_2(\theta)$ vanishes identically.  However, if $\Lam_0=\Lam_2=0$
  then
 \begin{align}
    \xi^2 + 2a\xi\eta + 6a^2\xi &= 0  \notag \\
    \xi^2 + 2a\xi\eta -8a^2\eta^2&=0  \notag
    \end{align}
where $(\xi,\eta)=( m_c/ m_q, m_l/ m_q)$, supposing $ m_q\ne0$   (otherwise $\Lam_0=\Lam_2=0$ implies $ m_c= m_l=0$ also). Subtracting gives $3\xi= -4\eta^2$ and so the first equation factors into $\xi=0$ (so $\eta=0$) or
  \[
-4/3\eta^2 + 2a\eta + 6a^2 = -\frac23(2\eta + 3a)(\eta - 3a)=0
\]
giving $(\xi,\eta) = (-3a^2,-3a/2)$ or $(-12a^2,3a)$.
Thus $\Lam_0=\Lam_2=0$ just when $( m_c: m_l: m_q)=(0:0: m_q)$ or $(-12a^2:3a:1)$ or $(6a^2:3a:-2)$; we exclude these possibilities. 

\msk

Therefore assuming $\Lam_2\ne0$ there exist solutions $\theta\ne0,\pi/2\in[0,\pi)$ to~\eqref{e:othersol2} if and only if $3\frac{\Lam_0}{\Lam_2}\frac{\tau_\lam}{\tau_\mu} + 1 > 4$, that is  $\frac{\Lam_0}{\Lam_2}\frac{\tau_\lam}{\tau_\mu}>1$. Hence
\begin{cor}   \label{c:2zeros}
  The second order term $f_2(\theta)$ of the bifurcation function has no zeros $\theta\ne 0,\pi/2\in[0,\pi)$ if $\frac{\Lam_0}{\Lam_2}\frac{\tau_\lam}{\tau_\mu}\leq 1$, while if $\frac{\Lam_0}{\Lam_2}\frac{\tau_\lam}{\tau_\mu}>1$ there are two zeros
  $\theta=\pi/2\pm\Theta$ with $\Theta\to0$ as $\frac{\Lam_0}{\Lam_2}\frac{\tau_\lam}{\tau_\mu}\to1$.
  \qed
\end{cor}
In the specific case of the  Beris-Edwards model~\eqref{e:sys1} with $\LL(\Q)D$ given by~\eqref{e:newterm} with ratios
\[
 ( m_c: m_l: m_q) = (\frac23:1:-2)
\]
we observe that $\Lam_0=\Lam_2$ regardless of the value of the coefficient~$a$. It happens that the  simpler Olmsted-Goldbart model~\cite{Chillingworth2001}, \cite{OlmGol}, \cite{VAWS2003}  for which $( m_c: m_l: m_q) = (1:0:0)$ also yields~ $\Lam_0=\Lam_2$.  Thus in both these cases we have a tidier result.
\details{ for $(2/3 : 1 : -2)$ we have
  \[
  \Lambda_0 =  m_c^2  +2a  m_c  m_l + 6 a^2  m_c m_q=
4/9 m_l^2 + 4/3 a m_l^2 -  6a^2*2/3*2  m_l^2 =
(4/9  + 4/3 a  -  8a^2) m_l^2
\]
 and
 \[
\Lam_2 =  m_c^2 + 2a  m_c m_l - 8 a^2  m_l^2=
4/9  m_l^2 + 4/3 a  m_l^2 - 8a^2   m_l^2=\Lam_0
\]
}
\begin{cor}   \label{c:2zerosBE}
For the Beris-Edwards model and the Olmsted-Goldbart model the second order term    $f_2(\theta)$ of the bifurcation function has no zeros $\theta\ne0,\pi/2\in[0,\pi)$ if ${\tau_\lam}/{\tau_\mu}\le1$, while if ${\tau_\lam}/{\tau_\mu}>1$ there are two zeros
  $\theta=\pi/2\pm\Theta$ with $\Theta\to0$ as ${\tau_\lam}/{\tau_\mu}\to1$.
  \qed
\end{cor}
These models both have $ m_c\ne0$.   If $ m_c=0$ with $ m_l\ne0$ (so $\LL(\Q)D$ has linear but no constant term) then $\Lam_0=0$ while $\Lam_2\ne0$ and we see from~\eqref{e:othersol1} that $f_2(\theta)$ does not vanish for any~$\theta\ne0,\pi/2\bmod\pi$. 
\subsection{Periodic orbits} 
Since 
\begin{equation}  \label{e:ftb}
f(\theta,\beta)=\beta^2(f_2(\theta) + O(\beta))
\end{equation}
\ as in~\eqref{e:f2zcoords}, the Implicit Function Theorem implies that if $\theta=\theta_{\,0}$ is a simple zero of $f_2$
%(that is $f_2(\theta_{\,0})=0$ and $\d f_2(\theta_{\,0})/\d\theta\ne0$)
then for sufficiently small $|\beta|>0$ there exists a unique $\theta_\beta$ close to $\theta_{\,0}$ such that the right hand side of~\eqref{e:ftb} vanishes at $\theta=\theta_\beta$ and $\theta_\beta\to \theta_{\,0}$ as $\beta\to 0$.  Thus $\theta_\beta$ corresponds to a solution $Z_\beta=Z(\theta_\beta,\phi)\in \M_\phi$ to the bifurcation equation $\F(Z,\beta)=0$ for sufficiently small $|\beta|>0$ with $Z_\beta\to Z(\theta_{\,0},\phi)$ as $\beta\to 0$.
\msk

In fact we know by Proposition~\ref{e.perturbedEquator} that the solutions $\theta=0,\pi/2$ corresponding to the north pole $Q^*$ and equator~$\C$ do persist for sufficiently small $|\beta|>0$, and we verify that
\begin{equation}  \label{e:deriv1}
   \frac{\d f_2(\theta)}{\d\theta}\Big|_{\theta=0}
  =-\frac{{\sqrt2}T_0}{3a}\Lam_2\tau_\mu\,, \quad
  \frac{\d f_2(\theta)}{\d\theta}\Big|_{\theta=\frac\pi2}
  =-\frac{ T_0}{2 \sqrt2 a}(\Lam_0\tau_\lam-\Lam_2\tau_\mu)
\end{equation}
\details{since from \eqref{e:f2cz+} using correct sign 
\begin{align*}
f_2'(\theta) &= \frac{T_0}{12\sqrt2 a}2\cos2\theta\big(3\Lam_0\tau_\lam \sin^2\theta
 - \Lam_2\tau_\mu(3+\cos^2\theta)\big)
  + \frac{T_0}{12\sqrt2 a}\sin2\theta\big(6\Lam_0\tau_\lam \sin\theta\cos\theta
  +2\Lam_2\tau_\mu\cos\theta\sin\theta\big)\\
 & =  \frac{T_0}{6\sqrt2a}\cos2\theta\big(3\Lam_0\tau_\lam \sin^2\theta- \Lam_2\tau_\mu(3+\cos^2\theta)\big)
 + \frac{T_0}{6\sqrt2 a}\sin2\theta\cos\theta\sin\theta\big(3\Lam_0\tau_\lam  
  +\Lam_2\tau_\mu\big)
  \end{align*}
  and so $f_2'(\pi/2)  =  \frac{T_0}{6\sqrt2 a}3(-\Lam_0\tau_\lam  + \Lam_2\tau_\mu)$.
}
and so the north pole solution is always a simple solution, while the equator solution is a simple solution provided $\Lam_0\tau_\lam\ne\Lam_2\tau_\mu$.
In general, if $\theta=\theta_{\,0}:=\pi/2\pm\Theta$
 is another zero of $f_2$ then 
\begin{align}  
\frac{\d f_2(\theta)}{\d\theta}\Big|_{\theta=\theta_{\,0}}
&=\frac{T_0}{12\sqrt2 a}\sin^22\theta_{\,0}(3\Lam_0\tau_\lam +\Lam_2\tau_\mu) \notag \\
&=\frac{T_0}{12\sqrt2  a}\sin^22\theta_{\,0}\Lam_2\tau_\mu(3\frac{\Lam_0\tau_\lam}{\Lam_2\tau_\mu} + 1) \label{e:deriv2}
\end{align}
which is nonzero since $\Lam_0\tau_\lam,\Lam_2\tau_\mu$ have the same sign by Corollary~\ref{c:2zeros}, and so $\theta_{\,0}$ is also a simple solution.
\msk

When $\Lam_0=\Lam_2$ as in the Beris-Edwards or Olmsted-Goldbart models we thus have the following result on periodic orbits after perturbation.
\begin{cor}  \label{c:casei}
For the Beris-Edwards or Olmsted-Goldbart models
under Assumptions 1-4 for fixed $\lam,\mu$ with $\tau_\lam/\tau_\mu<1$ the equator $\C$ is the unique periodic orbit on~$\kO$  (other than the equilibrium~$\Q^*$) that persists for sufficiently small~$|\beta|>0$; its period is close to $\pi/\omega$. For   $\tau_\lam/\tau_\mu>1$ there is in addition $\beta_0>0$ and a smooth path $\{Q(\beta):|\beta|<\beta_0\}$ in~$V$ with $\Q(0)=Z(\theta,\phi)\in\M_\phi$ where $\theta=\pi/2\pm\Theta$ as in Corollary~\ref{c:2zerosBE} such that there is a periodic orbit of~\eqref{e:sys1} through $\Q(\beta)$ with  period $T(\Q(\beta),\beta)\to T_0=2\pi/\omega$ as $\beta\to 0$.  \qed
\end{cor}
The perturbed equator represents a periodic orbit close to tumbling, possibly with a small kayaking and/or biaxial component.
The periodic orbit through~$\Q(\beta)$ represents a kayaking orbit that (for fixed~$\lam,\mu$) arises from a particular kayaking orbit on~$\kO$ persisting after perturbation.  The two values~$\theta=\pi/2\pm\Theta$  correspond to the two intersections of the {\em same} periodic orbit with the Poincar\'e section: see the geometric description at the end of Section~\ref{s:inout}.   Thus if (sufficiently small) $\beta\ne0$ is fixed and $\tau_\lam/\tau_\mu$ increases through~$1$, the equator tumbling orbit generates a kayaking orbit through a period-doubling bifurcation.
\subsection{Stability}  \label{s:stab} 
So far the discussion has rested on Assumption~3 ensuring the normal hyperbolicity of the $\SO(3)$-orbit~$\kO$ under the dynamics of the system~\eqref{e:sys1} when $\beta=0$. In this section we investigate dynamical stability of the periodic orbits on $\kO$ that persist close to~$\kO$ for sufficiently small~$|\beta|>0$. A necessary condition for stability is that~$\kO$ itself be an attracting set, and so we make now the following further assumption:
\msk

\noindent \textbf{Assumption 5}: The eigenvalues $\lam, \mu$ of $\D G(Q^*)$ are negative.
\msk

Consequently the perturbed flow-invariant manifold~$\kO(\beta)$ is normally hyperbolic and attracting for sufficiently small~$|\beta|>0$, therefore the stability of any equilibrium or periodic orbit lying on $\kO(\beta)$ is determined by its stability or otherwise relative to the system~\eqref{e:sys1} restricted to~$\kO(\beta)$.  The manifold $\kO(\beta)$ can be seen as the image of a section of the normal bundle of $\kO$, its intersection with $\U_\M^\eps$ being the image $\M(\beta)$ of a section $\til\sigma(\cdot,\beta)$ of this normal bundle restricted to~$\M=\M_\phi$. The 1-manifold $\M(\beta)$ is invariant under the Poincar\'e map $P(\cdot,\beta)$, the restriction of  $P(\cdot,\beta)$ to  $\M(\beta)$ determining a 1-dimensional discrete dynamical system on $\M(\beta)$ whose fixed points correspond to periodic orbits (or fixed points) of $F(\cdot,\beta)$ on $\kO(\beta)$.
\msk

In our analysis, rather than use $\til\sigma(\cdot,\beta)$ which is harder to compute, we have used $\sigma(\cdot,\beta)$ and the method of Lyapunov-Schmidt to construct a vector field
\[
Z\mapsto P_{11}(Z+\sigma(Z,\beta),\beta)=\F(Z,\beta)E_{11}^Z 
\]
on $\M$ whose zeros correspond to the periodic orbits (or fixed points) of $F(\cdot,\beta)$ on $\kO(\beta)$.  It follows from the general Principle of Reduced Stability~\cite{KieLau83}, \cite{VBH} that stability of periodic orbits on $\kO(\beta)$ corresponds to stability of the corresponding zeros of the vector field $\F(Z,\beta)E_{11}^Z$ on~$\M$ in the present context where~$\dim(\M)=1$. However, we now show this directly, using a simple geometric argument taken from~\cite[Section~9.4]{CH}.
Recall that in terms of the $\theta$-coordinate for~$Z$ on~$\M$ we have  $\F(Z,\beta)=f(\theta,\beta)$.
\begin{prop}  \label{p:CHstab}
  For fixed $\beta$, let $\Q_0(\beta):=Z_0+\til\sigma(Z_0,\beta)\in\M(\beta)$ be a hyperbolic fixed point for the Poincar\'e map~$P(\cdot, \beta)$, with $Z_0=Z(\theta_{\,0},\phi)$ for $\theta_{\,0}$ a hyperbolic zero of the system $\dot\theta=f(\theta,\beta)$ on~$\M$.  Then $\Q_0(\beta)$ is stable (attracting) if and only if $\theta_{\,0}$ is  stable (attracting). 
    \end{prop}
\begin{proof}
  Suppose this fails for a given fixed value of~$\beta$, so that (without loss of generality) $\Q(\beta)$ is attracting on~$\M(\beta)$ while $\theta_{\,0}$ is repelling on~$\M$. In particular this means that there is an interval $(\theta_-,\theta_{\,0})$ such that all corresponding points on $\M(\beta)$ are moved to the right (greater $\theta$-value) by the Poincar\' e map $P(\cdot,\beta)$, and there is also an interval $(\theta_{\,0},\theta_+)$ on which $f(\theta,\beta)>0$. Now consider a perturbation of the system~\eqref{e:sys1} which adds a vector field of the form $\Q\mapsto\zeta(\Q)E_{11}^Z$ where $\zeta:V\to\bR$ is a smooth non-negative bump function with $\zeta(\Q(\beta))>0$ and vanishing outside a sufficiently small neighbourhood $U$ of $\Q(\beta)$ in~$V$.
  Note that such a perturbation will be far from $\SO(3)$-equivariant as it is localised on~$U$. For sufficiently small $\zeta$ the effect of the perturbation will be to ensure that there is a larger open interval $J_-\supset(\theta_-,\theta_{\,0}]$ on which corresponding points on $\M(\beta)$ are moved to the right, while there is  a larger open interval $J_+\supset [\theta_{\,0},\theta_+)$ on which $f(\theta,\beta)>0$.
Therefore the fixed point $\Q(\beta)$ of the perturbed Poincar\'e map must have $\theta$-coordinate greater than $\theta_{\,0}$, while the zero of $f(\cdot,\beta)$ is a point on $\M$ with $\theta$-coordinate less than $\theta_{\,0}$.  However, this contradicts the fact that fixed points of the Poincar\'e map correspond to zeros of the bifurcation function via projection in the normal bundle over~$\M$, and so proves the Proposition.  
\end{proof}
\begin{corollary}   \label{c:stab2}
  Under Assumptions 1-5, if $\theta$ is a simple zero of $f(\cdot,\beta)$ then the corresponding periodic orbit (or fixed point)
  of~\eqref{e:sys1} is linearly stable or unstable according as
  $\d f_2(\theta,0)/\d\theta$ is negative or positive.
  \qed
\end{corollary}
%%%%%%%%%%%%%%%%%%%%%%%%%
\msk

\subsection{Stable kayaking orbits}
We are now able to describe the global dynamics close to~$\kO$ for the Beris-Edwards model, under the standing Assumptions~1-5.
From Corollary~\ref{c:casei} and~\eqref{e:deriv1},\eqref{e:deriv2} we deduce the following stability result.
\begin{theorem}  \label{t:stability1}
  For the Beris-Edwards model first
suppose $\Lam_2>0$. Then for $\tau_\lam/\tau_\mu<1$ and sufficiently small $|\beta|>0$ the perturbed equator $\C(\beta)$ is an attracting limit cycle (close to tumbling) on the invariant manifold $\kO(\beta)$ that is the perturbed $\SO(3)$-orbit~$\kO$, its basin of attraction on~$\kO(\beta)$ being the whole of~$\kO(\beta)$ apart  from the perturbed equilibrium~$\Q^*\!(\beta)$ (log-rolling).  For  $\tau_\lam/\tau_\mu>1$ the perturbed equator $\C(\beta)$ is a repelling limit cycle, and there is precisely one other limit cycle on~$\kO(\beta)$: this limit cycle (kayaking) is attracting, and has period approximately twice that of~$\C(\beta)$.  \,If $\Lam_2<0$ the attraction/repulsion is reversed.
  \qed
\end{theorem}
For the simpler Olmsted-Goldbart model we have $\Lam_2= m_c^2>0$  and so stability of the kayaking orbit (when it exists) automatically holds.  In general we have
\[
\Lam_2=( m_c-2a m_l)( m_c+4a m_l)
\]
and so the stability condition $\Lam_2>0$ holds precisely when $w<-4a$ or $w>2a$ where $w:= m_c/ m_l$, supposing $ m_l\ne0$.  If $ m_l=0$,$\, m_c\ne0$ the kayaking orbit is automatically stable if it exists, while if $ m_l\ne0$, $\, m_c=0$ there is no kayaking orbit.  For the Beris-Edwards model we have $w=2/3$, and in this case stability depends on the coefficient~$a$ and holds automatically given that $a<1/3$.  Thus, to summarise:
\begin{cor}  \label{c:stabsum}  
For the Beris-Edwards and Olmsted-Goldbart models, if the $\SO(3)$-orbit of the logrolling equilibrium $\Q^*$ is normally hyperbolic and attracting (so that $\Q^*$ is a stable equilibrium state in the absence of the shear flow, up to rigid rotations) then the kayaking orbit, when it exists, is an asymptotically stable limit cycle.
  \end{cor}

\begin{remark}  \label{r:klammu}
Given Assumption~5 the condition $\tau_\lam/\tau_\mu>1$ is the same as $\tau_\lam<\tau_\mu$, that is   
\[
k(\lam,\mu)<0
\]
where
\begin{align}
 k(\lam,\mu):=\tau_\lam-\tau_\mu&= \lam r_{\lam}^{-2}- \mu r_{\mu}^{-2}   \notag \\
  &=r_\mu^{-2}r_\lam^{-2}\bigl(\lam(\mu^2+4\omega^2)-\mu(\lam^2+4\omega^2)\bigr) \notag \\
  &=r_\mu^{-2}r_\lam^{-2}(\lam-\mu)(4\omega^2-\lam\mu). 
\end{align}
\end{remark}
Our result on kayaking orbits for the Beris-Edwards model can therefore be expressed as follows:
\begin{theorem}  \label{t:lammu}
  For the Beris-Edwards model~\eqref{e:sys1},\eqref{e:newterm} the condition for the existence of a kayaking orbit for sufficiently small~$|\beta|>0$ is that $\lam-\mu$ and $4\omega^2-\lam\mu$ have opposite signs; such a kayaking orbit is automatically linearly stable given that $a<1/3$ for physical reasons (see~\eqref{e.Q*}).  
 \qed
\end{theorem}
\subsection{The gradient case}  \label{s:grad}
In the Beris-Edwards model and others widely used in the literature the equivariant interaction field $G$ is the {\em negative} gradient of a smooth free energy function $V\to\bR$ which is frame-indifferent, thus invariant under the action of $\SO(3)$ on~$V$.  From general theory~\cite{SCH}, such a function has the form 
\[
Q\mapsto f(X_1(\Q),X_2(\Q),\ldots,X_m(\Q))
\]
where $f:\bR^m\to\bR$ is a smooth function and $\{X_1,X_2,\ldots,X_m\}$ are a basis for the ring of $\SO(3)$-invariant polynomials on~$V$. It is well known in the liquid crystal literature (see for example~\cite[eq.(4.9)]{MacM92a}) that such a basis is given by~$\{X,Y\}$ where 
\[
X(\Q)= \tr\,\Q^2,\quad Y(\Q)=\tr\,\Q^3,
\]
a proof being given in~\cite[Ch.XV, \S6]{GSS} via reduction to the group of symmetries of an equilateral triangle.  Note that for $\Q\in V$ the  Cayley-Hamilton Theorem shows immediately that $\tr\,\Q^3=3\det\Q$.

With $f_X,f_Y$ denoting the partial derivatives of $f$ we find that
the functions $g ,\bar g $ of~\eqref{e:gexpr} are then given by
\begin{equation} 
  g (\Q)=-2f_X(\Q), \qquad   \bar g (\Q)=-\frac32f_Y(\Q) 
\end{equation}
\details{$\frac{\d Y}{\d \Q}W=3\tr(\Q^2W)=3\tr\bigl((\frac12[\Q,\Q\,]^+ +\frac13\tr(\Q^2)I)W\bigr)=\bigl(\frac32[\Q,\Q\,]^+W\bigr)$}
and so also for their derivatives  
\begin{equation} 
  \D g  =-2\D f_X, \qquad   \D \bar g  =-\frac32 \D f_Y. 
\end{equation}
The equilibrium condition~\eqref{e:equilg} is
\begin{equation}  \label{e:equilf}
2f_X^* + 3af_Y^* = 0
\end{equation}
\details{$g^*+2ah^*=0$ so $-2f_X^* + 2a(-3/2 f_Y^*)=0$}
where $f_X^*,f_Y^*$ denote $f_X(\Q^*),f_Y(Q^*)$ respectively.
The eigenvalues of $\D G(\Q^*)$ are $\lambda,\mu$ and $0$ where 
 by \eqref{e:newlam} and \eqref{e:mudef},
\begin{align} 
  \lambda&=2f_X^* - 2\Delta f_X^* - 3a\,\Delta f_Y* \\
  \mu&= -2f_X^* + 6af_Y^* = -6f_X^* = 9af_Y^*  
\end{align}
with $\Delta f_X^*:=\D f_X(\Q^*)\Q^*$ and likewise $\Delta f_Y^*$.
\details{$2ah^* + \Delta g^* + 2a\Delta h^* = 2a(-3/2 f_Y^*) -2\Delta f_X^* +2a(-3/2\Delta f_Y^*) = -3af_Y^* - 2\Delta f_X^* - 3a\Delta f_Y^* = = 2f_X^* - 2\Delta f_X^* - 3a\Delta f_Y^*$
}
\details{$\mu = -6ah^* = -6a(3/2 f_Y^*) = 9af_Y^*$}
For the particular and important case of the Landau -~de~Gennes potential
\begin{equation}  \label{e:ldg0}
 f(X,Y):= \frac12 \tau X - \frac13 bY + \frac14 cX^2
\end{equation}
in which $b,c>0$ we have 
\begin{align}  
  G(\Q) &= -2f_X\Q - \frac32 f_Y [\Q,\Q]^+ \\
        &= -(\tau + c|\Q|^2)\,\Q + \frac{b}2\, [\Q,\Q]^+  \label{e:ldgx}
  \end{align}
and
\begin{equation}  \label{e:ldg1}
f_X^*= \frac12\tau + \frac12 c|Q^*|^2 = \frac12\tau+3ca^2, \qquad f_Y^*= -\frac13b
  \end{equation}
giving  
\begin{equation}  \label{e:ldg2} 
\Delta f_X^* = c\left<\Q^*,\Q^*\right>  =  6a^2c, \qquad \Delta f_Y^*=0.
\end{equation}
The equilibrium condition~\eqref{e:equilf} is thus that the coefficient $a>0$ should satisfy
\begin{equation}  \label{e:ldgeq}
\tau + 6a^2c - ab = 0
\end{equation}
and the eigenvalues $\lam,\mu$ are given by
\details{
\[
\lam = 2(\frac12\tau + 3ca^2) - 2(6a^2 c) = \tau -6a^2 c= \tau - (ab-\tau)
\]
}
\begin{equation}  \label{e:eigsf}
\lam = 2\tau - ab = ab - 12a^2c, \qquad \mu = -3ab .
\end{equation}
Here $\mu$ is automatically negative, and it is straightforward to check that~\eqref{e:ldgeq} has two real solutions~$0 < a_1 < a_2$  provided $0<\tau<b^2/(24c)$.  Then $a_2>\frac12(a_1+a_2)=b/(12c)$ and so $\lam<0$ for $a=a_2$ and we choose $a=a_2$ in the definition of $\Q^*$.
\msk

\details{
  \[
  a = b/(12 c) \pm \sqrt{b^2/(144 c^2) - \tau/6c)  } \text{\ so
    need\ } b^2/(144 c^2) >  \tau/6c \text{\ or\ } b^2/(24 c)>\tau.
  \]
  }
  \begin{cor}
    In this setting the result of Theorem~\ref{t:lammu} giving the condition for the existence of kayaking orbits becomes
    \begin{equation}  \label{e:kaycond}
\bigl((a+3)b-12a^2c\bigr)\bigl(4\omega^2+3b(ab-12a^2c)\bigr)<0
\end{equation}
with stability for $a=a_2<1/3$.
\qed
\end{cor}
\details{\begin{align*} \lam \mu &= -3b(ab-12 a^2 c), \text{\ so\ } 4\omega^2 -\lam\mu =(4\omega^2+ 3b(ab-12 a^2 c)) \\
    &\text{\ and\ } (\lam-\mu)(4\omega^2 -\lam\mu)<0 \text{\ gives\ } ((ab-12 a^2 c)+ 3b)(4\omega^2+ 3b(ab-12 a^2 c))<0.\end{align*}}
It is natural to ask for what range of values of $b,c,\tau$ and $\omega$ these conditions can simultaneously hold.
\begin{prop}  \label{p:stabcond}
A necessary condition for the existence of stable kayaking orbits is $b<4c$.  Given that this holds, then if $5b<2c$ such orbits exist for all~$\omega>0$ while if $5b>2c$ they exist for
\[
 4\omega^2 < b(4c-b).
\]
The range of $\tau$ or which these orbits exist is given by
\begin{equation}  \label{e:taucond}
  \frac13(b-2c)<\tau<b^2/(24c).
\end{equation}
\end{prop}
\proof
From~\eqref{e:eigsf} the condition $\lam<0$ is $a>b/(12c)$ given that $a>0$, so the condition $a<1/3$ for stability (and physicality)  implies~$b<4c$.  Then
\[
a_2\in J_0:= (b/(12c),1/3)
\]
and this corresponds to~\eqref{e:taucond} since $\tau(a):= ab-6a^2c$ is  monotonic decreasing on $J_0$ (its maximum is at $a=b/(12c)$) and we have $\tau(1/3)=(b-2c)/3$ while $\tau(b/(12c))=b^2/(24c)$.
With the notation
  \begin{align}
  \Xi(a)&=3b + ba -12ca^2  \\
  \Omega(a)&=12a^2c-ba
\end{align}
the kayaking condition~\eqref{e:kaycond} is
\begin{equation}
  \Xi(a)(4\omega^2/(3b)-\Omega(a))<0,
  \end{equation}
which since $\Xi(a)+\Omega(a)=3b$ may be written
\begin{equation}  \label{e:Ystab}
  (3b-\Omega(a))(4\omega^2/(3b) - \Omega(a)) < 0.
  \end{equation}
This holds if and only if $\Omega(a)$ lies in the open interval~$J_1$ bounded by $3b$ and $4\omega^2/(3b)$, so the condition for the existence of a stable kayaking orbit (for some choice of~$\tau$) is therefore
\begin{equation}  \label{e:kay3}
 \Omega(J_0) \cap J_1 \ne \emptyset.
\end{equation}
Now $\Omega(b/(12c))=0$ and $\Omega(1/3)=(4c-b)/3>0$ and so
\[
\Omega(J_0)=(0,\frac13(4c-b)),
\]
hence~\eqref{e:kay3} holds if and only if
\begin{equation}  \label{e:stabcond0}
\frac13(4c-b) > \min\{3b,4\omega^2/(3b)\}.
\end{equation}
Observe that
\[
3b-\frac13(4c-b) = \frac23(5b-2c)
\]
and so if $5b<2c$ then~\eqref{e:stabcond0} automatically holds (regardless of~$\omega$), while if $5b>2c$ the condition~\eqref{e:stabcond0} is  
\begin{equation} \label{e:stabcond2}
(4c-b)>4\omega^2/b \quad\text{i.e.}\quad b(4c-b)>4\omega^2
\end{equation}
as stated. \qed
\section{Conclusion}
The geometry of uniaxial and biaxial nematic liquid crystal phases is most naturally expressed in terms of the action of the rotation group $\SO(3)$  
on the 5-dimensional space~$V$ of (symmetric, traceless) $\Q$-tensors. In this paper we have used techniques from bifurcation theory related to symmetry, applied to a rather general class of ODEs on~$V$ widely used to model a homogeneous nematic liquid crystal in a simple shear flow, in order to prove the existence under certain conditions of an asymptotically stable limit cycle
representing a \lq kayaking' orbit, where the principal axis of molecular orientation of the ensemble of rigid rods lies out of the shear plane and rotates periodically about the vorticity axis.
Our key assumption, however, is that the dynamical effect of the symmetric part of the flow-gradient tensor should be small compared to that of the anti-symmetric (rotational) part, so that the system we study is viewed as a perturbation of the co-rotational case which involves only the (frame-indifferent) molecular interaction field in addition to the rotation of the fluid.
 The results require expansion to second order in the perturbation parameter, as a consequence of the assumed linearity of the molecular aligning effect of the flow in terms of its velocity gradient. 
In cases where the  molecular interaction field is the negative gradient of a free energy function, such as the Landau-de~Gennes fourth order potential, we give explicit criteria on the coefficients to ensure the existence of the stable kayaking orbit for sufficiently small contribution from the symmetric part of the flow gradient.   The admissible size of this contribution is not estimated, so that care must be taken in interpreting experimental or numerical verification. 
\msk
 
\begin{acknowledgements}   This collaboration arose during a workshop at the
  Mathematics of Liquid Crystals Programme at the Isaac Newton Institute in Cambridge in~2013 where the problem of existence and stability of the kayaking orbit was raised by GF, remaining open in spite of decades of overwhelming numerical evidence together with convincing experimental evidence.  An active discussion followed and co-authors DC, RL and CW continued to work, with intermittent exchanges with GF, toward the resolution presented here. The research was supported by the Isaac Newton Institute, Cambridge and (DC) a Leverhulme Emeritus Research Fellowship; in addition CW was grateful to the Free University Berlin for hospitality. The authors also express thanks to Jaume Llibre for helpful conversations about higher-order averaging, and to Stefano Turzi for valuable input concerning invariants.
\end{acknowledgements}
\section*{Conflict of interest}
The authors declare that they have no conflict of interest.
%%%%%%%%%%%%%%%%%%%%%%%%%%%%%%%%%%%%%%%%

\appendix
\section{Equivariant maps and vector fields}  \label{s:emvf}
A map (vector field) $G:V\to V$ is {\em equivariant} (sometimes called {\em covariant}) with respect to a subgroup $\Sigma$ of $\SO(3)$ (or $\Sigma$-{\em equivariant}) when it respects all the symmetries represented by~$\Sigma$, that is
\begin{equation} \label{e:Gequi}
G(\wtr\Q)=\wtr G(\Q)
\end{equation}
for all $R\in \Sigma$ and all $\Q\in V$.
Differentiating~\eqref{e:Gequi} with respect to~$\Q$ gives
\begin{equation}  \label{e:DGequi}
  \D G(\wtr\Q) \wtr = \wtr\,\D G(\Q) : V\to V.
\end{equation}
Thus $\D G(\wtr\Q)$ is conjugate to $\D G(\Q)$ so they have the same eigenvalues, while $\wtr$ takes the eigenvectors of $\D G(\Q)$ to those of~$\D G(\wtr\Q)$.  In particular if $\Q$ is {\em fixed} by the subgroup $\Sigma$ of $\SO(3)$ then~\eqref{e:DGequi} reads
\begin{equation}  \label{e:DGfix}
  \D G(\Q) \wtr = \wtr\,\D G(\Q) 
\end{equation}
for $R\in\Sigma$, so the linear map $\D G(\Q):V \to V$ is also $\Sigma$-equivariant.
\msk

Differentiating \eqref{e:DGequi} with respect to~$\Q$ gives the expression
\begin{equation} \label{e:D2Gequi}
\D^2G(\wtr\Q)(\wtr H,\wtr K) = \wtr \D^2G(\Q)(H,K)
\end{equation}
for $H,K\in V$ and $R\in\SO(3)$.  Therefore in the case when $\Q$ is {\em fixed} by the subgroup $\Sigma$ of $\SO(3)$ the bilinear map $B=\D^2G(\Q)$ is $\Sigma$-equivariant in the sense that 
 \begin{equation}  \label{e:bilin}
B(\wtr H,\wtr K) = \wtr B(H,K)
    \end{equation}
for all $H,K\in V$ and $R\in\Sigma$.
\begin{exam}   \label{x:ex2}
  Let $G^{\,0}:V\to V$ be the $\SO(3)$-equivariant map $\Q\mapsto \Q^2-\frac13\tr(\Q^2)I$.  Here $\D G^{\,0}(\Q)H=[\Q,H]^+$ for $H\in V$, with the notation as in~\eqref{e:HK+def}. Each $Z\in\kO$ is fixed by $\Sigma_{\z}$ and so the linear map from $V$ to $V$, given by  $ H\mapsto [Z,H]^+$
is $\Sigma_{\z}$-equivariant. It therefore respects the isotypic decomposition~\eqref{e:isotyp} of $V$ into  $\Sigma_{\z}$-invariant eigenspaces of~  $ [Z,\cdot\,]^+$, with eigenvalues independent of~$Z\in\kO$.
\end{exam}
Using the characterisations of $\{V_i^*\}$ given by~\eqref{e:spanq*}-\eqref{e:uq*} it is straightforward to calculate the corresponding eigenvalues for $Z=\Q^*$ and hence for all~$Z\in\kO$. 
\begin{prop}  \label{p:Leigs}
For $Z\in\kO$   the eigenvalues for $ [Z,\cdot\,]^+$ corresponding to the eigenspaces $V_0^Z,V_1^Z,V_2^Z$  are respectively 
  \[
2a\,,a\,,-2a\,.
\]
\qed
\end{prop}
\details{\fbox{CW added proof for sake of completeness:}
From \eqref{e.Q*} we get $\Q^* = a \diag(-1,-1,2)$ so that
\[
[\Q^*,\Q^*]^+= 2 a^2 \diag(1,1,4) -6a^2 \frac{2}3 \id = 
a^2 \diag(-2,-2,4)=2 a\Q^*
\]
and from \eqref{e:E_j_def}
\[
[\Q^*, \tilde\Q(0)]^+= 2a^2 \diag(-1,-1,2)\diag(1,-1,0) = -2 a \tilde\Q(0)
\]
and from \eqref{e:E_j_def}
\begin{align*}
[\Q^*, K(0)]^+ &= \sqrt 3 a^2 \diag(-1,-1,2) \left(\begin{array}{ccc}0&0& 1\\
0 & 0 & 0\\
1 & 0 & 0 \end{array} \right) + 
\sqrt 3 a^2 \left(\begin{array}{ccc}0&0& 1\\
0 & 0 & 0\\
1 & 0 & 0 \end{array} \right)   \diag(-1,-1,2)
\\
&=\sqrt 3 a^2 \left(\begin{array}{ccc}0&0& -1\\
0 & 0 & 0\\
2 & 0 & 0 \end{array} \right) 
 + \sqrt 3 a^2 \left(\begin{array}{ccc}0&0& 2\\
0 & 0 & 0\\
-1 & 0 & 0 \end{array} \right)  =  a K(0).
\end{align*}
}
\subsection{Bilinear maps} 
From~\eqref{e.R3Action} and equivariance it follows that the element $R_\z(\pi)\in\Sigma_\z$ acts on each isotypic component $V_i^Z$ by
  \[
R_\z(\pi)v_i = (-1)^iv_i
  \]
  for $v_i\in V_i^Z,\, i=0,1,2$, and so from~\eqref{e:bilin} we see that any $\Sigma_\z$-equivariant bilinear map $B:V\times V\to V=V_0^Z\oplus V_1^Z\oplus V_2^Z$ satisfies
\begin{align*}
  \wtr_\z(\pi)B(v_i,v_j) &= B((-1)^iv_i,(-1)^jv_j)  \\
                               &=(-1)^{i+j}B(v_i,v_j).
\end{align*}
Thus $\wtr_\z(\pi)$ fixes $B(v_i,v_j)$ when $i+j$ is even and multiplies it by $-1$ when $i+j$ is odd.  As a consequence we have the following result, extremely useful for simplifying calculations.
\begin{prop}  \label{p:vij}
For $v_i\in V_i^Z, \,i=0,1,2$
\begin{align}
  B(v_i,v_j)&\in V_0^Z\oplus V_2^Z, \quad i+j \text{ even}, \label{e:evenb}\\
  &\in V_1^Z, \quad i+j \text{ odd}.  \label{e:oddb}
\end{align} \qed
\end{prop}
\begin{cor}\label{c.B1}
  If $\Q_i$ denotes the component of $\Q$ in $V_i^Z,i=0,1,2$, then for $H,K\in V$ the component $B_1$ of $B$ in $V_1^Z$ is given by
  \begin{equation}  \label{e:b1}
B_1(H,K)=B(H_1,K_0+K_2) + B(H_0+H_2,K_1).
    \end{equation} \qed
\end{cor}
\begin{cor}  \label{c:hkprop}
  The result~\eqref{e:b1} applies to $B=\D^2G(\Q)$ for any $\SO(3)$-equivariant $G:V\to V$ when $\Q$ is fixed by~$\Sigma_\z$.  In particular it applies in the case of the quadratic map  $G^{\,0}:Q\mapsto\Q^2-\frac13\tr(\Q^2)I$ of Example~\ref{x:ex2} where we have  $B(H,K)=\D^2G^{\,0}(\Q)(H,K)=[H,K]^+$ independent of~$\Q$. \qed
  \end{cor}
\subsection{Specific form of $G$}
It is a standard result from group representation theory that a basis for the module of smooth $\SO(3)$-equivariant vector fields over the ring of smooth $\SO(3)$-invariant functions on $V$ is given by the pair of vector fields
\[
\{\,\Q,\,[\Q,\Q\,]^+\}
\]
(see~\cite[XV,\,Section 6\,]{GSS} for example); in other words any smooth $\SO(3)$-equivariant map (or vector field) $G:V\to V$ may be written in the form
\begin{equation}  \label{e:gexpr}
G(\Q) = g (\Q)\,\Q + \bar g (\Q)\,[\Q,\Q\,]^+
\end{equation}
where $g ,\bar g :V\to\bR$ are smooth $\SO(3)$-invariant functions.  Thus $G$ is completely determined once the two functions $g $ and $\bar g $ are chosen.
\msk

The condition for $\Q=Q^*$ to be a zero of $G$ is
\begin{equation*}
    0 = G(\Q^*) = g (\Q^*)\Q^* + \bar g (\Q^*)[\Q^*,\Q^*]^+   
               = \bigl(g (\Q^*) + 2a\bar g (\Q^*)\bigr)\Q^*
\end{equation*}
using Proposition \ref{p:Leigs}, that is
\begin{equation}  \label{e:equilg}
  \hat g(\Q^*) = 0
\end{equation}
where $\hat g:=g  + 2a\bar g $.
\subsubsection{First derivative of~$G$}
Differentiating~\eqref{e:gexpr} we have for any $\Q,H\in V$
\begin{equation}  \label{e:diffgexpr}
\D G(\Q)H = \D g (\Q)H\,Q + g (\Q)H + \D \bar g (\Q)H\,[\Q,\Q\,]^+ + 2\bar g (\Q)[\Q,H]^+.
\end{equation}
Therefore
\begin{equation}\label{e.lambda}
\D G(\Q^*)\Q^* = \lambda\Q^*
\end{equation}
where
\begin{equation}
  \lambda = g (\Q^*) + 4a\bar g (\Q^*) + \bigl(\D g (\Q^*)+2a\,\D \bar g (\Q^*)\bigr)\Q^*.
\end{equation}
With $G(\Q^*)=0$ this gives  
\begin{equation}
  \lam = 2a\bar g ^* + \Delta g ^* + 2a\Delta \bar g ^* 
  = 2a\bar g ^* + \Delta \hat g^*  \label{e:newlam}
  \end{equation}
using Proposition~\ref{p:Leigs} and~\eqref{e:equilg}, where
$g ^*$ denotes $g (\Q^*)$ and $\Delta g ^*:=\D g (\Q^*)\Q^*$ etc..
\msk

Likewise from~\eqref{e:diffgexpr} we find
\[
\D G(\Q^*)E_2(\al) = \mu E_2(\al) 
\]
where
\begin{equation}  \label{e:mudef}
  \mu =g ^*-4a\bar g ^* = 3g ^* = -6a\bar g ^*
  \end{equation}
taking account of the fact that $\D g (\Q^*) E_2(\al) =\D \bar g (\Q^*) E_2(\al) =0$ by Proposition~\ref{p:fixed}.  Also
\[
\D G(\Q^*)E_1(\al) = g ^*E_1(\al) + 2\bar g ^*[\Q^*,E_1(\al)]^+ = \hat g^*\,E_1(\al) = 0
\]
using~\eqref{e:equilg} and Proposition ~\ref{p:Leigs}, the result expected since $\kT^*=\Span\{E_1(\al)\}_{\al\in[0,\pi)}$.
In summary:
\begin{prop}
The eigenvalues of $\D G(\Q^*)$ corresponding to the eigenspaces $V_0^*,V_1^*,V_2^*$ are $\lam,0,\mu$ respectively, with $\lam,\mu$ given by~\eqref{e:newlam} and~\eqref{e:mudef}.  \qed
  \end{prop}
\subsubsection{Second derivative of~$G$}  \label{a:B2}
Differentiating~\eqref{e:diffgexpr} again we have for $H,K\in V$
\begin{align}
  \D^2G(\Q)(H,K) &= \D^2g (\Q)(H,K)\,Q + \bigl(\D g (\Q)H\bigr)K + \bigl(\D g (\Q)K\bigr)H \notag\\
  &\quad  + 2\bigl(\D \bar g (\Q)H\bigr)[\Q,K]^+ + 2\bigl(\D \bar g (\Q)K\bigr)[\Q,H]^+ \notag\\
  &\qquad + \D^2\bar g (\Q)(H,K)\,[\Q,Q\,]^+ + 2\bar g (\Q)\,[H,K]^+.  \label{e:d2ghk}
\end{align}
In the main text we need to evaluate the component of this expression tangent to the $\SO(3)$-orbit~$\kO$ of the uniaxial matrix~$\Q^*$ at points $Z\in\kO$.  Here we calculate this for $Z=\Q^*$ making significant use of Proposition~\ref{p:vij} and Corollary~\ref{c:hkprop}, and will be able to transfer the result to a general $\Q=Z\in\kO$ by applying the $\SO(3)$~action. 
\msk

Let $G_1$ denote the component of $G$ in $V_1^*$, and write $B_1=D^2G_1(\Q^*)$.  
\begin{prop} \label{p:d2ghk}
  {}\hfill
  \begin{enumerate}
\item  If $H,K\in V_0^*\oplus V_2^*$ or $H,K\in V_1^*$ then
  \begin{equation}  \label{e:nneqn}
B_1(H,K)= 0.
  \end{equation}
\item If  $H=H_0+H_2\in V_0^*\oplus V_2^*$ and $K=K_1\in V_1^*$ then
\begin{align}
  B_1(H_0+H_2,K_1)&=\bigl(\D g (\Q^*)H_0\bigr)K_1 + 2\bigl(\D \bar g (\Q^*)H_0\bigr)[\Q^*,K_1]^+ + 2\bar g ^*[H_0+H_2,K_1]^+ 
  \notag \\ &=\bigl(\D \hat g(\Q^*)H_0\bigr)K_1+ 2\bar g ^*[H_0,K_1]^+ + 2\bar g ^*[H_2,K_1]^+.  
  \label{e:d2g1hk} 
\end{align}
\end{enumerate}
  \end{prop}
\proof The result (1) is immediate from Corollary \ref{c.B1}.
Part (2) follows from \eqref{e:d2ghk}, using the fact that $\Q^*$ and $[\Q^*,\Q^*]^+$ lie in $V_0^*$, together with Proposition~\ref{p:fixed} applied to the $\SO(3)$-invariant functions $g$ and~$\bar g$. For the term involving $[\Q^*,K_1]^+$ we use the eigenvalue result from Proposition~\ref{p:Leigs}. 
 \qed
\subsection{Explicit expression for $ [H,K]^+_1$}  
Finally, an explicit expression for the $V_1^*$-component $[H,K]^+_1$ of $[H,K]^+$ is needed in order to evaluate the bifurcation function~\eqref{e:f2zterms}.  Using the identity 
\begin{equation}
%  [\tq(\al),\tq(\al')]^+ &= -2a\cos2(\al-\al')\,Q^* \label{e:tqtq+} \\
  [E_2(\al), E_1(\al')]^+ = \frac{1}{\sqrt2} E_1(2\al-\al') \label{e:tqk+}
%  [K(0),K(\al)]^+ &= a\cos\al\,Q^* + \sqrt3a\tq(\al/2) \label{e:kk+}
\end{equation}
we see 
\begin{align} 
  [E_{21},E_{11}]^+ &= [E_{22},E_{12}]^+ = \frac1{\sqrt2} E_{11} \label{e:Eeq1} \\
  [E_{22},E_{11}]^+ = \frac1{\sqrt2}& E_{12} ,\quad [E_{21},E_{12}]^+ = -\frac1{\sqrt2} E_{12}  \label{e:Eeq2}
\end{align}
since $E_1(-\pi/2)=-E_1(\pi/2)$ from~\eqref{e:E_j_def}.
Then  writing
\begin{align}
  H&=(h_{01},h_{11},h_{12},h_{21},h_{22})  \label{e:hcoords} \\
  K&=(k_{01},k_{11},k_{12},k_{21},k_{22})  \label{e:kcoords}
\end{align}
with respect to the basis $\B^*$ for~$V$ as given by~\eqref{e:b*def} 
we find
\begin{align} 
  [H_2&,K_1]^+  
    = \bigl[\,h_{21} E_{21} + h_{22} E_{22}\,,\, k_{11}  E_{11} + k_{12} E_{12}\,\bigr]^+  \notag \\
    & =  \frac{1}{\sqrt2} (h_{21} k_{11}  +  h_{22} k_{12})E_{11}
       +  \frac{1}{\sqrt2} (h_{22} k_{11}  -  h_{21} k_{12})E_{12}
    \label{e:part2}
\end{align}
using~\eqref{e:Eeq1} and~\eqref{e:Eeq2}.
We therefore arrive at the following:
\begin{prop}\label{p.[H,K]^+_1}
 For $H, K$ as in~\eqref{e:hcoords},\eqref{e:kcoords}
\begin{align*}
 [H,K]^+_1 &=
 \left(\frac{1}{\sqrt 6} ( h_{01} k_{11} + h_{11} k_{01}) +   \frac1{\sqrt 2}( h_{11} k_{21} +h_{21} k_{11}  +     h_{22}k_{12} +h_{12} k_{22})  \right)   E_{11}\\
 &+
  \left(\frac{1}{\sqrt 6} ( h_{01} k_{12} + h_{12} k_{01}) + 
\frac1{\sqrt2} (h_{11}k_{22}+ h_{22}k_{11}   -  h_{12} k_{21} - h_{21}k_{12} )\right) E_{12}.
\end{align*}
\end{prop}
\proof
Let $H = H_0 + H_1 + H_2$, $K = K_0 + K_1 + K_2$ with $H_i, K_i \in V_i^*$,
 $i=0,1,2$.
Using Corollary~\ref{c:hkprop} we see that 
  \[
  [H ,K]^+_1 =  [H_0+H_2,K_1]^+ + [H_1,K_0+K_2]^+ .
  \]
 Since $H_0=h_{01}E_0^Z = \tfrac1{\sqrt6a} h_{01} Z$ it follows from Proposition~\ref{p:Leigs} that $[H_0,K_1]^+=\tfrac1{\sqrt 6} h_{01}K_1$. We then use
  \eqref{e:part2} to obtain
  \begin{align*}
    [H_0+H_2,K_1]^+ &= [H_0,K_1]^+ + [H_2,K_1]^+  \\
  &= \frac{1}{\sqrt 6}h_{01}(k_{11}E_{11} + k_{12}E_{12})
  +  \frac1{\sqrt 2}( h_{21} k_{11} +h_{22} k_{12} ) E_{11} +
\frac{1}{\sqrt2}  (h_{22}k_{11}   -  h_{21} k_{12}  ) E_{12}.
  \end{align*}
Exchanging the roles of $H$ and $K$ gives the result.
\qed
\begin{cor}  
By $\SO(3)$-equivariance the same formula applies to give the $V_1^Z$-component of $[H,K]^+$, the coordinates~\eqref{e:hcoords},\eqref{e:kcoords} in this case being taken with respect to the basis~$\B^Z$.  \qed
\end{cor}
We can now be even more specific: the expression~\eqref{e:d2g1hk} simplifies to
\begin{align}  
 B_1(H_0+H_2,K_1)&= 
 h_{01}\bigl(\D \hat g(\Q^*)E_0\bigr)K_1+ \frac{\sqrt2}{\sqrt 3}h_{01}\bar g ^*K_1 + 2\bar g ^*[H_2,K_1]^+  \notag \\
 &
=\frac{\lam}{\sqrt6a}h_{01}K_1  - \frac\mu{3a}[H_2,K_1]^+ 
  \label{e:d2g1simp}
\end{align}
\details{
\begin{align*}
 h_{01}\bigl(\D \hat g(\Q^*)E_0\bigr)K_1+\frac{\sqrt2}{\sqrt 3}h_{01}\bar g ^*K_1
 = \frac{1}{\sqrt6a}h_{01}\Delta k^* K_1+\frac{\sqrt2}{\sqrt 3}h_{01}h^*K_1
=   \frac{1}{\sqrt6a}(h_{01}\Delta k^*+2 ah_{01}h^*)  K_1
 =\frac{\lam}{\sqrt6a}(h_{01}  h_{01}K_1
\end{align*}
}
using~\eqref{e:newlam} and~\eqref{e:mudef}.  Thus we conclude from Proposition~\ref{p:d2ghk}, \eqref{e:d2g1simp}  
and \eqref{e:part2}:
\begin{prop} \label{p:hkform}
  For $H = H_T + H_N$ and $ K = K_T + K_N\in V = V_1^*\oplus\bigl(V_0^*\oplus V_2^*\bigr)$ and $B_1=D^2G_1(\Q^*)$ we have
  \begin{align}  \label{e:d2g1final}
    B_1(H_N,K_T)   
     =\kappa_1 E_{11} + \kappa_{2} E_{12}
   \end{align}
where with notation as in~\eqref{e:hcoords},\eqref{e:kcoords}
   \begin{align}
     \kappa_1 &= \frac{\lam}{\sqrt6a} h_{01}k_{11} - \frac\mu{3\sqrt2 a} (h_{21} k_{11}  +  h_{22} k_{12})       \label{e:kappa0}  \\
     \kappa_{2} &=\frac{\lam}{\sqrt6a}  h_{01}k_{12} - \frac\mu{3\sqrt2 a} (h_{22} k_{11}  -  h_{21} k_{12}). \label{e:kappa2}
   \end{align}
   \qed
\end{prop}
\begin{cor} \label{c:hkform2}
  By $\SO(3)$-equivariance the same expressions~\eqref{e:kappa0},~\eqref{e:kappa2}
  apply relative to the decomposition~$V = V_1^Z\oplus\bigl(V_0^Z\oplus V_2^Z\bigr)$.  \qed
\end{cor}
It is only $\kappa_1$ that we need in the calculation of the bifurcation function.
\section{General form for $\LL(\Q)D$}  \label{s:genform}
The term $\LL(\Q)D$ in~\eqref{e:sys1} representing the effect on the dynamics of~$\Q$ from the symmetric part $D$ of the flow velocity gradient is $\SO(3)$-equivariant in $(\Q,D)$ and linear in~$D$. From the expression in~\cite[\S40]{RivErick} giving the general form of an $\SO(3)$-equivariant (isotropic) polynomial matrix-valued function of two matrices (here $3\times3$) we find that in our context in~$V$ we have
\begin{equation}  \label{e:5gens}
\LL(\Q)D =  w_1D + w_2[\Q,D\,]^+ + w_3[\Q^2,D\,]^+ +  w_4\Q  + w_5[\Q,\Q\,]^+ 
\end{equation}
where the coefficients $w_i=w_i(\Q,D), i=1,\ldots,5$ are $\SO(3)$-invariant polynomials in~$(\Q,D)$ such that $w_1,w_2,w_3$ are functions of~$\Q$ only  while $w_4,w_5$ are linear in~$D$.  The only candidates for $w_4$ or $w_5$ are $\tr(\Q D)$ and $\tr(\Q^2 D)$ multiplied by invariant functions of $\Q$ alone, and thus we find as in~\cite{MacM92a}
\begin{align}
  \LL(\Q)D =  v_1D + v_2[\Q,D\,]^+ + v_3[\Q^2,&\,D\,]^+ + v_4\,\tr(\Q D)\Q + v_5\,\tr(\Q^2 D)\Q \notag \\ &+ v_6\,\tr(\Q D)[\Q,\Q\,]^+ +v_7\,\tr(\Q^2 D)[\Q,\Q\,]^+ \label{e:7gens}
\end{align}
where $v_1,\ldots,v_7$ are polynomial functions of $\tr\Q^2$ and $\tr\Q^3$
with $v_i=w_i$ for $i=1,2,3$ and 
\begin{align}
  w_4 &= v_4 \,\tr(\Q D) + v_5 \,\tr(\Q^2D) \label{e:w4eq}  \\
  w_5 &= v_6 \,\tr(\Q D) + v_7 \,\tr(\Q^2D). \label{e:w5eq}
\end{align}
\details{
  Any function $h:V\times V\to\bR$ linear in the second factor has the form $(\Q,D)\mapsto\left<u(\Q),D\right>=\tr(u(\Q)D)$ for some vector field $u:V\to V$.
If $h$ is invariant then $\tr(u(\wtr\Q),\wtr D)=\tr(u(\Q),D)$.  In any case $\tr(u(\Q),D)=\tr(\wtr u(\Q),\wtr D)$, so $u$ is equivariant and has the form~\eqref{e:gexpr}.  Note that $\tr(\Q^2D)=\tr([\Q,\Q]^+D)$.
  }
That~\eqref{e:7gens} holds also in the smooth case follows from the results in~\cite{SCH}.
\msk

Replacing $D$ by $\widetilde D:=\wtr_3(-\omega t)D$ in~\eqref{e:7gens} and using the eigenspace properties of $[Z,\cdot\,]^+$ from Proposition~\ref{p:Leigs}
to see that
\begin{equation}  \label{e:ZZ}
[Z,Z\,]^+=2aZ\quad\text{so that}\quad Z^2=aZ + \tfrac13\tr(Z^2)I = aZ+2a^2I
\end{equation}
we find
\begin{equation} \label{e:lztd}
\widetilde \LL(Z)D = L(Z)\widetilde D= v_1^*\widetilde D + v_2^*[Z,\widetilde D\,]^+ + v_4^*\,\tr(Z\widetilde D)Z
\end{equation}
where 
\[
v_1^*=v_1+4a^2v_3\,,\quad   v_2^*=v_2 + av_3\,,\quad v_4^*=(v_4+av_5)+2a(v_6+av_7)
\]
evaluated at $\Q=Z$.  Since the functions  $v_1,\ldots,v_7$ are $\SO(3)$-invariant their values at $Z$ are the same as their values at $\Q^*$ and depend only on $a$.
\details{
We have from (B.2) and (B.3) that
\begin{align*}
L(Z) \widetilde D &= v_1 \widetilde D + v_2[Z, \widetilde D\,]^+ + v_3 (a [Z,\widetilde D\,]^+
+  4 v_3 a^2 \widetilde D
+ v_4 \tr(Z \widetilde D) Z + v_5 a \tr(Z \widetilde D)\tilde Z  + v_6  \tr (Z \widetilde D) 2aZ + v_7 a \tr(Z \widetilde D)2a Z \\
&=( v_1 + v_34 a^2) \widetilde D  + (v_2  + v_3 a) [Z,\widetilde D\,]^+
+ (v_4   + v_5 a  + 2v_6 a +2v_7 a^2)  \tr (Z \widetilde D) Z \\
&= v_1^*  \widetilde D  + v_2^* [Z,\widetilde D\,]^+  + v_4^* \tr (Z \widetilde D) Z
\end{align*}
where
\begin{align*}
  v_1^*&= v_1 + 4 a^2 v_3, \\
v_2^* &= v_2  + av_3 ,  \\
v_4^* &= v_4 + av_5 + 2av_6 + 2a^2 v_7 
\end{align*}
and we use 
  \begin{align*}
    [Z^2,\widetilde D\,]^+&=[aZ+2a^2I,\widetilde D\,]^+=a[Z,\widetilde D\,]^++4a^2\widetilde D \\
    \tr(Z^2\widetilde D)&=a\tr(ZD) \\
    [Z,Z\,]^+&=2aZ     
    \end{align*}
}
\msk

Observing from~\eqref{e:ytq} that $y(t,\Q)$ is a linear function of~$D$, as also is $\chi(t,Z)$ from~\eqref{e:chiexpr}, we see that the expressions for $y(t,Z)$ and $\chi(t,Z)$ arising from~\eqref{e:7gens} and~\eqref{e:lztd} are therefore given by
\begin{align}
  y(t,Z) &= v_1^*y^c(t,Z) + v_2^*\,y^l(t,Z) + v_4^*\,y^q(t,Z)  \\
  \chi(t,Z) &= v_1^*\chi^c(t,Z) + v_2^*\,\chi^l(t,Z) + v_4^*\,\chi^q(t,Z) \label{e:newchi}
  \end{align}
with the notation of Subsection~\ref{s:ycalc}.
Consequently the $B_{11}^Z$ term in the second order term~\eqref{e:f2zterms}    of the bifurcation function is exactly as evaluated in Subsection~\ref{s:biffn} but with the coefficients $ m_c, m_l, m_q$  replaced by the coefficients~$v_1^*,v_2^*,v_4^*$ respectively.
\msk

Next, to obtain the $\D\widetilde \LL$ term of the second order term of the bifurcation function~\eqref{e:f2zterms} we differentiate~\eqref{e:5gens} with respect to~$\Q$ at $\Q=Z\in\kO$. For $H\in V$ this gives
\begin{align}
  \big(\D L(Z)H\big)\widetilde D = \bar w_1\widetilde D &+ \bar w_2[Z,\widetilde D\,]^+ + \bar w_3[Z^2,\widetilde D\,]^+ +  \bar w_4Z  + \bar w_5[Z,Z\,]^+ \notag \\
  &+ w_2[H,\widetilde D\,]^+ + w_3[[Z,H]^+,\widetilde D\,]^+ + w_4 H + 2w_5[Z,H]^+ \label{e:5gdiff}
\end{align}
where $\bar w_i$ denotes the $\Q$-derivative of $w_i$ at $\Q=Z$ applied to~$H$ for $i=1,\ldots,5$.  With $p$ denoting $p_{11}^Z$ and writing $pH=H_{11}^Z$ etc.  we see that the expression obtained by applying~$p$ to~\eqref{e:5gdiff} simplifies to 
\begin{align}
  p\big(\D L(Z)H\big)\widetilde D = \bar w_1\widetilde D_{11}^Z &+ a\bar w_2\widetilde D_{11}^Z + 5a^2\bar w_3\widetilde D_{11}^Z  \notag \\
  &+ w_2\,p[H,\widetilde D\,]^+ + w_3\,p[\widehat H,\widetilde D\,]^+  + (w_4+2aw_5)H_{11}^Z\label{e:5gdsimp}
\end{align}
where
\[
  [Z,H]^+=\widehat H:= 2aH_0^Z + aH_1^Z -2aH_2^Z
  \]
  from the eigenspace decomposition of Proposition~\ref{p:Leigs}.
  Here we again use~\eqref{e:ZZ} as well as $[Z,\widetilde D\,]^+_{11}=a\widetilde D_{11}^Z$, and the coefficients $w_i$ are evaluated at~$\Q=Z$ so that in particular
from~\eqref{e:w4eq} and~\eqref{e:w5eq} with~\eqref{e:ZZ}
  \begin{align}
  w_4 &= (v_4  + av_5) \,\tr(Z \widetilde D) \\
  w_5 &= (v_6  + av_7) \,\tr(Z \widetilde D).
\end{align}
\details{
From~\eqref{e:5gdiff} we get
\begin{align*}
  p (D L(Z) H) \widetilde D &= \bar w_1 \widetilde D^Z_{11}  + a \bar w_2 \widetilde D^Z_{11}
+ \bar  w_3 a^2 \widetilde D^Z_{11} + 4 a^2 \bar w_3 \widetilde D^Z_{11}
 + w_2  p [H,\widetilde D\,]^+ + p w_3[ [Z,H]^+, \widetilde D\,]^+ + w_4 H^Z_{11}
 + 2 w_5 a H^Z_{11} \\
 &= (\bar w_1   + a \bar w_2+ \bar  w_3 a^2  + 4 a^2 \bar w_3  ) \widetilde D^Z_{11} +  w_2 p [H,\widetilde D\,]^+
  + w_3p[ [Z,H]^+, \widetilde D\,]^+
  + (w_4 + 2 a w_5 ) H^Z_{11}.
  \end{align*}
 
 To prove: We have
 \[
 (w_4 + 2 w_5 a ) H^Z_{11} =  v_4^* tr(Z\widetilde D)  H^Z_{11}.
 \]
 By definition
 \[
 v_4^*= v_4 + v_5 a+ 2 v_6 a + 2 v_7 a^2
 \]
 and
 \begin{align*}
 w_4 + 2 w_5 a  &= v_4 tr(Z \widetilde D) + v_5 tr(Z^2,\widetilde D) + 2a(v_6 tr(Z\widetilde D) + v_7 tr(Z\widetilde D)) \\
 &= v_4  tr(Z D) + v_5 a tr(Z\tilde,D) + 2a (v_6 tr(ZD) + v_7 a tr(Z\widetilde D)) \\
 &= (v_4 + a v_5 + 2a v_6  + 2a^2 v_7) tr(Z\widetilde D).
 \end{align*}
}
The contribution that~\eqref{e:5gdsimp} makes to the second order term $F_2(Z)$ of the bifurcation function~\eqref{e:f2zterms} is obtained by substituting $\chi(t,Z)$ for~$H$ and integrating from $t=0$ to $t=T_0$. Since
\[
\int_0^{T_0}\widetilde D_{11}^Z\d t = p\int_0^{T_0}\wtr_\z\wtr_3(-\omega t)D\,\d t = 0
  \]
and also from~\eqref{e:intq1}-\eqref{e:intq3} 
\begin{equation*}
\int_0^{T_0}\tr(Z\widetilde D(t))\chi_{11}^Z \d t = 0
\end{equation*}
  we obtain
\begin{equation}  \label{e:5gdsimpler}
  p\int_0^{T_0}\big(\D L(Z)\chi(t,Z)\big)\widetilde D \d t =
  w_2 \int_0^{T_0}p[\chi(t,Z),\widetilde D\,]^+ \d t
  +   w_3 \int_0^{T_0}p[\hat\chi(t,Z),\widetilde D\,]^+ \d t 
\end{equation}
with
\begin{equation}  \label{e:hatchi}
  \hat\chi=2a\chi_0+a\chi_1-2a\chi_2.
\end{equation}
Hence, just as in Subsection~\ref{s:biffn}, it is only $\LL^l(\Q)$ (see~\eqref{e:lincombE}) that contributes to the $\D\widetilde \LL$ term in~\eqref{e:f2zterms}.
\msk

If $w_3=0$ we therefore see that the second order term $f_2(\theta)$ of the bifurcation function $f(\theta)$ in the general case~\eqref{e:7gens} is obtained from the expression~\eqref{e:f2cz+} but now with the coefficients $ m_c, m_l, m_q$ that define $\Lam_0,\Lam_2$ in~\eqref{e:lamplus} simply replaced by the coefficients~$v_1^*,v_2^*,v_4^*$ respectively. Observe that~\eqref{e:7gens} corresponds to~\eqref{e:lincombE}  with $v_1,v_2,v_4= m_c, m_l, m_q$ and the remaining coefficients~$v_j=0$.
\msk

When $w_3\ne0$ there is the further term arising from $\int_0^{T_0}p[\hat\chi(t,Z),\widetilde D\,]^+ \d t$.
Writing~\eqref{e.DETerm} as
\[
\int_0^{T_0}p[\chi^c(t,Z),\widetilde D\,]^+ \d t = -as_0(\lam,\theta)+3as_2(\mu,\theta)
\]
we see from~\eqref{e:hatchi} that
\begin{align}
  \int_0^{T_0}p[\hat\chi^c(t,Z),\widetilde D\,]^+ \d t &= -as_0(\lam,\theta)\times(2a)+3as_2(\mu,\theta)\times(-2a)  \notag \\
  &=-2a^2s_0(\lam,\theta)-6a^2s_2(\mu,\theta)  \label{e:newcterm}
\end{align}
so that also from~\eqref{e:tilchi}
\begin{align}
  \int_0^{T_0}p[\hat\chi^l(t,Z),\widetilde D\,]^+ \d t &= -2a^2s_0(\lam,\theta)\times(2a)-6a^2s_2(\mu,\theta)\times(-2a)  \notag \\
  &=-4a^3s_0(\lam,\theta)+12a^3s_2(\mu,\theta)  \label{e:newlterm}
\end{align}
and from~\eqref{e:detilchihat}
\begin{equation}
  \int_0^{T_0}p[\hat\chi^q(t,Z),\widetilde D\,]^+ \d t = -6a^3s_0(\lam,\theta)\times(2a)
  =-12a^4s_0(\lam,\theta).  \label{e:newqterm}
\end{equation}
Consequently in the second order term of the bifurcation function the coefficients $\Lam_0,\Lam_2$ in~\eqref{e:f2cz+} are replaced by their counterparts with  the coefficients~$v_1^*,v_2^*,v_4^*$ in place of $ m_c, m_l, m_q$, together with the coefficients arising from~\eqref{e:newcterm},\eqref{e:newlterm},\eqref{e:newqterm}, giving
\begin{align}
  \Lam_0 &= {v_1^*}^2 + 2av_1^*v_2^* + 6a^2v_1^*v_4^*
  - w_3\big(2a^2v_1^*+4a^3v_2^*+12a^4v_4^*\big) \label{e:newLam0} \\
  \Lam_2 &=  {v_1^*}^2 + 2av_1^*v_2^* - 8a^2{v_2^*}^2
  - w_3\big(6a^2v_1^* - 12a^3v_2^*\big) \label{e:newLam2}
 \end{align}
where we recall that $w_3=v_3$.
\bibliographystyle{plain}

  \end{document}